\documentclass[10pt,reqno]{amsart}
\usepackage{amsmath,amsfonts,amsthm,amssymb,amsxtra}
\usepackage{amsxtra, amssymb, mathrsfs}
\usepackage[usenames,dvipsnames]{color}

\newtheorem{theorem}{Theorem}[section]
\newtheorem{proposition}[theorem]{Proposition}

\newtheorem{lemma}[theorem]{Lemma}

\theoremstyle{definition}

\newtheorem{assumption}[theorem]{Assumption}

\theoremstyle{remark}
\newtheorem{remark}[theorem]{\bf Remark}
\numberwithin{equation}{section}

\newcommand{\ai}{{\rm Ai}}
\newcommand{\bi}{{\rm Bi}}
\newcommand{\B}{\mathscr{B}}

\newcommand{\C}{\mathbb{C}}

\newcommand{\D}{\mathcal{D}}
\newcommand{\G}{\mathcal{G}}

\newcommand{\h}{\mathfrak{H}}

\newcommand{\el}{\mathcal{L}}
\newcommand{\U}{\mathcal{U}}

\newcommand{\rt}{{\rm curl}}
\newcommand{\pd}{\partial}
\newcommand{\eps}{\varepsilon}

\newcommand{\F}{\mathcal{F}}

\newcommand{\N}{\mathbb{N}}

\newcommand{\R}{\mathbb{R}}
\newcommand{\rr}{\mathfrak{R}}
\newcommand{\af}{\mathfrak{a}}
\newcommand{\bb}{\mathfrak{b}}

\newcommand{\Z}{\mathbb{Z}}

\begin{document}

\title[Two-dimensional magnetic Schr\"odinger operators]
{Resolvent expansion and time decay of the wave functions for two-dimensional magnetic Schr\"odinger operators}

\author {Hynek Kova\v{r}\'{\i}k}

\address {Hynek Kova\v{r}\'{\i}k, DICATAM, Sezione di Matematica, Universit\`a degli studi di Brescia, Italy}

\email {hynek.kovarik@unibs.it}

%\date{\today}

\begin {abstract}
We consider two-dimensional Schr\"odinger operators $H(B,V)$ given by equation \eqref{ham} below. We prove that, under certain regularity and decay assumptions on $B$ and $V$, the character of  the expansion for the resolvent $(H(B,V)-\lambda)^{-1}$ as $\lambda\to 0$ is determined by the flux of the magnetic field $B$ through $\R^2$. Subsequently,  we derive the leading term of the asymptotic expansion of the unitary group $e^{-i t H(B,V)}$ as $t\to \infty$ and show how the magnetic field improves its decay in $t$ with respect to the decay of the unitary group $e^{-i t H(0,V)}$. 
\end{abstract}

\maketitle

%%%%%%%%%%%%%%%%%%%%%%%%%%%%%%%%%%%%%%%%%%%%%%%%%%%%%%%%%%%%%%%%%%%%%%%%%%%%

\section{\bf Introduction}

\noindent The present paper is concerned with the  Schr\"odinger operator associated to a magnetic field $B$:
\begin{equation} \label{ham}
H(B,V) = (i \nabla +A)^2  + V \quad \text{in} \quad L^2(\R^2),
\end{equation}
where $\rt \, A  =B$ and $V$ is an electric scalar potential. Both $B$ and $V$ are assumed to have a polynomial decay at infinity.  We will analyze the asymptotic expansion of the resolvent $(H(B,V) -\lambda)^{-1}$ as $\lambda\to 0$ and the long time behavior of the unitary group $e^{-i t H(B,V)}$ generated by $H(B,V)$. It is well-known that the two problems are closely related to each other, see e.g. \cite{JK,mu,sch2}.  
Asymptotic expansions of the resolvent have been well studied in the absence of magnetic filed, i.e.~for the operator $H(0,V)$. In the case of dimension two, in particular, it has been shown that if zero is a regular point of $H(0,V)$, which means that zero is neither an eigenvalue nor a resonance of $H(0,V)$, and if $V$ decays fast enough at infinity, then as $\lambda\to 0$ 
\begin{equation} \label{2D-res}
 (H(0,V) -\lambda)^{-1}  = T_0 + T_1\, (\log \lambda)^{-1} + o((\log \lambda)^{-1}) \qquad  \left(\, \text{in\  \ } \R^2 \right)
\end{equation}
holds in suitable weighted $L^2-$spaces; see \cite{jn, mu, sch}.  In dimension three, still under the condition that  zero is a regular point of $H(0,V)$, the term $(\log \lambda)^{-1}$ in the above equation must be replaced by $\lambda^{1/2}$,  cf. \cite{JK, jn, mu}. 

As far as the long time behavior of the operator $e^{-i t H(0,V)}$ is concerned, the classical results say that in dimension three, for sufficiently short range potentials and under the condition that zero is a regular point, one has for $t\to\infty$ the following asymptotic equation in $L^2(\R^3)$:
\begin{equation} \label{r3-asymp}
(1+|x|)^{-s}\, e^{-i t H(0,V)}\, P_{ac} \, (1+|x|)^{-s}  \, = \, S_3\  t^{-\frac 32} + o(t^{-\frac 32}), \qquad \left(\, \text{in\  \ } \R^3 \right)
\end{equation}
which holds in $L^2(\R^3)$ for $s$ large enough , see \cite{JK, je,mu}, and with $P_{ac}$ being the projection on the absolutely continuous subspace of $H(0,V)$. Note that is the decay rate  $t^{-3/2}$ corresponds to the free evolution in $\R^3$. For higher-dimensional results we refer to \cite{je2, mu}. 

The situation in dimension two is different since in this case by adding a potential $V$ one may improve the decay rate of $e^{-i t H(0,V)}$ with respect to the $t^{-1}$ decay rate of the free evolution operator,  provided the weight function  $(1+|x|)^s$ grows fast enough. More precisely, it was proved by Murata, see \cite{mu}, that if zero is a regular point then in $L^2(\R^2)$ we have for $t\to\infty$ the asymptotic expansion
\begin{equation} \label{murata}
(1+|x|)^{-s}\, e^{-i t H(0,V)}\, P_{ac} \, (1+|x|)^{-s}\, =\, S_2\ t^{-1}\, (\log t)^{-2} +  o(\, t^{-1}\, (\log t)^{-2})  \qquad \left(\, \text{in\  \ } \R^2 \right)
\end{equation}
with $s>3$ and $|V(x)| \lesssim (1+|x|)^{-6-0}$, see section \ref{sec-notation} for the precise meaning of the latter condition.

\smallskip

The aim of this paper is to show that the presence of a magnetic field in $\R^2$ changes completely the character of the expansions \eqref{2D-res} and \eqref{murata}.  In order to describe how the magnetic field affects these asymptotic expansions, we define the normalized flux of $B$ through $\R^2$ by 
\begin{equation} \label{flux}
\alpha := \frac{1}{2\pi} \int_{\R^2} B(x)\, dx.
\end{equation}
The main results of this paper show then if $B$ is sufficiently smooth and decays fast enough at infinity, then the behavior of 
$(H(B,V) -\lambda)^{-1}$ for $\lambda\to 0$ as well as the behavior of $e^{-i t H(B,V)}$ for $t\to \infty$ is determined by the distance between $\alpha$ and the set of integers: 
\begin{equation} \label{min}
\mu(\alpha):= \min_{k\in\Z} |k+\alpha|.
\end{equation}
More precisely, if $\alpha$ is finite and non-integer, and if zero is a regular point of $H(B,V)$, then for $\lambda\to 0$ the expansion
\begin{equation} \label{2D-res-mag}
 (H(B,V) -\lambda)^{-1}  = F_0 + F_1\, \lambda^{\mu(\alpha)} + o(\lambda^{\mu(\alpha)}) \qquad    \alpha\not\in\Z, \qquad \left(\, \text{in\  \ } \R^2 \right)
\end{equation}
holds true in certain weighted $L^2-$spaces with suitably chosen weights, see Theorem \ref{thm-1}. Accordingly we obtain a faster decay rate of $e^{-i t H(B,V)}$ in $t$ with respect to the decay rate of $e^{-i t H(0,V)}$. In fact we show in Theorem \ref{thm-3} that there exists a bounded operator $\mathcal{K}$ in  $L^2(\R^2)$ such that for $ t\to\infty$ 
\begin{equation} \label{disp-mag}
 (1+|x|)^{-s}\, e^{-i t H(B,V)}\, P_{ac} \, (1+|x|)^{-s}  \, = \,  t^{-1-\mu(\alpha)}\, \mathcal{K} +o(\, t^{-1-\mu(\alpha)})   \qquad   \alpha\not\in\Z, \qquad \left(\, \text{in\  \ } \R^2 \right)
\end{equation}
holds 
in $L^2(\R^2)$ provided $s> 5/2$, $ |V(x)| \lesssim (1+|x|)^{-3-0},$ and  $B$ satisfies suitable decay and regularity conditions, see assumption \ref{ass-B}.  On the other hand, For integer values of $\alpha$ one has qualitatively the same behavior as in \eqref{2D-res} and \eqref{murata}, see Theorems \ref{thm-2} and \ref{thm-4}. We also give an explicit formulae for the operators $F_0, F_1$ and $\mathcal{K}$ in the case $\mu(\alpha)< 1/2$.
Hence the character of the expansions for $(H(B,V) -\lambda)^{-1}$ as $\lambda\to 0$ and $e^{-i t H(B,V)}$ as $t\to \infty$ is completely determined by the flux of $B$. 

\smallskip
 
 To understand what makes the family of magnetic fields with equal fluxes distinguished, we refer to Lemma \ref{lem-gauge} and equation \eqref{arg-stokes}. The latter implies that a difference between two magnetic Hamiltonians $H(B_1,V)-H(B_2,V)$ is a first order differential operator with sufficiently short range coefficients, provided we choose a suitable gauge, {\it if and only if} $B_1$ and $B_2$ have equal flux through $\R^2$. Therefore, only in this case the coefficients of $H(B_1,V)-H(B_2,V)$ may decay fast enough, depending on the decay of $B_1$ and $B_2$, in order to compensate for the growth of the weight function $(1+|x|)^{s}$. When the fluxes of $B_1$ and $B_2$ are different, then the coefficients of the first order term in $H(B_1,V)-H(B_2,V)$ cannot decay faster than $|x|^{-1}$ even if both $B_1$ and $B_2$ have compact support, cf. equation \eqref{arg-stokes}.
 
 This indicates a natural strategy for proving \eqref{2D-res-mag}, and consequently 
\eqref{disp-mag}; we choose a concrete magnetic field $B_0$, see equation \eqref{B0}, for which it is possible to calculate the resolvent for small values of $\lambda$ explicitly. Using this fact we first show that the expansion \eqref{2D-res-mag} holds for $H(B_0,0)$, see Proposition \ref{prop-1}, and then we extend the result to all magnetic fields with the same flux (and sufficient decay) by using the perturbation theory in combination with Lemma \ref{lem-gauge}. The proof in the case of integer flux follows the same strategy, cf. Proposition \ref{prop-2}. We point out that although our perturbative approach cannot be applied when the flux of $B$ differs from that of $B_0$ by a non-zero integer, the respective resolvent expansions are qualitatively the same since they depend only on $\mu(\alpha)$. 

\smallskip

For comparison it should be mentioned that in the case of dimension three the effect of a magnetic field on the asymptotic expansion \eqref{r3-asymp} is much weaker. Indeed, from \cite[Thms. 8.11]{mu} it follows that, under sufficient regularity and decay assumptions on  $B$,  the unitary group $e^{-i t H(B,V)}$ satisfies again the asymptotic expansion \eqref{r3-asymp} (with different coefficients), see also \cite{kk}.

%\smallskip
\newpage

The paper is organized as follows. In section \ref{sec-prelim} we introduce some necessary notation and formulate our main results. The proofs are given in section \ref{sec-proofs}. The concrete model associated to the operator $H(B_0,0)$ is treated in sections \ref{sec-model} and \ref{sec-integer}. In section \ref{sec-aux-bessel} we give some auxiliary technical results needed for the proofs of Propositions \ref{prop-1} and \ref{prop-2}.

%%%%%%%%%%%%%%%%%%%%%%%%%%%%%%%%%%%%%%%%%%%%%%%%%%%%%%%

\section{\bf Main results}
\label{sec-prelim}

%%%%%%%%%%%%%%%%%%%%%%%%%%%%%%%%%%%%%%%%%%%%%%%%%%
\subsection{Notation} \label{sec-notation}
Before we formulate our main results we need to introduce some notation. 
Let $\rho: \R^2 \to \R$ be given by $\rho(x) =\rho(|x|)= 1+|x|$ and let $s\in\R$. We define
$$
L^2(\R^2, s) = \{ u : \|  \rho^s\,  u\|_{L^2(\R^2)} < \infty \}, \qquad \|u\|_s :=  \|  \rho^s \, u\|_{L^2(\R^2)}.
$$
We will denote by $\B(X,Y)$ the space of bounded linear operators from a Banach space $X$ into a Banach space $Y$ and by $\| \cdot \|_{\B(X,Y)}$ the corresponding operator norm. For the sake of brevity we will make use of the following shorthands: instead of 
$\B(L^2(\R^2, s), L^2(\R^2, s'))$ we write $\B(s,s' )$, and if $X=Y$, then we use the notation $ \B(X)=\B(X,X)$. 
Given $R>0$ and a point $x\in\R^2$ we denote by $\D (x,R)\subset\R^2$ the open disc with radius $R$ centred in $x$. 
The scalar product in a Hilbert space $\mathscr{H}$ will be denoted by $\langle \cdot \, , \cdot \rangle_{\mathscr{H}}$. 
For $x=(x_1, x_2)\in\R^2$ and  $y=(y_1, y_2) \in \R^2$  we will often use the polar coordinates representation
\begin{equation} \label{polar-repr}
x_1 + i x_2  = r  e^{i \theta}, \quad y_1 + i y_2 =r'  e^{i \theta'},  \qquad r,r' \geq 0, \quad \  \theta, \theta' \in [0, 2\pi).
\end{equation}
Given functions $f,g\in L^\infty(\R^2)$ we will write $f(x) \lesssim g(x)$ if there exists a numerical constant $c$ such that $f(x) \leq c\, g(x)$ for all $x\in\R^2$. The symbol $f(x) \gtrsim g(x)$ is defined analogously. Finally, for any $f\in L^\infty(\R^2)$ and $\alpha\in\R$ we will make use of the notation
$$
f(x)  \, \lesssim\,   (1+|x|)^{-\alpha-0} \quad \Rightarrow \quad  \ \lim_{|x|\to \infty} (1+|x|)^{\alpha}\  f(x) = 0.
$$

\smallskip

\noindent Here is the assumption on the magnetic field:

\begin{assumption} \label{ass-B}
Assume that for any multi-index $\beta\in\N_0^2$ and some $s>4$ it holds
\begin{equation} \label{ass-B-eq}
| \pd^\beta B(x)| \ \lesssim\ (1+|x|)^{-s-|\beta|}.
\end{equation} 
\end{assumption}

%\smallskip 

\noindent
Under this condition $B$ obviously belongs to $L^1(\R^2)$ and therefore has finite flux $\alpha$.  To any magnetic field which satisfies \eqref{ass-B-eq} may be  associated a bounded vector field $A$ such that $\rt\, A =B$, see Lemma \ref{lem-gauge}. We then add an electric potential $V\in L^\infty(\R^2)$ and define the operator $H(B, V)$ through the closed quadratic form
$$
Q[u]= \int_{\R^2}\, \big ( |(i\nabla +A) \, u|^2 + V |u|^2 \big)\, dx, \qquad u\in W^{1,2}(\R^2).
$$
Moreover, we will always assume that $V(x) \to 0$ as $|x|\to \infty$. Hence by standard compactness arguments 
$$
\sigma_{es} (H(B,V)) = [0,\infty) .
$$
Consequently, we use the standard definition of a regular point; we say that {\em zero is a regular point of $H(B,V)$} if there exists $s >1/2$ such that 
\begin{equation} \label{def-regular}
\limsup_{\lambda\to 0}\, \|\, \rho^{-s}\, (H(B,V) -\lambda-i0)^{-1}\, \rho^{-s}\|_{\B(L^2(\R^2))} \ < \ \infty.
\end{equation} 

\smallskip

\begin{assumption} \label{ass-BV}
Suppose that the operator $H(B,V)$ has no positive eigenvalues.
\end{assumption}

\noindent 

%%%%%%%%%%%%%%%%%%%%%%%%%%%%%%%%%%%%%%%%%%%%%%%%%%%%%%%%%%%%
\subsection{Resolvent expansion at threshold} Let $A_0:\R^2\to\R^2$ be a vector potential given by 
\begin{equation}  \label{a0}
A_0(x) = \alpha\, (-x_2, x_1) \left\{
\begin{array}{l@{\quad}l}
|x|^{-1} \, &\quad |x| \leq 1 \  , \\
|x|^{-2}\,  &\quad |x| >1 \ , 
\end{array}
\right.
\end{equation}
and let $T(B,V)$ be defined by
\begin{equation} \label{tb}
T(B,V)  =  2 i (A-A_0)\cdot \nabla  +i \nabla\cdot A + (|A|^2-|A_0|^2)  + V. 
\end{equation}
Below $G_0, G_1, \G_0$ and $\G_1$ are integral operators in $L^2(\R^2)$ defined in sections \ref{sec-model} and \ref{sec-integer}, see equations \eqref{G0}, \eqref{G1}, \eqref{G0-a} and \eqref{G1-a} respectively.

\begin{theorem} \label{thm-1}
Let $\alpha\not\in\Z$. Suppose that assumptions \ref{ass-B} and \ref{ass-BV} are satisfied.  Suppose moreover that $s> 3/2$ and $|V(x)| \lesssim (1+|x|)^{-3-0}$. If zero is a regular point of $H(B,V)$, then there exist operators $F_j(B,V)\in  \B(s,-s ), \, j=0,1$, such that 
\begin{equation} \label{eq-zero-en1}
(H(B,V)- \lambda- i0)^{-1} = F_0(B,V) + F_1(B,V)\, \lambda^{\mu(\alpha)}+ o(\lambda^{\mu(\alpha)})
 \quad \text{in } \ \ \B(s,-s )
\end{equation}
as $\lambda\to 0$. Moreover, for $\mu(\alpha) < 1/2$ it holds
\begin{align*}
F_0(B,V) & = (1+G_0\, T(B,V))^{-1}\, G_0, \quad F_1(B,V)  = (1+G_0\, T(B,V))^{-1}\, G_1\, (1+T(B, V)\,  G_0 )^{-1}.
\end{align*} 
\end{theorem}

\smallskip

\noindent It would be possible to provide an explicit formula for the operators $F_0(B,V)$ and $F_1(B,V)$ also in the case $\mu(\alpha)=1/2$. In order to avoid complicated expressions with heavy notation we prefer not to do so, see section \ref{sec-model} for details. 
In the case of an integer flux we have

\begin{theorem} \label{thm-2}
Let  assumptions \ref{ass-B} and \ref{ass-BV} be satisfied and assume that $\alpha\in\Z$. Let  $s>3/2$. 
If $|V(x)| \lesssim (1+|x|)^{-3-0}$ is such that zero is a regular point of $H(B,V)$, then
\begin{equation} \label{eq-zero-en2}
(H(B,V)- \lambda- i0)^{-1} = \F_0(B,V) + \F_1(B,V)\, (\log \lambda)^{-1}+ o((\log |\lambda|)^{-1})
 \quad \text{in } \ \ \B(s,-s )
\end{equation}
as $\lambda\to 0$, where 
\begin{align*}
\F_0(B,V) & = (1+\G_0\, T(B,V))^{-1}\, \G_0, \quad \F_1(B,V)  = (1+\G_0\, T(B,V))^{-1}\, \G_1\, (1+T(B, V)\,  \G_0 )^{-1}.
\end{align*}
\end{theorem}

\smallskip

%%%%%%%%%%%%%%%%%%%%%%%%%%%%%%%%%%%%%%%%%%%%%%%%%%%%%%%%%%%%
\subsection{Time decay} 
Let $\mathcal{H}_d$ be the subspace of  $L^2(\R^2)$ spanned by normalized eigenfunctions  corresponding to discrete eigenvalues of $H(B,V)$. 
We denote by $P_c$ the projection on the orthogonal complement of $\mathcal{H}_d$ in $L^2(\R^2)$ .

\begin{theorem} \label{thm-3}
Let  assumptions \ref{ass-B} and \ref{ass-BV} be satisfied. Assume that $\alpha\not\in\Z$ and let $s> 5/2$. If $|V(x)| \lesssim (1+|x|)^{-3-0}$ is such that zero is a regular point of $H(B,V)$, then there exists $K(B,V)\in\B(s,s^{-1})$ such that as 
$t\to\infty$ 
\begin{equation} \label{eq-time1}
e^{- i t H(B,V)}\, P_c = K(B,V)\ t^{-1-\mu(\alpha)} \ + \  o(\, t^{-1-\mu(\alpha)})
\end{equation}
in $\B(s,-s )$, where 
$$
K(B,V) = \frac{i}{\pi}\, \sin(\pi\mu(\alpha))\, e^{i \pi\mu(\alpha)/2}\, \Gamma(1+\mu(\alpha))\ F_1(B,V).
$$
\end{theorem}

\begin{remark} The improved decay rate induced by a two-dimensional magnetic field was observed also for the heat semi-group $e^{-t H(B,0)}$, \cite{k11, kr}.
\end{remark}

\smallskip

\begin{theorem} \label{thm-4}
Let  assumptions \ref{ass-B} and \ref{ass-BV} be satisfied. Assume that $\alpha\in\Z$ and let $s> 5/2$. If $|V(x)| \lesssim (1+|x|)^{-3-0}$ is such that zero is a regular point of $H(B,V)$, then  as $t\to\infty$ 
\begin{equation} \label{eq-time2}
e^{- i t H(B,V)}\, P_c =  -i \, t^{-1} (\log t)^{-2} \ \F_1(B,V) \ + \ o \big(\, t^{-1} (\log t)^{-2}\big)
\end{equation}
in $\B(s,-s )$. 
\end{theorem}

\begin{remark}
It should be pointed out that, in view of the magnetic Hardy-type inequality \eqref{hardy}, zero is a regular point of $H(B,V)$ whenever $|V(x)| \leq V_0\, (1+|x|)^{-2}$ with $V_0$ small enough.  
\end{remark}

%%%%%%%%%%%%%%%%%%%%%%%%%%%%%%%%%%%%%%%%%%%%%%%%%%
\section{\bf Discussion}

\subsection{Hardy inequality} If the magnetic field satisfies assumption \ref{ass-B}, then by \cite{lw, timo, kov} there exists a positive constant $C_h=C_h(B)$ such that the inequality 
\begin{equation} \label{hardy}
H(B) \, \geq \, \frac{C_h}{w}, \qquad \qquad w(x) = 
\left\{
\begin{array}{l@{\quad}l}
1+ |x|^2, &\quad \alpha\not\in\Z  , \\
 & \\
1+ |x|^2\, (\log |x|)^2, &\quad \alpha\in\Z
\end{array}
\right.
\end{equation}
holds in the sense of quadratic forms on $W^{1,2}(\R^2)$.  

\subsection{Dispersive estimates} Since $e^{- i t H(B,V)}$ is a unitary operator from $L^2(\R^2)$ onto itself it is obvious that 
$\| (1+|x|)^{-s}\,  e^{- i t H(B,V)} \, (1+|x|)^{-s} \|_{L^2(\R^2)} \leq 1$ for every $s\geq 0$ and $t>0$. In combination with Theorem \ref{thm-3} we thus get the dispersive estimate
\begin{equation} \label{disp-L2}
\| (1+|x|)^{-s}\, e^{- i t H(B,V)}\, P_c\, (1+|x|)^{-s}\, \|_{L^2(\R^2)} \ \lesssim \ t^{-1-\mu(\alpha)} \qquad \forall\ t >0.
\end{equation}
We note once again that the effect of faster decay with respect to the non-magnetic evolution is absent in dimension three, see \cite{kk} and \cite{egs1, egs2}. Three-dimensional dispersive estimates in weighted $L^2-$spaces in the absence of magnetic fields were first obtained by Rauch, \cite{ra}.
Extensions of these estimates to the $L^1 \to L^\infty$ setting, also in dimensions higher than three, were established in \cite{jss, gsch}. The case of dimension two was treated by Schlag in \cite{sch}.

\smallskip

An $L^1\to L^\infty$  dispersive estimate which corresponds to Murata's asymptotic expansion \eqref{murata} has been obtained only recently by Erdogan and Green in \cite{eg}. They showed that 
\begin{equation} \label{erd-green}
\| \, (\log(2+|x|)^{-2}\, e^{-i t H(0,V)}\, P_{ac} \, (\log(2+|x|)^{-2} \|_{L^1(\R^2)\to L^\infty(\R^2)} \, \lesssim\, t^{-1}\, (\log t)^{-2} \qquad t>2,  
\end{equation}
provided zero is a regular point of $H(0,V)$ and $|V(x)| \lesssim (1+|x|)^{-3-0}$. It is interesting to observe that the logarithmic factor on the left hand side of \eqref{erd-green} appears also in the weight function $w$ for the Hardy inequality \eqref{hardy} in the case $\alpha\in\Z$.

\subsection{The case of zero flux} \label{zero-alpha}When $\alpha=0$, then we cannot apply our perturbative approach, since the reference magnetic field $B_0$ is identically zero in this case, see equations \eqref{a0} and \eqref{B0}. Instead, we treat the operator $H(B,V)$ as a perturbation the free Laplacian $-\Delta$ and apply a result of Murata, see \cite[Thms.8.4\&7.5]{mu}. This is possible thanks to the fact that for a magnetic field with zero flux we can find a corresponding vector potential 
with sufficient decay at infinity, cf. Lemma \ref{lem-gauge-2}. 

It must be mentioned, however, that this is the only case in which 
the perturbation  with respect to the free Laplacian, i.e. the operator $H(B,V) + \Delta$, has coefficients decaying fast enough in order to compensate for the growth of the weight function $(1+|x|)^s$ with $s>1$. Indeed, if $\rt\, A_1=B_1$ and $\rt\, A_2=B_2$, then by the Stokes Theorem we have
\begin{equation} \label{arg-stokes}
|A_1(x)-A_2(x)| = o(|x|^{-1}) \quad \text{as}\quad |x|\to\infty \quad  \Rightarrow \quad \int_{\R^2} B_1(x)\, dx = \int_{\R^2} B_2(x)\, dx,
\end{equation}
see also \cite{hb}. Hence if $B_1=0$ and $B_2=B$ is such that $\int_{\R^2} B \neq 0$, then the coefficients of the perturbation $H(B,V) + \Delta= 2 i\, A\cdot \nabla  +i \nabla\cdot A + |A|^2  + V$ cannot decay faster than $|x|^{-1}$ irrespectively of the decay rate of $B$.

\subsection{Long range magnetic fields} One might expect that the main results of this paper should remain valid also if $B$ has a slower decay than the one required by assumption \ref{ass-B}, as long as the flux $\alpha$ remains finite. The asymptotic expansion of the resolvent for decaying magnetic fields with infinite flux is an open question.

\subsection{$ L^1\to L^\infty$ estimates} It would be interesting to extend the dispersive estimate \eqref{disp-L2} to an $L^1\to L^\infty$ setting. So far this is known only for the Aharonov-Bohm magnetic field $B_{ab}$ with flux $\alpha$, see \cite{fffp,gk}. Such a magnetic field is 
generated by the vector potential 
$$
A_{ab}(x)= \frac{\alpha}{|x|^2} \, (-x_2, x_1).
$$ 
Denote by $H(B_{ab}, 0)$  the Friedrichs extension of the operator 
$(i \nabla +A_{ab})^2$ defined on $C_0^\infty(\R^2\setminus\{0\})$. In \cite[Thm.1.9]{fffp} it is proved that 
\begin{equation} \label{fffp}
\| \ e^{-i t H(B_{ab},0)}\|_{L^1(\R^2)\to L^\infty(\R^2)} \, \lesssim\, t^{-1}  , \qquad t>0.
\end{equation}
A related weighted estimate with an improved decay rate was obtained in \cite[Cor.3.3]{gk}:
\begin{equation} \label{gk}
\| \, (1+|x|)^{-\mu(\alpha)}\, e^{-i t H(B_{ab},0)}\, (1+|x|)^{-\mu(\alpha)}  \|_{L^1(\R^2)\to L^\infty(\R^2)} \, \lesssim\, t^{-1-\mu(\alpha)}  \qquad t>0.
\end{equation}
Note that \eqref{gk} gives the decay rate $t^{-1}$ when $\alpha\in\Z$. This is not surprising since the Aharonov-Bohm operator with an integer flux is unitarily equivalent to the free Laplacian $-\Delta$ in $\R^2$. 

%\smallskip

%%%%%%%%%%%%%%%%%%%%%%%%%%%%%%%%%%%%%%%%%%%%%%%%%%
\section{\bf Proofs of the main results}
\label{sec-proofs}

\noindent In this section we assume throughout that $B$ and $V$ satisfy assumptions of Theorems \ref{thm-1} and \ref{thm-2}.
Under these conditions the limiting absorption principle holds for the operator $H(B,0)$ and 
\begin{equation} \label{lap}
 \|\, \rho^{-s}\, (H(B,0) -\lambda-i0)^{-1}\, \rho^{-s}\|_{\B(L^2(\R^2))} \ < \ \infty
\end{equation}
for any $s >1/2$ and any $\lambda>0$, see \cite[Sect.~5]{ro}, \cite{jmp}. We have

\begin{lemma} \label{lem-lap}
Assume that $|V(x)| \lesssim (1+|x|)^{-3-0}$. Then
$$
 \|\, \rho^{-s}\, (H(B,V) -\lambda-i0)^{-1}\, \rho^{-s}\|_{\B(L^2(\R^2))} \ < \ \infty
$$
for any $s \geq 1$ and any $\lambda>0$.
\end{lemma}

\begin{proof}
Let $1/2 < s_0 \leq 1$ and let $\lambda>0$. For convenience we denote
\begin{equation} \label{res-B}
R_B(\lambda+i0) = (H(B,0)-\lambda-i0)^{-1}.
\end{equation}
Since $\rho(x)^{s_0} V(x) \to 0$ as $|x|\to \infty$, the operator $\rho^{s_0} V\, (-\Delta -\lambda-i\eps)^{-1}$ is compact in $L^2(\R^2)$ for any $\eps>0$. Hence by the diamagnetic inequality the same is true for $\rho^{s_0} V\, R_B(\lambda+i0)\, \rho^{-s_0}$. From $\rho^{2s_0} V \in L^\infty(\R^2)$ and from \eqref{lap} it thus follows that $\rho^{s_0} V\, R_B(\lambda+i0) \, \rho^{-s_0}$ is compact in $L^2(\R^2)$ being a uniform limit of compact operators. In other words, 
$ V\, R_B(\lambda+i0) $ is compact in $\B(s_0, s_0)$. 

Hence $-1$ is not an eigenvalue of $V\, R_B(\lambda+i0) $ since otherwise $\lambda$ would be an eigenvalue of the operator $H(B,V)$ which contradicts assumption \ref{ass-BV}. This implies that the operator $1+V\, R_B(\lambda+i0) $ is invertible in $\B(s_0, s_0)$ and by the resolvent equation it follows that 
$$
 (H(B,V) -\lambda-i0)^{-1} = (1+V\, R_B(\lambda+i0) )^{-1}\, R_B(\lambda+i0) 
$$
exists in $\B(s_0, s_0^{-1})$. The claim thus follows from the fact that $s_0\leq 1$. 
\end{proof}

\noindent By the above Lemma the positive part of the spectrum of $H(B,V)$ is purely absolutely continuous. The negative part of the spectrum is either empty or consists of a finite number of eigenvalues each having finite multiplicity, see \cite[Thm.3.1]{kov}. 

\subsection{Expansion at threshold} If zero is a regular point of $H(B,V)$, then the statement of Lemma \ref{lem-lap} can be extended to $\lambda=0$. Moreover, from the finiteness of the discrete spectrum of $H(B,V)$ it then follows that $H(B,V)$ has no spectrum in the left neighborhood of zero. Hence 
$$
\sup_{-\delta \leq \lambda \leq \delta}\, \|\, \rho^{-1}\, (H(B,V) -\lambda-i0)^{-1}\, \rho^{-1}\|_{\B(L^2(\R^2))} \ < \ \infty
$$
for $s \geq 1$ and $\delta>0$ small enough. 

\smallskip

\noindent To prove Theorems \ref{thm-1} and \ref{thm-2} (for $\alpha\neq 0$) we will employ the perturbation procedure mentioned in the introduction. First we establish the asymptotic expansions of the type \eqref{eq-zero-en1} and \eqref{eq-zero-en2} for the resolvent of an operator $H(B_0) =H(B_0,0)$, where $B_0$ is  by equation \eqref{B0}, see Propositions \ref{prop-1}, \ref{prop-2}.  Then we show by the perturbative technique that any other magnetic field with the same flux as $B_0$ gives rise to an operator with (qualitatively) the same asymptotic expansion of the resolvent. Adding a bounded electric potential $V$ with a fast enough decay at infinity then won't change the character of the obtained expansion. 
In the case $\alpha=0$, which concerns Theorem  \ref{thm-2} only, we repeat the same procedure with $H(B_0)$ replaced by  $-\Delta$. 

\smallskip

\subsubsection*{\bf The reference operator} The  reference operator $H(B_0)$, which will play the role of the free Hamiltonian when $\alpha\neq 0$, is associated to the radial magnetic field $B_0$
given by
\begin{equation} \label{B0}
B_0(x)= B_0(|x|) = \frac{\alpha}{|x|} \quad \text{if}\quad |x| < 1, \qquad B_0(x) =0 \quad \text{otherwise}.
\end{equation}
It is easily seen that $B_0 = \rt \, A_0$ and that the flux of $B_0$ through $\R^2$ is equal to $\alpha$. 
Let
\begin{equation} \label{def-r0}
R_0(\lambda+i0)= (H(B_0)-\lambda-i0)^{-1}.
\end{equation}
We have

\begin{proposition} \label{prop-1}
Let $\alpha\not\in\Z$ and let $s> 3/2$. Then there exists $\widetilde G_1\in  \B(s,-s )$ such that 
\begin{equation} \label{B0-eq-1}
R_0(\lambda+i0) = G_0 + \lambda^{\mu(\alpha)}\, G_1 + G_2(\lambda) \quad \text{in} \ \ \B(s,-s ),
\end{equation}
where
$$
\|\, G_2(\lambda) \|_{\B(s,-s )} = o\big(|\lambda|^{\mu(\alpha)} \big), \qquad \| |\nabla G_2(\lambda) |\|_{\B(s,-s )} = o\big(|\lambda|^{\mu(\alpha)}\big) \qquad  \lambda\to 0. 
$$
Moreover, if $\mu(\alpha) <1/2$, then $\widetilde G_1= G_1$.
\end{proposition}

\smallskip

\begin{proposition} \label{prop-2}
Let $\alpha\in\Z,\, \alpha\neq 0$ and let  $s> 3/2$. Then
\begin{equation} \label{B0-eq-2}
R_0(\lambda+i0) = G_0 + (\log \lambda)^{-1}\, \G_1 + \G_2(\lambda) \quad \text{in} \ \ \B(s,-s ),
\end{equation}
where 
$$
\|\, \G_2(\lambda) \|_{\B(s,-s )} = o\big(\log |\lambda|)^{-1}\big), \qquad \| |\nabla \G_2(\lambda) |\|_{\B(s,-s )} = o\big(\log |\lambda|)^{-1}\big) \qquad  \lambda\to 0.
$$
\end{proposition}

\smallskip

\noindent Proofs of Propositions \ref{prop-1} and \ref{prop-2} are given in sections \ref{sec-model} and \ref{sec-integer} respectively.

\begin{remark}
Note that the field $B_0$ does not satisfy assumption \ref{ass-B}.
\end{remark}

\noindent The following Lemma plays a crucial role in our approach, for it allows us to extend the results of Propositions \ref{prop-1} and \ref{prop-2} to all magnetic fields satisfying assumption \ref{ass-B}. 

\begin{lemma} \label{lem-gauge}
Let $B$ satisfy assumption \ref{ass-B} and let $\alpha$ be given by \eqref{flux}. Suppose that $\alpha\neq 0$. Then there exists a differentiable vector field $A:\R^2\to \R^2$ such that $\rt \, A = B$ and such that 
\begin{equation} \label{vp}
 | \nabla \cdot A(x)|  \ \lesssim \  (1+|x|)^{-3-0} \ , \qquad
| \, A(x) - A_0(x) | \ \lesssim \  (1+|x|)^{-3-0} 
\end{equation}
where $A_0$ is given by \eqref{a0}.
\end{lemma}

\begin{proof}
Let $\widehat  A$ be the vector potential associated to $B$ by the Poincar\'e gauge: 
\begin{equation} \label{ahat}
\widehat  A (x) = (-x_2, x_1)\ \int_0^1 B( t\, x_1, t\, x_2)\, t\, dt,
\end{equation}
see e.g. \cite[Eq. (8.154)]{th}. A direct calculation then shows that $\rt \, \hat A =B$. Passing to the polar coordinates we get
\begin{align}
\widehat  A(r,\theta) &= r (-\sin\theta, \cos\theta) \int_0^1 B( t\, r , \theta) \, t dt = \frac{(-\sin\theta, \cos\theta) }{r}\, \int_0^r B(z,\theta)\, z\, dz \label{hat-1} 
%\\ & = \frac{(-\sin\theta, \cos\theta) }{r}\, \Big(\psi(\theta) - \int_r^\infty B(z,\theta)\, z\, dz\Big), 
\end{align}
Hence with the notation 
\begin{equation} \label{psi}
\psi(\theta) = \int_0^\infty B(z,\theta)\, z\, dz
\end{equation}
we obtain the decomposition
$$
\widehat  A(r,\theta)  = F_1(r, \theta)+ F_2(r,\theta),
$$
where 
\begin{equation} \label{f12}
F_1(r, \theta)  = \frac{(-\sin\theta, \cos\theta) }{r}\ \psi(\theta), \qquad F_2(r,\theta) = \frac{(\sin\theta, -\cos\theta) }{r}\  \int_r^\infty B(z,\theta)\, z\, dz.
\end{equation}
Note that by equation \eqref{flux} 
\begin{equation} \label{psi2}
\int_0^{2\pi} \psi(\theta) = 2\pi \alpha .
\end{equation}
Now fix an $R>1$ and let $\gamma\subset\R^2\setminus D(0,R)$ be a piece-wise regular simple closed curve. We denote by $\Omega_\gamma\subset\R^2$ the region enclosed by $\gamma$. Note that  $\rt\, F_1 = \rt\,  A_0 = 0$ in $\R^2\setminus D(0,R)$. Hence if $\Omega_\gamma$ does not intersect $D(0,R)$, then $\Omega_\gamma \setminus D(0,R)$ is simply connected and
$$
\oint_\gamma F_1  = \oint_\gamma A_0  =0.
$$
On the other hand, if $\Omega_\gamma$ does intersect $D(0,R)$, then it contains $D(0,R)$ as a proper subset, for $\gamma\cap D(0,R) = \emptyset$. In this case, in view of \eqref{f12} and \eqref{psi2}, it  turns out that
$$
\oint_\gamma  F_1 = 2\pi \alpha.
$$
By the Stokes Theorem we then have
$$
\oint_\gamma F_1  = 2 \pi \alpha = \int_{D(0,2)}  B_0(x) \, dx = 
\oint_\gamma A_0. 
$$
So, in either case it holds
$$
\oint_\gamma (F_1-A_0)  =0
$$
This means that the vector field $F_1 -A_0$ is conservative in $\R^2\setminus  \overline{D(0,R)}$, since the latter is an open connected subset of $\R^2$. Moreover, by definition of $F_1$ and $A_0$ and by the  hypothesis on $B$ it follows that $F_1, A_0 \in C^1(\R^2\setminus  \overline{D(0,R)})$. Therefore there exists a scalar field $\varphi \in C^2(\R^2\setminus \overline{D(0,R)}) $ such that 
\begin{equation} \label{eq-cons} 
F_1 +\nabla \varphi = A_0 \qquad \text{in } \qquad \R^2\setminus  \overline{D(0,R)} \qquad  \ \ R >1.
\end{equation}
We now put $R \in(1,2)$. Then $\varphi \in C^2(\R^2\setminus D(0,2))$. 
Since $\R^2\setminus D(0,2)$ is closed, we can extend $\varphi$ into a function $\widehat \varphi \in C^2(\R^2)$ in such a way that
 $\widehat  \varphi = \varphi$ on $\R^2\setminus D(0,2)$, see e.g. \cite[Thm.1]{wi}. It now remains to define
$$
A =   \widehat A + \nabla \widehat  \varphi = F_1 +F_2 +\nabla \widehat  \varphi \qquad \text{in } \qquad \R^2. 
$$
Then $\rt\, A = \rt\, \widehat A= B$ and $\nabla \cdot A = \nabla \cdot \widehat A + \Delta \widehat\varphi$. However, $\nabla \cdot \widehat A \in L^\infty(\R^2)$ by equation \eqref{hat-1} and hypothesis on $B$, and $\Delta \widehat\varphi\in L^\infty(R^2)$ by construction of $\widehat\varphi$. Therefore $\nabla \cdot A \in L^\infty(\R^2)$. Finally, from \eqref{ass-B-eq} we easily verify that 
$$
 | \nabla \cdot F_2 (x)|  \ \lesssim \  (1+|x|)^{-3-0} \ , \qquad | \, F_2(x)  | \ \lesssim \  (1+|x|)^{-3-0} 
$$
Since $A-A_0 = F_2$ on $\R^2\setminus D(0,2)$ by \eqref{eq-cons} and since $\nabla \cdot A_0 =0$, this implies
\eqref{vp}.
\end{proof}

\begin{lemma} \label{lem-gauge-2}
Let $B$ satisfy assumption \ref{ass-B} and let $\alpha=0$. Then there exists a differentiable vector field $A:\R^2\to \R^2$ such that $\rt \, A = B$ and such that 
\begin{equation} \label{vp-2}
 | \nabla \cdot A(x)|  \ \lesssim \  (1+|x|)^{-3-0} \ , \qquad
| \, A(x) | \ \lesssim \  (1+|x|)^{-3-0} 
\end{equation}
\end{lemma}

\begin{proof}
In this case we replace $B_0$ by a continuous radial field $\widetilde B_0$ of compact support and with zero flux: 
\begin{equation} \label{zero-flux}
\int_{\R^2} \widetilde B_0(x)\, dx =\int_{\R^2} \widetilde B_0(|x|)\, dx= 0.
\end{equation}
Accordingly, we define 
\begin{align*}
\widetilde  A_0(x) & = (-x_2, x_1)\ \int_0^1 \widetilde B_0( t\, x_1, t\, x_2)\, t\, dt =  (-x_2, x_1)\ \int_0^1 \widetilde B_0( t\, |x|)\, t\, dt  = \frac{(-x_2, x_1)}{|x|}\ \int_0^{|x|}\, \widetilde B_0(s)\, s\, ds. 
\end{align*}
Then $\rt\, \widetilde  A_0 = \widetilde  B_0$. Equation \eqref{zero-flux} and the fact that the support of $\widetilde B_0$ is compact then imply $\widetilde A_0(x)=0$ for $|x|$ large enough. Now it remains to follow the proof of Lemma \ref{lem-gauge} with $A_0$ replaced by  $\widetilde A_0$ and with $\alpha=0$. 
\end{proof}

\noindent From now on we will associate to any $B$ satisfying \eqref{ass-B-eq} a vector potential $A$ given by Lemma \ref{lem-gauge} when $\alpha\neq 0$, or by Lemma \ref{lem-gauge-2} when $\alpha=0$.  

\begin{remark}
The fact that we use a particular vector potential generating the magnetic field $B$ represents no restriction, since all our statements are gauge invariant. Note also that $A$ is  not uniquely defined by Lemmata \ref{lem-gauge} respectively \ref{lem-gauge-2}. 
\end{remark}

\noindent  Since $\nabla\cdot A_0=0$, see \eqref{a0}, from the definition of $T(B,V)$ it follows that
\begin{equation} \label{tb-2}
T(B,V)  = H(B, V) -H(B_0). 
\end{equation}
Note also that $T(B,V)$ is symmetric  on $W^{1,2}(\R^2)$, see \eqref{tb}.

\begin{lemma} \label{lem-aux1}
If $B$ satisfies assumption \ref{ass-B}, then there exists $s_0> 3/2$ such that  the operator $1+ T(B,V)\, G_0$ is invertible in $\B(s,s)$ for all $3/2 < s \leq s_0 $. 
\end{lemma}

\begin{proof}
By Lemma \ref{lem-gauge} we can find $s_0>3/2$ such that the functions $ \rho^{s_0}\, (\nabla\cdot A)\, \rho^{s_0} $ and $\rho^{s_0}\,  |A-A_0|\, \rho^{s_0}$ are bounded. From Lemma \ref{lem-g0} it then follows that the operator $T(B,V)\, G_0$ is compact from $L^2(\R^2, s)$ to $L^2(\R^2, s)$ for any $s\in (3/2, s_0)$. Assume that there exists $u \in L^2(\R^2, s)$ such that 
\begin{equation} \label{no-ev}
u+ T(B,V) G_0\, u =0.
\end{equation}
The resolvent equation in combination with \eqref{tb-2} says that for every $\eps>0$ 
$$
 R_0(i \eps)= (H(B,V)+i\eps)^{-1}\, (1+T(B,V)\,  R_0(i \eps))
$$
holds on $L^2(\R^2, s)$. Hence using equation \eqref{def-regular}, Proposition \ref{prop-1} and passing to the limit $\eps\to 0$ we arrive at 
$$
G_0\, u = H(B,V)^{-1}\, (1+T(B,V) \, G_0)\, u =0,
$$
since $(1+T(B,V) \, G_0)\, u=0$ by \eqref{no-ev}. But then $u=0$ again in view of \eqref{no-ev}. This means that Ker$(1+T(B,V)\, G_0)=\{0\}$ and by the Fredholm alternative $1+T(B,V)\,  G_0$ is invertible.
\end{proof}

\begin{lemma} \label{lem-aux2}
If $B$ satisfies assumption \ref{ass-B}, then there exists $s_0> 3/2$ such that the operator $1+ G_0\,T(B,V)$ is invertible in $\B(-s,-s)$ for all $3/2 < s \leq s_0 $.  
\end{lemma}

\begin{proof}
The claim follows by duality from Lemma \ref{lem-aux1}.
\end{proof}

\begin{lemma} \label{lem-rv}
Assume \eqref{ass-B} and let $\alpha \not\in\Z$. Then  for $\lambda\to 0$ we have 
\begin{align} \label{eq-rv}
& (1+ R_0(\lambda+i0) \, T(B,V))^{-1}  = (1+G_0\, T(B,V))^{-1}\, - \\ 
& \qquad \qquad \qquad\qquad  - (1+G_0\, T(B,V))^{-1}\, G_1\, T(B,V)\, (1+G_0\, T(B,V))^{-1}\, \lambda^{\mu(\alpha)} \nonumber + o(\lambda^{\mu(\alpha)})
\end{align}
in $\B(-s, -s)$ for some $s>3/2$. 
\end{lemma} 

\begin{proof}
From Lemma \ref{lem-gauge} and Proposition \ref{prop-1} it follows that $\| T(B,V)\, G_2(\lambda)\|_{\B(s, s)} =  
o\big(|\lambda|^{\mu(\alpha)}\big)$ as $\lambda\to 0$. Hence by \eqref{B0-eq-1} and duality we have
\begin{equation}
1+R_0(\lambda+i0)\, T(B,V)= 1+G_0\, T(B,V) + G_1\, T(B,V)\, \lambda^{\mu(\alpha)} + o(|\lambda|^{\mu(\alpha)})
\end{equation}
in $\B(-s, -s)$ for $\lambda\to 0$. The operator $1+G_0\, T(B,V)$ is invertible in $\B(-s, -s)$ in view of Lemma \ref{lem-aux2}. Hence  for $|\lambda|$ small enough the operator $1+R_0(\lambda+i0)\, T(B,V)$ is invertible too and with the help of the Neumann series we arrive at \eqref{eq-rv}.
\end{proof}

\begin{proof}[\bf Proof of Theorem \ref{thm-1}]
It suffices to prove the statement for $s\leq s_0$ with $s_0$ given by Lemma \ref{lem-aux1}.  Similarly as for the free resolvent we introduce the notation 
\begin{equation} \label{res-def}
R(\lambda+i0)= (H(B, V)-\lambda-i0)^{-1}.
\end{equation}
Since $1+R_0(\lambda+i0)\, T(B,V)$ is invertible in $\B(-s, -s)$ for $\lambda$ small enough, the
resolvent equation yields
\begin{equation} \label{2-res}
R(\lambda+i0)=(1+R_0(\lambda+i0)\, T(B,V))^{-1}\, R_0(\lambda+i0),
\end{equation}
which in combination with \eqref{B0-eq-1} and \eqref{eq-rv} gives 
\begin{equation} \label{eq-2.7}
R(\lambda+i0)=   (1+G_0\, T(B,V))^{-1}\, G_0 + F_1(B,V)\, \lambda^{\mu(\alpha)}+ o(\lambda^{\mu(\alpha)})
\end{equation}
as $\lambda\to 0$, where
\begin{align*}
 F_1(B,V) &  = (1+G_0\, T(B,V))^{-1}\, G_1 -(1+ G_0 \, T(B,V) )^{-1}\, G_1\, T(B,V)\, (1+G_0\, T(B,V))^{-1} \, G_0 \\
& = (1+G_0\, T(B,V))^{-1}\, G_1(1+ T(B,V)\, G_0)^{-1} \big [1+T(B,V)\, G_0  \\
& \qquad\qquad\qquad\qquad \qquad\qquad-  (1+ T(B,V)\, G_0)\, T(B,V)\, (1+G_0\, T(B,V))^{-1} \, G_0 \big]
\end{align*}
in $\B(s,-s )$. Now let $u\in L^2(\R^2, s)$. Then $G_0\, u \in L^2(\R^2, s^{-1})$ and given a $g\in L^2(\R^2, s^{-1})$ it follows by Lemma \ref{lem-rv} that
$$
(1+G_0 \, T(B,V))^{-1} \, G_0\, u = g \quad \Leftrightarrow \quad G_0\, u= (1+G_0 \, T(B,V))\, g,
$$
and
$$
(1+T(B,V)\, G_0 )^{-1} \, T(B,V)\,  G_0\, u  = T(B,V)\, g \  \ \Leftrightarrow \  \ T(B,V))\, G_0\, u = T(B,V)(1+  G_0\, T(B,V))\, g.
$$
Hence
\begin{equation}
T(B,V)\, (1+G_0 \, T(B,V))^{-1} \, G_0 = (1+T(B,V)\, G_0 )^{-1} \, T(B,V)\,  G_0,
\end{equation}
holds on $L^2(\R^2, s)$, which implies that 
$$
 F_1(B,V) =  (1+G_0\, T(B,V))^{-1}\, G_1\, (1+ T(B,V)\, G_0)^{-1}.
$$
The claim then follows from \eqref{eq-2.7}.
\end{proof}

\begin{proof}[\bf Proof of Theorem \ref{thm-2}]
For $\alpha\neq 0$ the result follows from Proposition \ref{prop-2} and Lemma \ref{lem-gauge} in the same way as in the case $\alpha\not\in\Z$; one only needs to replace the operators $G_0$ and $G_1$ by $\G_0$ and $\G_1$ respectively, and use Lemma \ref{lem-g01-int} instead of Lemma \ref{lem-g0}. 

\smallskip

\noindent When $\alpha=0$, then we replace the reference operator $H(B_0)$ by the Laplacian $-\Delta$. Consequently, we write   
$$
H(B, V) = -\Delta + 2 i\, A\cdot \nabla  +i \, \nabla\cdot A + |A|^2  + V. 
$$ 
The statement now follows by Lemma \ref{lem-gauge-2} and Theorems 8.4 and 7.5.(iii) of \cite{mu}. 
\end{proof}

%%%%%%%%%%%%%%%%%%%%%%%%%%%%%%%%%%%%%%%%%%%
\subsection{Time decay} We will use the formula
\begin{equation} \label{eq-repr}
e^{-i t H(B,V)}  = \frac{1}{2\pi i}\, \int_\R e^{-i t \lambda}\, R(\lambda+i0)\, d\lambda, \qquad t>0.
\end{equation}
To prove Theorems \ref{thm-3} and \ref{thm-4} we have to estimate the behavior of $R(\lambda+i0)$ for $|\lambda|\to \infty$. 

\begin{lemma} \label{lem-high-en}
Let  $s >5/2$. Suppose that $B$ satisfies assumption \ref{ass-B}  and that $|V(x)| \lesssim (1+|x|)^{-\beta}, \, \beta >3$. Let $R^{(j)}(\lambda+i0)$ denote the $j$th derivative of $R(\lambda+i0)$ with respect to $\lambda$ in $\B(s,-s )$. Then 
\begin{align}
\| R^{(2)}(\lambda+i0)\|_{\B(s,-s )} & = \mathcal{O}(\lambda^{-3/2} )\qquad  \ \ \lambda\to \infty,    \label{+infty} \\
\| R^{(2)}(\lambda+i0)\|_{\B(s,-s )} & = \mathcal{O}(|\lambda|^{-3}) \qquad\quad   \lambda\to -\infty.    \label{-infty}
\end{align}
\end{lemma}

\begin{proof}
We will apply a perturbative argument. This time  the unperturbed operator will be $H(B,0)=H(B)$, so that $H(B,V)= H(B)+V$. 
By assumption \ref{ass-B} and \cite[Thm.5.1]{ro} we have 
\begin{equation} \label{decay-rb}
\| R_B^{(j)}(\lambda+i0)\|_{\B(s',-s' )} = \mathcal{O}(\lambda^{-\frac{j+1}{2}}) \qquad j=0,1,2,  \qquad \lambda\to\infty
\end{equation}
for any $s' > j +\frac 12$, where $R_B(\lambda+i0)$ is given by \eqref{res-B}. Now $R_B(\lambda+i0) \in \B(s', s'-\beta)$ for all $1/2 < s' < \beta-1/2$ by \eqref{decay-rb} and \eqref{lap}.  Since $V\in \B(s'-\beta, s')$ by assumption, it follows from \eqref{decay-rb} that 
\begin{equation} \label{decay-rb-1}
\| V\,  R_B(\lambda+i0)\|_{ \B(s', s')} \ \to 0 \qquad \forall\ s' \in (1/2, \, \beta -1/2).
\end{equation}
Hence $\|(1+V\,  R_B(\lambda+i0))^{-1}\|_{\B(s', s')} = \mathcal{O}(1) $ for all $1/2 < s' < \beta-1/2$ as $\lambda\to \infty$. By duality the same holds for $\|(1+R_B(\lambda+i0)\, V)^{-1}\|_{\B(-s', -s')}$. This in combination with \eqref{res-B} and the resolvent equation implies that the identities
\begin{align*}
R^{(1)}(\lambda+i0) & = (1+R_B(\lambda+i0)\, V)^{-1} \, R_B^{(1)}(\lambda+i0) \, (1+V\,  R_B(\lambda+i0))^{-1}\\
R^{(2)}(\lambda+i0) & =  (1+ R_B(\lambda+i0)\, V)^{-1} \, R_B^{(2)}(\lambda+i0) \, (1+V\,  R_B(\lambda+i0))^{-1} \\
& \quad -2 R^{(1)}(\lambda+i0) \, V \, R_B^{(1)}(\lambda+i0) (1+V\,  R_B(\lambda+i0))^{-1} , 
\end{align*}
hold in $B(s, -s)$ for all $5/2 < s < \beta -1/2$. The first equation shows that $\| R^{(1)}(\lambda+i0)\|_{\B(s',-s' )} = \mathcal{O}(\lambda^{-1})$. By inserting this together with \eqref{decay-rb} into the second equation  we obtain  \eqref{+infty}. 

\smallskip

\noindent As for the negative values of $\lambda$, we note that $R(\lambda)=(H(B,V)-\lambda)^{-1}$ is analytic in $\B(L^2(\R^2))$  for $|\lambda|$ large enough. This is a consequence of the fact the $H(B,V)$ has finitely many negative eigenvalues, \cite[Thm.3.1]{kov}. Hence 
$$
\| R^{(2)}(\lambda+i0)\|_{\B(L^2(\R^2))} = \| (H(B,V)-\lambda)^{-3}\|_{\B(L^2(\R^2))} = \mathcal{O}(|\lambda|^{-3}) \qquad\quad   \lambda\to -\infty,
$$ 
where we have used the fact that $ \| (H(B,V)-\lambda)^{-1}\|_{\B(L^2(\R^2))} = ({\rm dist}(\sigma(H(B,V)), \, \lambda)^{-1}$. This implies equation \eqref{-infty}. 
\end{proof}

\noindent Before we come to the proof of Theorems \ref{thm-3} and \ref{thm-4}, we recall \cite[Lem.10.1]{JK}, from which it follows that if $F:\R \to \B(s,-s )$ is such that $F(\lambda)=0$ in a vicinity of zero and $F^{(2)} \in L^1(\R; \B(s,-s ))$, then 
\begin{equation} \label{eq-jk}
\int_\R e^{-i t \lambda}\,F(\lambda)\, d\lambda = o(\, t^{-2}) \qquad t\to \infty
\end{equation}
in $ \B(s,-s )$.

\begin{proof}[\bf Proof of Theorems \ref{thm-3} and \ref{thm-4}] As usual we will split the integral \eqref{eq-repr} into two parts relative to small and large energies. To this end we introduce a function $\phi\in C_0^\infty(\R)$ such that $0\leq \phi \leq 1$ and $\phi =1$ in a vicinity of $0$. By the resolvent equation 
$R(\lambda+i0) = (1+V\,  R_B(\lambda+i0))^{-1}\, V\,  R_B(\lambda+i0) $ and equations \eqref{decay-rb}, \eqref{decay-rb-1} we have $\|R(\lambda+i0)\|_{\B(s,-s )} =  \mathcal{O}(\lambda^{-\frac{1}{2}})$ as $\lambda\to\infty$ for $1/2 <s $. In view of \eqref{lap} Theorems \ref{thm-1}, \ref{thm-2} it thus follows that  $R(\lambda+i0)$ is uniformly bounded on $(0,\infty)$ in $\B(s,-s )$ for $1/2 <s $. On the other hand for $\lambda <0$ the operator $R(\lambda+i0)\, P_c$ is analytic in $\lambda$ with respect to the norm $\|\cdot\|_ {\B(L^2(\R^2))}$. Hence Lemma \ref{lem-high-en} in combination with equation \eqref{eq-jk} gives
\begin{equation} \label{large-cont}
\int_\R e^{-i t \lambda}\, (1-\phi(\lambda))\,  R(\lambda+i0) \, P_c \, d\lambda = o(t^{-2}) \qquad t\to \infty
\end{equation}
in $ \B(s,-s )$ for all $s> 5/2$. 
To estimate the contribution to \eqref{eq-repr} from small values of $\lambda$, we recall two results on Fourier transform: 
\begin{equation} \label{fourier-1}
\frac{1}{2\pi i}\, \int_\R e^{-i t \lambda}\, (\lambda+i0)^\nu\, d\lambda = \frac{i \, \sin(\pi\nu)}{\pi}\, e^{i \pi\nu/2}\, \Gamma(1+\nu)\ t^{-1-\nu}  \qquad \nu\in\R,
\end{equation}
and 
\begin{equation} \label{fourier-2}
\frac{1}{2\pi i}\, \int_\R e^{-i t \lambda}\, (\log(\lambda+i0))^{-k}\, d\lambda = i\, \sum_{j=k}^2  (-1)^k\, k \ t^{-1}\, (\log t)^{-k-1} + \mathcal{O}( t^{-1} (\log t)^{-4})
\end{equation}
for $k=1,2$ as $t \to \infty$, see e.g. \cite[Lems.6.6 -6.7]{mu}. The last two equations in combination with \eqref{eq-repr}, \eqref{eq-jk} and Theorems \ref{thm-1} and \ref{thm-2} then imply that as $t\to \infty$ 
\begin{equation} \label{final-1}
e^{-i t H(B,V)} \, P_c = i\, F_1(B,V)\  \frac{\sin(\pi\mu(\alpha))}{\pi}\ e^{i \frac{\pi\mu(\alpha)}{2}}\ \Gamma(1+\mu(\alpha))\ t^{-1-\mu(\alpha)} \,+  \, o(t^{-1-\mu(\alpha)}) \quad \alpha\not\in\Z
\end{equation}
and 
\begin{equation} \label{final-2}
e^{-i t H(B,V)} \, P_c = -i\,  \F_1(B,V)\ t^{-1} \, (\log t)^{-2} + o(\, t^{-1}\,  (\log t)^{-2}) \qquad\qquad \qquad\qquad \qquad  \  
 \ \alpha\in\Z
\end{equation}
in $ \B(s,-s )$ for all $s> 5/2$.
\end{proof}

\noindent Hence our main results are established provided we can prove auxiliary Propositions \ref{prop-1} and \ref{prop-2}. 
This will be done in the following two sections. However, the analysis of the resolvent of the operator $H(B_0)$ leads to rather lengthly calculations. Therefore, in order to keep the exposition as smooth as possible, we will often make use of auxiliary technical results presented in section \ref{sec-aux-bessel} and of selected properties of certain special functions which are described in Appendices \ref{sec-mu}, \ref{app-bessel} and \ref{app}.

%%%%%%%%%%%%%%%%%%%%%%%%%%%%%%%%%%%%%%%%%%%%
%%%%%%%%%%%%%%%%%%%%%%%%%%%%%%%%%%%%%%%%%%%%%%%%%%%%%%%

\section{\bf Operator $H(B_0)$: non-integrer flux}
\label{sec-model}

\noindent We are going to study the resolvent $(H(B_0) -\lambda-i0)^{-1}$ separately for positive and negative values of $\lambda$. We first derive an explicit expression for the integral kernel, and then we will discuss the behavior of $(H(B_0) -\lambda-i0)^{-1}$ in the vicinity of zero in a suitable operator norm. For the sake of brevity we will suppose that $\mu(\alpha) < 1/2$, which means that 
\begin{equation} \label{nalpha}
\exists\, ! \ k(\alpha) : \ \  \mu(\alpha) = |k(\alpha)+\alpha|.
\end{equation}
The case $\mu(\alpha)= 1/2$ when the minimum in \eqref{min} is attained for two different values of $k\in\Z$ can be treated in a completely analogous way.

%%%%%%%%%%%%%%%%%%%%%%%%%%%%%%%%%%%%%%%%%%%%%%%%%%%%%
\subsection{The case $\lambda>0$}
\label{sec-l-pos}
We have
\begin{lemma} \label{lem-no-int}
Assume that $\alpha\not\in\Z$. For any $x,y\in\R^2$ it holds
\begin{equation} \label{r0}
R_0(\lambda; x,y) = G_0(x,y) + G_1(x,y)\, \lambda^{\mu(\alpha)} + G_2^+(\lambda;x,y), 
\end{equation}
where $G_0(x,y), G_1(x,y)$ are given by equations \eqref{G0}, \eqref{G1}, and $G_2^+(\lambda;x,y) = o(\lambda^{\mu(\alpha)})$ as $\lambda\to 0+$.
\end{lemma}

\begin{proof}
Without loss of generality we may assume that $\alpha>0$. 
To calculate $R_0(\lambda; x,y)$ we write the vector potential
$A_0$ associated to the field $B_0$ through \eqref{a0} in polar coordinates: $ A_0(r,\theta) =  a_0(r)\, (-\sin\theta,\, \cos\theta)$, where 
\begin{equation} \label{a0-polar}
a_0(r) = \alpha \quad \text{if}\quad r < 1, \qquad a_0(r) = \frac{\alpha}{r} \quad \text{if} \quad r\geq 1.
\end{equation}
The quadratic form associated to $H(B_0)$ now reads 
\begin{equation} \label{q-form}
 \int_0^\infty\! \int_0^{2\pi} \left(|\pd_r u|^2+  |i\, r^{-1}\pd_\theta u+ a_0(r)\, u|^2\right) r\, dr d\theta .
\end{equation}
By expanding a given test function $u\in L^2(\R_+\times (0,2\pi))$ into a
Fourier series with respect to the basis $\{e^{i m\theta}\}_{m\in\Z}$ of $L^2((0,2\pi))$, we obtain
the decomposition
\begin{equation} \label{sum-gen}
H(B_0) =   \sum_{m\in\Z}  \oplus \left( h_m \otimes\mbox{id}\right)
\Pi_m,
\end{equation}
where $h_m$ are the operators  in $L^2(\R_+, r
dr)$ acting on their domain as 
\begin{equation} \label{hm}
h_m\, f = -f'' -\frac{1}{r} \, f' +\Big(\frac mr+  a_0(r)\Big)^2\, f,
\end{equation}
and $\Pi_m$ is given by
$$
(\Pi_m\, u)(r,\theta) = \frac{1}{2\pi}\, \int_0^{2\pi}\, e^{im(\theta-\theta')}\, u(r,\theta')\, d\theta'.
$$
The integral kernel of $R_0(\alpha;\lambda)$ then splits accordingly:
\begin{equation} \label{hk-gen}
R_0(\lambda; x,y) = \frac{1}{2\pi}\, \sum_{m\in\Z}\  R_0^m(\lambda; r,r')\
e^{im(\theta-\theta')}\, .
\end{equation}
Here $R_0^m(\lambda; r,r')$ denotes the integral kernel of $(h_m -\lambda-i0)^{-1}$ in  $L^2(\R_+, r
dr)$.
Now consider the operators
\begin{equation} \label{lbeta}
\h_m = \U\, h_m\, \U^{-1} \qquad \text{in \, \, \, } L^2(\R_+, dr),
\end{equation}
where $\U: L^2(\R_+, r\, dr)\to L^2(\R_+, dr)$ is a unitary mapping defined by $(\U f)(r) = r^{1/2} f(r)$. Note that $\h_m $ is subject to Dirichlet boundary condition at $0$ and that it acts, on its domain, as
\begin{equation} \label{frakh}
\h_m  f = -f''\,  - \frac{1}{4 r^2}\, f + \Big(\frac mr+  a_0(r)\Big)^2\, f.
\end{equation}
Let $\rr_0^m(\lambda; r,r')$ denote the integral kernel of $(\h_m -\lambda-i0)^{-1}$ in $L^2(\R_+, dr)$ and let $g\in L^2(\R_+, dr)$.  Since 
$(\h_m -\lambda-i0)^{-1} = \U\, (h_m -\lambda-i0)^{-1}\, \U^{-1}$ in $L^2(\R_+, dr)$ by \eqref{lbeta}, keeping in mind the fact that $R_0^m(\lambda; r,r')$ is the integral kernel of $(h_m -\lambda-i0)^{-1}$ in  $L^2(\R_+, r dr)$ we obtain
\begin{align*}
((\h_m -\lambda-i0)^{-1} g)(r) & = \sqrt{r} \int_0^\infty R_0^m(\lambda; r,r')\, (\U^{-1} g)(r') \, r' dr' =  \int_0^\infty  \sqrt{r r'}\, R_0^m(\lambda; r,r') \, g(r') \, dr' .
\end{align*}
This yields the identity
\begin{equation} \label{rtor}
R_0^m(\lambda; r,r') = \frac{1}{\sqrt{r r'}}\ \rr_0^m(\lambda; r,r').
\end{equation}
Hence it suffices to calculate $\rr_0^m(\lambda; r,r')$. To do so  we will find two solutions, $f_{m,\lambda}$ and $\phi_{m,\lambda}$ to the generalized eigenvalue equation 
\begin{equation} \label{eq-fg}
-f''\,  - \frac{f}{4 r^2} + \Big(\frac {m^2}{r^2}+ \frac{2\, m\, a_0(r)}{r} +a^2_0(r)\Big)\, f = \lambda\, f,
\end{equation}
such that $f_{m,\lambda+i\eps} \in H_0^1((0,1), r dr)$ and $\phi_{m,\lambda+i \eps} \in H^1((1,\infty), r dr)$ for $\eps>0$. The Sturm-Liouville theory then gives 
\begin{equation} \label{rm}
\rr^m_0(\lambda; r,r') =  \frac{1}{W_m(\lambda)}
\left\{
\begin{array}{l@{\quad}l}
f_{m,\lambda}(r)\, \phi_{m,\lambda}(r'), &\quad r \leq r', \\
 & \\
f_{m,\lambda}(r')\, \phi_{m,\lambda}(r), &\quad  r' < r ,
\end{array}
\right.
\end{equation}
where 
$$
W_m(\lambda)= \phi_{m,\lambda}\, f'_{m,\lambda} -\phi'_{m,\lambda} \, f_{m,\lambda}
$$
is the Wronskian of $\phi_{m,\lambda}$ and $f_{m,\lambda}$. In view of \eqref{frakh} it follows that $W\{\phi_{m,\lambda}, f_{m,\lambda}\}$ is constant. In order to find $f_{m,\lambda}$ and $\phi_{m,\lambda}$ we have to solve  equation \eqref{eq-fg} separately for $r\leq 1 $ and 
$r>1$ and match the solutions smoothly. Let 
\begin{equation} \label{kappa}
\kappa = \sqrt{\alpha^2-\lambda}\, ,
\end{equation}
To simplify the notation in the sequel we define the functions $v_{m}(\lambda, \cdot), u_m(\lambda, \cdot) : (0,1)\to \R$  by
\begin{align}
v_{m}(\lambda, r) &=  e^{-\kappa r}\, (2\kappa\, r)^{|m|}\, M\Big(\frac 12+|m|+\frac{m\alpha}{\kappa}, 1+2 |m|, 2\kappa r\Big) \label{vm}\\
u_m(\lambda, r) &=  e^{-\kappa r}\, (2\kappa\, r)^{|m|}\, U\Big(\frac 12+|m|+\frac{m\alpha}{\kappa}, 1+2 |m|, 2\kappa r\Big) \  \label{um},
\end{align}
where $M(a,b,z)$ and $U(a,b,z)$ are the Kummer's confluent hypergeometric functions, see \cite[Sec.13.1]{as}. Using equations \eqref{a0-polar}, \eqref{frakh} and a suitable change of variables we find they \eqref{eq-fg} lead to a Whittaker's equation for $r\leq 1$ and to the Bessel equation for $r>1$, see \cite[Chaps.9\&13]{as}. We thus obtain 
\begin{equation} \label{fm}
f_{m,\lambda}(r) =  \left\{
\begin{array}{l@{\quad}l}
\sqrt{r}\ v_m(\lambda,r), &\quad r \leq 1  , \\
 & \\
\sqrt{r}\, \big ( A_m(\lambda)\,  J_{|\alpha+m|}(\sqrt{\lambda}\, r) +B_m(\lambda)\, Y_{|\alpha+m|}(\sqrt{\lambda}\, r)\big), &\quad  1  < r ,
\end{array}
\right.
\end{equation}
where $A_m(\lambda), \, B_m(\lambda)$ are numerical coefficients whose values will be determined later. Similarly, 
\begin{equation}\label{fim}
\phi_{m,\lambda}(r) =  \left\{
\begin{array}{l@{\quad}l}
\sqrt{r}\  \big( C_m(\lambda)\, v_m(\lambda,r)
+D_m(\lambda)\, u_m(\lambda, r) \big) , &\quad r \leq 1  , \\
 & \\
\sqrt{r}\, \big ( J_{|\alpha+m|}(\sqrt{\lambda}\, r) +i \,Y_{|\alpha+m|}(\sqrt{\lambda}\, r)\big), &\quad  1  < r ,
\end{array}
\right.
\end{equation}
In order to find the coefficients $A_m(\lambda)$ and $B_m(\lambda)$ we impose the differentiability condition at $r=1$ on the function $\frac{1}{\sqrt{r}}\, f_{m,\lambda}(r)$ which is equivalent to the differentiability (at $r=1$) of $f_{m,\lambda}$. With the help of  \eqref{w-jy} we get
\begin{align} 
A_m(\lambda) & = \frac{\pi}{2}\, \Big( (v_m(\lambda,1)\, |\alpha+m| -v'_m(\lambda,1))\, Y_{|\alpha+m|}(\sqrt{\lambda}) -\sqrt{\lambda}\ v_m(\lambda,1)\,  Y_{|\alpha+m|+1}(\sqrt{\lambda}) \Big ) \label{am-eq} \\
%& \nonumber \\
B_m(\lambda) & = \frac{\pi}{2}\, \Big( (v'_m(\lambda,1) -|\alpha+m| \, v_m(\lambda,1))\, J_{|\alpha+m|}(\sqrt{\lambda}) +\sqrt{\lambda}\ v_m(\lambda,1)\, J_{|\alpha+m|+1}(\sqrt{\lambda}) \Big ) \label{bm-eq}.
\end{align} 
Here $v'_m(\lambda,1)$ and $u'_m(\lambda,1)$ denote the derivatives of $v_m$ and $u_m$ with respect to $r$. 
Similarly, the matching conditions for $\frac{1}{\sqrt{r}}\,  \phi_{m,\lambda}$ yield 
\begin{align} 
C_m(\lambda) &=  \frac{ \Gamma(\frac 12+|m|+\frac{m\alpha}{\kappa})}{\Gamma(1+2 |m|)} \ \Big( (u'_m(\lambda,1)-u_m(\lambda,1)\, |\alpha+m| )\, J_{|\alpha+m|}(\sqrt{\lambda}) +\sqrt{\lambda}\ u_m(\lambda,1)\,  J_{|\alpha+m|+1}(\sqrt{\lambda})   \nonumber \\
& \qquad\qquad  +i \, \big[ u'_m(\lambda,1)-(u_m(\lambda,1)\, |\alpha+m| )\, Y_{|\alpha+m|}(\sqrt{\lambda}) +\sqrt{\lambda}\ u_m(\lambda,1)\,  Y_{|\alpha+m|+1}(\sqrt{\lambda})\, \big ] \Big) \label{cm-eq}  \\
D_m(\lambda) &=  \frac{2\,  \Gamma(\frac 12+|m|+\frac{m\alpha}{\kappa})}{\pi\, \Gamma(1+2 |m|)} \ (i A_m(\lambda)-B_m(\lambda)),  \label{dm-eq}
\end{align}
where we have used the identity
\begin{equation} \label{w-xz}
v'_m(\lambda,1)\, u_m(\lambda,1)-  v_m(\lambda,1)\, u'_m(\lambda,1)  = \frac{ \Gamma(1+2 |m|)}{\Gamma(\frac 12+|m|+\frac{m\alpha}{\kappa})}\ ,
\end{equation}
see \cite[Eq.13.1.22]{as}.  From \eqref{w-jy} and \cite[Eq.13.1.22]{as} we then calculate the Wronskian
\begin{equation}  \label{wronsk} 
W_m(\lambda)= D_m(\lambda) \ \frac{\Gamma(1+2\, |m|)}{ \Gamma(\frac 12+|m|+\frac{m\alpha}{\kappa})} = \frac{2}{\pi} (B_m(\lambda)-i A_m(\lambda)).
\end{equation}
These formulas in combination with \eqref{rm} provide the expression for $\rr_0^m(\lambda;r,r')$ and consequently for $R_0(\lambda; x,y)$, via  \eqref{hk-gen} and \eqref{rtor}. To analyse the asymptotic behavior of $R_0(\alpha,\lambda; x,y)$ as $\lambda\to 0+$ we introduce the following shorthands:
$$
a_m = v_{m}(0, 1), \qquad a'_m = v_{m}'(0, 1), \qquad b_m =  u_{m}(0, 1), \qquad b'_m =  u'_{m}(0, 1)\, .
$$
Then,  in view of \eqref{nu-pos} and \eqref{am-eq}-\eqref{dm-eq} as $\lambda\to 0+$ we have:
\begin{align}
A_m(\lambda) & = \frac{\Gamma(|m+\alpha|)}{2}\, (a'_m +|m+\alpha|\, a_m)\, \big(\frac 12 \sqrt{\lambda}\big)^{-|\alpha+m|}\, \big(1+\mathcal{O}(\lambda)\big)  \nonumber\\
& \quad + \frac{i\, \pi\,  \cot(|m+\alpha|\, \pi)}{2\, \Gamma(|m+\alpha|+1)} (a'_m -|m+\alpha|\, a_m)\, \big(\frac 12 \sqrt{\lambda}\big)^{|\alpha+m|}\, \big(1+\mathcal{O}(\lambda)\big ) \label{am} \\
B_m(\lambda) &= \frac{\pi}{2\, \Gamma(|m+\alpha|+1)} (a'_m -|m+\alpha|\, a_m)\, \big(\frac 12 \sqrt{\lambda}\big)^{|\alpha+m|}\, \big(1+\mathcal{O}(\lambda)\big ) \label{bm},
\end{align}
where the error terms are uniform in $m$. Similarly we find
\begin{align}
C_m(\lambda) &=  \frac{\Gamma(\frac 12+m+|m|)\, (1+i \cot(|m+\alpha|\, \pi)) }{\Gamma(1+2 |m|)\, \Gamma(|m+\alpha|+1)} \, (b'_m-\, |m+\alpha|\, b_m)\,  \big(\frac 12 \sqrt{\lambda}\big)^{|\alpha+m|}(1+\mathcal{O}(\lambda)) \nonumber \\
& \quad -i\,  \frac{\Gamma(\frac 12 + m+ |m|)\, \Gamma(|m+\alpha|)}{\pi \Gamma(1+2 |m|)} \, (b'_m +|m+\alpha|\, b_m) \big(\frac 12 \sqrt{\lambda}\big)^{-|\alpha+m|}\, (1+\mathcal{O}(\lambda)) \label{cm} \end{align}
Hence from equations \eqref{hk-gen}, \eqref{rtor}, \eqref{rm} and \eqref{fm}-\eqref{wronsk}, after elementary but somewhat lengthly calculations, we obtain 
\begin{equation}  \label{G0}
\lim_{\lambda\to 0+}\, R_0(\lambda;x,y) \, =: G_0(x,y) \, = \, \sum_{m\in\Z} \, G_{m,0}(r,r') \ e^{im(\theta-\theta')} \, ,
\end{equation} 
where 
\begin{align}
G_{m,0}(r,r') 
& = \frac{\Gamma(\frac 12 +m+|m|)}{\Gamma(1+2 |m|)}\, v_{m}(0,r) \Big(u_m(0,r')- \frac{b'_m+|m+\alpha|\, b_m}{a'_m+|m+\alpha|\, a_m}\, v_m(0,r')\Big) &  r< r' \leq 1, \nonumber %\label{G01}
\\
& \nonumber \\
G_{m,0}(r,r') 
& = \frac{v_{m}(0,r)}{a'_m+|m+\alpha|\, a_m}\  (r')^{-|m+\alpha|} \, & r \leq 1 < r' , \label{G02}\\
& \nonumber \\
G_{m,0}(r,r') & = \frac{1}{2|m+\alpha|}\, \left[\Big(\frac{r}{r'}\Big)^{|m+\alpha|} - \frac{a'_m-|m+\alpha|\, a_m}{a'_m+|m+\alpha|\, a_m}\ (r r')^{-|m+\alpha|} \right]  &1 <r <r'   \nonumber%\label{G03}.
\end{align}

\smallskip
 
\noindent Note that $a'_m+|m+\alpha|\, a_m>0$ for all $m\in\Z$, see \eqref{m-infty}. Consider now the remainder term in \eqref{G0}. Since $\alpha>0$ by assumption, we have $k(\alpha) \leq 0$. To simplify the notation we will write
$$
k(\alpha) = k  , \qquad \mu(\alpha)=\mu=|k+\alpha|.
$$ 
Then, using again   \eqref{hk-gen}, \eqref{rtor}, \eqref{rm} and \eqref{fm}-\eqref{wronsk} we find that
\begin{equation} \label{G1}
\lim_{\lambda\to 0+}\,  \lambda^{-\mu}\, \big(R_0(\lambda; x,y) - G_0(x,y) \big) = G_1(x,y) = g_1(r,r') \ e^{i k (\theta-\theta')} ,
\end{equation} 
where
\begin{equation} \label{g1}
\begin{aligned}
g_1(r,s) & =  \frac{2\, \pi \, v_k(0, r)\,  v_k(0, s)\, }{ 4^\mu\, \Gamma^2(\mu) \, (a'_k +\mu\, a_k)^2}\ (i-\cot(\mu\, \pi)) \,  \qquad \qquad\qquad \qquad\qquad \qquad\quad  r< s \leq 1 ,\\
& \\
g_1(r,s) & =  \frac{ \pi\, v_k(0, r) \,(i-\cot(\mu\, \pi))}{\mu\, 4^\mu\, \Gamma^2(\mu)\, (a'_k +\mu\, a_k) }\ \Big(s^{\mu}- \frac{a'_k-\mu\, a_k}{a'_k+\mu\, a_k}\, s^{-\mu}\Big)  \qquad \qquad\qquad \qquad\ \  r\leq 1< s ,\\
& \\
g_1(r,s) & =  \frac{ \pi\,  (i-\cot(\mu\, \pi)) }{4^\mu\, 2\, \mu^2\, \Gamma^2(\mu)} \,  \Big(r^{\mu}- \frac{a'_k-\mu\, a_k}{a'_k+\mu\, a_k}\, r^{-\mu}\Big)\, \Big(s^{\mu}- \frac{a'_k-\mu\, a_k}{a'_k+\mu\, a_k}\, s^{-\mu}\Big) 
 \qquad\  1\leq r< s.
\end{aligned}
\end{equation}
This implies \eqref{r0}. 
\end{proof}

\noindent In the sequel we denote by $G_0$ and $G_1$ the operators on $L^2(\R^2)$ with kernels $G_0(x,y)$ and  $G_1(x,y)$. The following Lemma shows that these operators have the properties needed for the proof of Proposition \ref{prop-1}.

\begin{lemma} \label{lem-g0}
Let $\alpha\not\in\Z$ and let $ s>1$.  Then $\rho^{-s}\, G_j\, \rho^{-s}$ and $\rho^{-s}\, \nabla\, G_j\, \rho^{-s}$ with $j=0,1$ are compact operators from $L^2(\R^2)$ to $L^2(\R^2)$ and from $L^2(\R^2)$ to $L^2(\R^2,\C^2)$ respectively. 
\end{lemma}

\begin{proof} Consider first the operator $G_0$. We denote by $G_{m,0}$ the integral operator in $L^2(\R_+,r dr)$ with the kernel $ G_{m,0}(r,r')$. For $m=0$ it is easily seen from \eqref{G0}, using equations \eqref{mu} and \eqref{u-int}, that the kernels  $ G_{0,0}(r,r')$ and 
$\partial_{r'}G_{0,0}(r,r')$ generate Hilbert-Schmidt operators in $L^2(\R_+,r dr)$. Hence we may suppose in the rest of the proof that $m\neq 0$. 

\noindent To continue we note that with the help of equations \eqref{mu}, \eqref{mu-der}, \eqref{u-upperb} and \eqref{u'-upperb}  it is straightforward to verify that 
\begin{equation}  
|b'_m+|m+\alpha| \, b_m | \ \lesssim \ \frac{|2\alpha|^{-|m|}\, \Gamma(2|m|)}{\Gamma(\frac 12+2|m|)},   \quad  |a'_m+|m+\alpha| \, a_m|  \  \gtrsim  \ (1+| m|) \, |2\alpha|^{|m|} \label{m-infty} \qquad \forall\ m\in\Z.
\end{equation}
On the other hand, equations \eqref{vm}, \eqref{um} in combination with \eqref{mu} and \eqref{u-upperb} imply 
\begin{equation} \label{vu0-upperb}
|v_m(0,r) | \, \leq \, e^{\kappa r}\, (2\kappa r)^{|m|}, \qquad 
|u_m(0,r')| \, \leq \, C_\alpha \  e^{\kappa r'}\, (2\kappa r')^{-|m|}\, \Gamma(2 |m|),
\end{equation}
with a constant $C_\alpha$ independent of $m$.  From\eqref{G02} and \eqref{vu0-upperb} we thus obtain the following estimates:
\begin{equation} \label{gm0}
 | G_{m,0}(r,r')| \ \lesssim \ \frac{1}{|m+\alpha|} 
\left\{
\begin{array}{l@{\quad}l}
(r/r')^{|m|} &\quad r <r ' \leq 1  , \\
& \\
(r/r')^{|m+\alpha|}  &  \quad 1 \leq r < r'\\
& \\
r^{|m|}\, (r')^{-|m+\alpha|} 
 &\quad  r < 1  < r' .
\end{array}
\right.
\end{equation}
and 
\begin{equation} \label{gm0-grad}
\big | \, \frac{m}{r'}\, G_{m,0}(r,r')\big | + |\,  \partial_{r'} G_{m,0}(r,r')| \ \lesssim \ \frac{1}{r'}\ |m+\alpha|\   | G_{m,0}( r,r')| ,
\end{equation}
where the roles of $r$ and $r'$ have to be interchanged if $r' <r$.  We thus find that 
\begin{equation} \label{gm0-hs}
\|\, \rho^{-s}\,  G_{m,0} \, \rho^{-s}\|_{HS(\R_+, r dr)} \ \lesssim \  |m+\alpha|^{-3/2} \qquad \forall\ m\in\Z.
\end{equation}
Here $\| \cdot\|_{HS(\R_+, r dr)}$ denotes the Hilbert-Schimdt norm on $L^2(\R_+, r dr)$. This shows that $\rho^{-s}\,  G_{m,0} \, \rho^{-s}$ is compact on $L^2(\R_+, r dr)$ for every $m$ and that $\|\rho^{-s}\,  G_{m,0}\,  \rho^{-s}\big\|_{L^2(\R_+,r dr)}$ converges to zero as $|m|\to\infty$. 
Hence $\rho^{-s}\,  G_0\,  \rho^{-s}$ is compact on $L^2(\R^2)$ in view of \eqref{G0}. Moreover, from \eqref{gm0} and  \eqref{gm0-grad} we infer that 
\begin{equation*} 
\big\|\rho^{-s} \, \frac{m}{r'}\, G_{m,0} \,  \rho^{-s}\big\|_{HS(\R_+, r dr)}   + \|\rho^{-s}\,  \partial_{r'} G_{m,0}  \, \rho^{-s}\|_{HS(\R_+, r dr)} \ \lesssim \ |m+\alpha|^{-1/2}. 
\end{equation*}
This implies that the operators $\rho^{-s} \, \frac{m}{r'}\, G_{m,0} \,  \rho^{-s}$ and $\rho^{-s}\,  \partial_{r'} G_{m,0}  \, \rho^{-s}$ are compact on $L^2(\R_+, rdr)$ for every $m$, and that their operator norm tends to zero as $|m|\to\infty$. Since the integral kernel of the operator $ \nabla G_0$ is given by
$$
\nabla G_0  (r,r',\theta, \theta') =  \sum_{m\in\Z} \, e^{im(\theta-\theta')} \ \Big(\pd_{r'} G_{m,0}(r,r')\ , \, \frac{m}{r'}\, G_{m,0}(r,r') \Big),
$$
this proves the compactness of the operator $\rho^{-s}\, \nabla\, G_0\, \rho^{-s}$ on $L^2(\R^2)$. 

\smallskip

\noindent The analysis of the operators $\rho^{-s}\, G_1\, \rho^{-s}$ and $\rho^{-s}\, \nabla G_1\, \rho^{-s}$ is easier, since we have a contribution only from $m=k(\alpha)$. Indeed, from the explicit expression for $g_1(r,r')$ it follows that 
$$
\rho^{-s}(r')\, g_1(r,r')\, \rho^{-s}(r), \qquad  \rho^{-s}(r')\, \pd_{r'}  g_1(r,r')\, \rho^{-s}(r), \qquad  \rho^{-s}(r')\, \frac{1}{r'}\, g_1(r,r')\,  \rho^{-s}(r),
$$
are Hilbert-Schmidt kernels on $L^2(\R_+, r dr)$. 
\end{proof}

\noindent Next we are going to study the behavior of the remainder term  in \eqref{r0}. Let us denote by $G_2^+(\lambda)$ the integral operator in $L^2(\R^2)$ with kernel $G_2^+(\lambda;x,y)$ given by \eqref{r0} and \eqref{hk-gen}.

\begin{lemma} \label{r0-asymp1}
Let $\alpha\not\in\Z$ and assume that $ 3/2 < s <  3/2 +\mu(\alpha)$. Then, as $\lambda\to 0+$ we have $G_2^+(\lambda) = o(\lambda^{\mu(\alpha)})$ in $\B(s,-s )$. 
\end{lemma}

\begin{proof}
From Lemma \ref{lem-no-int} it follows that
\begin{equation} \label{norm-sup}
\| G_2^+(\lambda) \|_{\B(s,-s )} = \| \rho^{-s}\, G_2^+(\lambda) \, \rho^{-s}\,\|_{\B(L^2(\R^2))} = \sup_{m\in\Z} \| \rho^{-s}\, G_{m,2}^+(\lambda) \, \rho^{-s}\,\|_{\B(L^2(\R_+,r dr))} ,
\end{equation}
where $G_{m,2}^+(\lambda)$ is the integral operator with kernel 
\begin{equation} \label{gm2}
 G_{m,2}^+(\lambda, r,r') = R_0^m(\lambda, r,r') - G_{m,0}(r,r') - \lambda^{\mu(\alpha)}\, \delta_{m,k(\alpha)}\, g_1(r,r'),
\end{equation}	
and $\delta_{jk}$ denotes the Kronecker delta. Let us denote 
\begin{align*}
M^1_{(0,1)\times(0,1)} [\, G_{m,2}^+(\lambda)]  & =
\sup_{0\leq r\leq 1} \int_0^1 | \rho^{-s}(r) \, G_{m,2}^+(\lambda, r,r')\, \rho^{-s}(r') |\, r' \, dr'   \\
M^2_{(0,1)\times(0,1)} [\, G_{m,2}^+(\lambda)]  & = \sup_{0\leq r'\leq 1} \int_0^1 | \rho^{-s}(r)\, G_{m,2}^+(\lambda, r,r')\, \rho^{-s}(r') |\, r\, dr
\end{align*}
We will estimate the norm of $\rho^{-s}\, G_{m,2}^+(\lambda) \, \rho^{-s}$ in $L^2(\R_+,r dr)$ as follows:
\begin{align} 
\| \rho^{-s}\, G_{m,2}^+(\lambda) \, \rho^{-s}\,\|^2_{\B(L^2(\R_+,r dr))} & \ \leq \ M^1_{(0,1)\times(0,1)} [G_{m,2}^+(\lambda)]\ M^2_{(0,1)\times(0,1)} [G_{m,2}^+(\lambda)] \label{0-1} \\
& \ + 2 \int_0^1\rho^{-2s}(r)   \int_1^\infty \, | G_{m,2}^+(\lambda, r,r')|^2\, \rho^{-2s}(r') \,  r'\, dr' \, r \, dr  \nonumber\\
& \ + \int_1^\infty \int_1^\infty \rho^{-2s}(r) \, | G_{m,2}^+(\lambda, r,r')|^2\, \rho^{-2s}(r')\,  \, r  r'\, dr \,  dr \nonumber
\end{align}
Here we use the Schur-Holmgren bound for the part of the operator relative to the region $(0,1)\times(0,1)$ and the Hilbert-Schmidt norm on the rest of $\R_+\times \R_+$. 

\smallskip

\noindent Let us first estimate the last term on the right hand side of \eqref{0-1}. We recall the formula for the resolvent kernel on $(1,\infty)\times(1,\infty)$:
\begin{align} 
R_0^m(\lambda, r,r') & = \frac{1}{W_m(\lambda)} \Big[ A_m(\lambda)\,  J_{|m+\alpha|}(\sqrt{\lambda}\, r)\, J_{|m+\alpha|}(\sqrt{\lambda}\, r') + i\, A_m(\lambda)\,  J_{|m+\alpha|}(\sqrt{\lambda}\, r)\, Y_{|m+\alpha|}(\sqrt{\lambda}\, r') \nonumber  \\
& \ \ +B_m(\lambda)\,  Y_{|m+\alpha|}(\sqrt{\lambda}\, r)\, J_{|m+\alpha|}(\sqrt{\lambda}\, r') +i\, B_m(\lambda)\,  Y_{|m+\alpha|}(\sqrt{\lambda}\, r)\, Y_{|m+\alpha|}(\sqrt{\lambda}\, r') \Big ],   \label{rm-2}
\end{align}
see equations \eqref{rm}, \eqref{fm} and \eqref{fim}. 

We now use identity \eqref{jy}, keeping in mind that $\alpha\not\in\Z$ to write $Y_{|m+\alpha|}$ in terms  of $J_{|m+\alpha|}$ and $J_{-|m+\alpha|}$, and estimate each term in the above sum separately. To this end we are going to use integral operators $T^m_{\pm }$ on $L^2((1,\infty), r dr)$ 
with kernels
\begin{align}
T^m_{+}(r,r') &= \frac{4^{-|m+\alpha|}\, (r r')^{|m+\alpha|}}{\Gamma^2(|m+\alpha|+1)} \, ,  \qquad \qquad\quad   T^m_{-}(r,r') = \left( \frac{r r'}{4}\right)^{-|m+\alpha|}\, \Gamma^2(|m+\alpha|) \ \frac{\sin^2(|m+\alpha| \pi)}{\pi^2} \nonumber
\\
T^m_{\pm}(r,r') & = - (r/r')^{|m+\alpha|}\ \frac{\sin(|m+\alpha|\, \pi)}{\pi\, |m+\alpha|}\, ,  \quad T^m_{\mp}(r,r')  = T^m_{\pm}(r',r). \label{eq-tm}
\end{align}
In view of Lemma \ref{lem-jj}, equation \eqref{eq-JJ-1} of Lemma \ref{lem-int-JJ}, and equation \eqref{nalpha} we then have 
\begin{equation} \label{asymp-jj-1}
J_{|m+\alpha|} (\sqrt{\lambda}\, r)\,  J_{|m+\alpha|} (\sqrt{\lambda}\, r') = 
\left\{
\begin{array}{l@{\quad}l}
\lambda^{\mu(\alpha)}\, T^{k(\alpha)}_+(r,r')(1+\mathcal{O}(\lambda^\eps) )&\quad m = k(\alpha) , \\
(1+m^2)^{-1}\, \mathcal{O}(\lambda^{\frac 12+\eps}) & \quad  m\neq k(\alpha)
\end{array}
\right.
\end{equation}
as $\lambda\to 0+$ with respect to the Hilbert-Schmidt norm on $L^2((1,\infty), r dr)$ and with the error terms uniform in $m$. Similarly, using equations  \eqref{eq-JJ-2} and  \eqref{eq-JJ-3} of Lemma \ref{lem-int-JJ}, we find that 
\begin{align*}
\lambda^{|m+\alpha|}\, J_{-|m+\alpha|} (\sqrt{\lambda}\, r)\,  J_{-|m+\alpha|} (\sqrt{\lambda}\, r') &= T^m_-(r,r') (1+\mathcal{O}(\lambda)) %\label{mm}
\\
J_{ |m+\alpha|} (\sqrt{\lambda}\, r)\,  J_{- |m+\alpha|} (\sqrt{\lambda}\, r') & =  T^m_\pm(r,r') + (1+|m|)^{-1}\, \mathcal{O}(\lambda^{\frac 12+\eps}) %\label{pm}
\\ 
J_{ -|m+\alpha|} (\sqrt{\lambda}\, r)\,  J_{ |m+\alpha|} (\sqrt{\lambda}\, r') & =  T^m_\mp(r,r') +  (1+|m|)^{-1}\, \mathcal{O}(\lambda^{\frac 12+\eps}) 
%\label{mp}
\end{align*}
as $\lambda\to 0+$  with respect to the Hilbert-Schmidt norm on $L^2((1,\infty), r dr)$ and error terms uniform in $m$. On the other hand, from equations \eqref{am} and \eqref{bm} it follows that 
\begin{align}
\frac{B_m(\lambda)}{W_m(\lambda)} &= \frac{i\, \pi^2 \, 4^{- |m+\alpha|}}{2\, |m+\alpha|\, \Gamma^2( |m+\alpha|)}\ \left(\frac{a_m'- |m+\alpha|\, a_m}{a_m'+ |m+\alpha|\, a_m}\right)\, \lambda^{ |m+\alpha|} (1+ \mathcal{O}(\lambda))  \nonumber \\
& +   \frac{ \pi^3 \, (1+i \cot(|m+\alpha|\, \pi))}{4^{2|m+\alpha|}\,  2\, |m+\alpha|^2\, \Gamma^4( |m+\alpha|)}\ \left(\frac{a_m'- |m+\alpha|\, a_m}{a_m'+ |m+\alpha|\, a_m}\right)^2\, \lambda^{2 |m+\alpha|} (1+ \mathcal{O}(\lambda)) \label{bmw}
\end{align}
and 
\begin{align}
\frac{A_m(\lambda)}{W_m(\lambda)} &= \frac{\pi}{2}\,  \left(i +  \frac{\pi \,  \cot(|m+\alpha|\, \pi)\,  \lambda^{|m+\alpha|}}{4^{|m+\alpha|}\,  |m+\alpha|\, \Gamma^2( |m+\alpha|)}\ \left(\frac{a_m'- |m+\alpha|\, a_m}{a_m'+ |m+\alpha|\, a_m}\right) \right)\, \left(1+\mathcal{O}(\lambda)\right),  \label{amw}
\end{align}
with the respective error terms uniform in $m$. We now write \eqref{rm-2} with $Y_{|m+\alpha|}$ replaced by the right hand side of \eqref{jy} and insert in the resulting equation for $R_0^m(\lambda, r,r') $ the asymptotic expansions \eqref{eq-tm}-\eqref{amw}. With the help of \eqref{m-infty} this yields 
\begin{equation} \label{hs-bessel}
 \int_1^\infty \int_1^\infty \rho^{-2s}(r) \, | G_{m,2}^+(\lambda, r,r')|^2\, \rho^{-2s}(r')\,  \, r  r'\, dr \,  dr'  \ \lesssim\  \frac{\lambda^{2\mu(\alpha)}}{1+|m|}   
\end{equation}
for $\lambda$ small enough. 

\smallskip

\noindent Next we are going to estimate the first term on the right hand side of \eqref{0-1}. In view of \eqref{rm}, \eqref{fm} and \eqref{fim} we have 
\begin{align} \label{rm-3}
R_0^m(\lambda, r,r') & = \frac{v_m(\lambda,r)}{W_m(\lambda)} \ \left(  C_m(\lambda)\, v_m(\lambda,r')+D_m(\lambda)\, u_m(\lambda, r') \right), \qquad r\leq r'\leq1.
\end{align}
As a first step we will  show that the second term on the right hand side of the above equation is differentiable with respect to $\lambda$ in the norm given by the right hand side of \eqref{0-1} and that the derivative is uniformly bounded in $m$.  By \eqref{wronsk} 
$$
\frac{D_m(\lambda)}{W_m(\lambda)} = \frac{\Gamma(\frac 12+|m|+\frac{m\alpha}{\kappa})}{\Gamma(1+2|m|)}.
$$
Therefore keeping in mind the definition of $u_m$, see equations \eqref{um} and \eqref{kappa}, and  using Lemmata \ref{lem-U} and \ref{ua'} we find that for $\lambda$ small enough it holds
\begin{equation*}
\Big|\, \pd_\lambda \left(\frac{D_m(\lambda)}{W_m(\lambda)}\ u_m(\lambda, r')\right) \, \Big| \ \lesssim \ (2\kappa\, r')^{-|m|} + \frac{ (4 \kappa\, r')^{|m|}}{\Gamma(2 |m|)}. 
\end{equation*}
This in combination with \eqref{vm} and Lemma \ref{lem-M} gives 
\begin{equation} \label{vdu-upperb}
\Big |\, v_m(\lambda,r) \ \pd_\lambda \Big(\frac{D_m(\lambda)}{W_m(\lambda)}\ u_m(\lambda, r')\Big) \, \Big| \ \lesssim \  \left(\frac{r}{r'}\right)^{|m|} + \frac{ (8 \kappa)^{|m|}\, (r r')^{|m|}}{\Gamma(2 |m|)}. 
\end{equation}
On the other hand, combining Lemma \ref{lem-M} with Lemma \ref{lem-der-m} we obtain that for $\lambda$ small enough
\begin{equation} \label{der-v-lam}
| \pd_\lambda v_m(\lambda, r) | \ \lesssim \ |m|\ (2\kappa\, r)^{|m|}\,  \qquad \forall\ r\in(0,1).
\end{equation}
Since 
$$
\Big |\frac{D_m(\lambda)}{W_m(\lambda)} \ u_m(\lambda, r') \Big |= \Big |  \frac{\Gamma(\frac 12+|m|+\frac{m\alpha}{\kappa})}{\Gamma(1+2|m|)}\ u_m(\lambda,r')\Big | \, \lesssim\, \frac{(2\kappa\, r')^{-|m|}}{|m|} +  \frac{ (4 \kappa\, r')^{|m|}}{\Gamma(1+2 |m|)},
$$
in view of Lemma \ref{lem-U}, by putting the above estimates together we arrive at 
$$
\Big | \, \pd_\lambda \Big (\frac{D_m(\lambda)}{W_m(\lambda)} \ u_m(\lambda, r') \ v_m(\lambda, r) \Big) \Big| \ \lesssim \ \left(\frac{r}{r'}\right)^{|m|} + (r r')^{|m|}\ , \qquad 0 < r\leq r' \leq 1,
$$
where we have used the fact that $(8 \kappa)^{|m|}/\Gamma(2 |m|)$ is bounded in $m\in\Z\setminus\{0\}$. As usual, $r$ and $r'$ have to be interchanged if $r' < r$. Since 
$$
\sup_{0< r < 1}\int_r^1 \big | \left(\frac{r}{r'}\right)^{|m|} + (r r')^{|m|}  \big |\, r' \, dr'  \, +\,  \sup_{0< r' < 1}\int_0^{r'} \big | \left(\frac{r}{r'}\right)^{|m|} + (r r')^{|m|}  \big |\, r \, dr \ \lesssim \ \frac{1}{1+|m|}
$$
we find that 
\begin{equation} \label{M-uv}
M^j_{(0,1)\times (0,1)}\left [\frac{D_m(\lambda)}{W_m(\lambda)} \ u_m(\lambda, \cdot) \ v_m(\lambda, \cdot)  - \frac{D_m(0)}{W_m(0)}\, u_m(0, \cdot) \ v_m(0, \cdot) \right] \lesssim \, \frac{\lambda}{1+|m|} , \qquad j=1,2
\end{equation}
holds for all $m\in\Z$ and $\lambda$ small enough. Here $\frac{D_m(0)}{W_m(0)}$ is to be understood as the limiting value of $\frac{D_m(\lambda)}{W_m(\lambda)}$. 
Consider now the first term on the right hand side of \eqref{rm-3}. 
From \eqref{der-v-lam} and Lemma \ref{lem-M} we obtain the bound 
\begin{equation} \label{der-vv-lam}
M^j_{(0,1)\times (0,1)}\left [\ \pd_\lambda \big( (v_m(\lambda, r)\,  v_m(\lambda, r') \big)\right] \ \lesssim\ (4\kappa^2)^{|m|}\ \leq \ (4\alpha^2)^{|m|}
\end{equation}
We are going to combine \eqref{der-vv-lam} with the asymptotic expansion of $\frac{C_m(\lambda)}{W_m(\lambda)}$:
\begin{align*}
\frac{C_m(\lambda)}{W_m(\lambda)} &= -\frac{\Gamma(\frac 12 +m+|m|)}{\Gamma(1+2 |m|)}\, \frac{b'_m+|m+\alpha|\, b_m}{a'_m+|m+\alpha|\, a_m}\ (1+\mathcal{O}(\lambda)) \\
& \quad\ +  \lambda^{|m+\alpha|}\  \frac{2 \pi\, (i-\cot(|m+\alpha|\, \pi))}{\Gamma^2(|m+\alpha|)\, (a'_m+|m+\alpha|\, a_m)^2} 
 \ (1+\mathcal{O}(\lambda)+\mathcal{O}(\lambda^{|m+\alpha|}))
\end{align*}
as $\lambda\to 0+$ with the error terms uniform in $m$, see \eqref{wronsk} and \eqref{cm}. In the calculation of the coefficient of $\lambda^{|m+\alpha|}$ in  the above equation we have used the fact that $a'_m b_m-a_m b'_m= \Gamma(1+|m|)/\Gamma(\frac 12+m+|m|)$, cf. \eqref{w-xz}.
Taking into account \eqref{m-infty} we conclude that for $j=1,2$ and $\lambda$ small enough 
$$
M^j_{(0,1)\times (0,1)} \left[ \frac{C_m(\lambda)}{W_m(\lambda)} \, v_m(\lambda, r)\,  v_m(\lambda, r') - 
\frac{C_m(0)}{W_m(0)} \, v_m(0, r)\,  v_m(0, r') \right] \, \lesssim\, \frac{\lambda}{1+|m|}
$$
when $m\neq k(\alpha)$, and
$$
M^j_{(0,1)\times (0,1)} \left[ \frac{C_m(\lambda)}{W_m(\lambda)} \, v_m(\lambda, r)\,  v_m(\lambda, r') - 
\frac{C_m(0)}{W_m(0)} \, v_m(0, r)\,  v_m(0, r') - \lambda^{\mu(\alpha)}\, g_1(r,r') \right] \, \lesssim\, \frac{\lambda^{2\mu(\alpha)}}{1+|m|} , 
$$
when $m=k(\alpha)$. This implies that for $\lambda$ small enough
\begin{equation}  \label{M-uv2}
M^j_{(0,1)\times (0,1)} [G_{m,2}^+(\lambda)] \ \lesssim\  \frac{\lambda^{2\mu(\alpha)}}{1+|m|} , \qquad j=1,2,
\end{equation}

\smallskip

\noindent It remains to bound the cross term on the right hand side of  \eqref{0-1}. By \eqref{jy} for $r\leq 1 < r'$ it holds
\begin{align} 
& R_0^m(\lambda, r,r')  = \frac{v_m(\lambda,r)}{W_m(\lambda)} \ \left(  J_{|m+\alpha|}(\sqrt{\lambda}\, r') + i\, Y_{|m+\alpha|} (\sqrt{\lambda}\, r') \right)=\nonumber  \\ 
& \qquad = \frac{v_m(\lambda,r)}{W_m(\lambda)} \ \left(  J_{|m+\alpha|}(\sqrt{\lambda}\, r')\, (1+i\cot( |m+\alpha|\, \pi) ) -  \frac{i}{\sin(|m+\alpha|\, \pi)} \, J_{-|m+\alpha|} (\sqrt{\lambda}\, r') \right), \label{rm-4}
\end{align}
We now insert the asymptotic expansion of the inverse Wronskian
\begin{align*}
\frac{1}{W_m(\lambda)}  & =  \frac{i\, \pi\ \lambda^{\frac{|m+\alpha|}{2}}}{\Gamma(|m+\alpha|)\, (a'_m+|m+\alpha|\, a_m)} \ (1+\mathcal{O}(\lambda))  \\
&  +\  \lambda^{|m+\alpha|}\  \frac{  \pi^2 (a'_m-|m+\alpha|\, a_m)\, (1+i \cot(|m+\alpha|\, \pi))}{|m+\alpha|\, \Gamma^3(|m+\alpha|)\, (a'_m+|m+\alpha|\, a_m)^2} \ (1+\mathcal{O}(\lambda) + \mathcal{O}(\lambda^{2|m+\alpha|}) )
\end{align*}
into \eqref{rm-4}. With the help of the integral estimates on the derivative of $ \lambda^{|m+\alpha|/2}\, J_{-|m+\alpha|} (\sqrt{\lambda}\ \cdot)$ with respect to $\lambda$, see inequality \eqref{mix-1} of Lemma \ref{lem-mixed},  and equations \eqref{m-infty}, \eqref{der-v-lam}, \eqref{nu-pos} we then obtain
\begin{align*}
\frac{-v_m(\lambda,r)}{\sin(|m+\alpha|\, \pi)\, W_m(\lambda)}  \, J_{-|m+\alpha|}(\sqrt{\lambda}\, r') &  =
\frac{-v_m(0,r)\, \lambda^{-\frac{|m+\alpha|}{2}}}{\sin(|m+\alpha|\, \pi)\, W_m(\lambda)}\, \lambda^{\frac{|m+\alpha|}{2}}\, J_{-|m+\alpha|}(\sqrt{\lambda}\, r')  \\ &
 =  G_{m,0}(r,r') + \frac{\mathcal{O}(\lambda)}{1+|m|}\ , \qquad \lambda\to 0+
\end{align*}
where the convergence is taken with respect to the Hilbert-Schmidt norm on $(1,\infty)\times(1,\infty)$. 
Similarly, using inequality \eqref{mix-2} of Lemma \ref{lem-mixed} and the above expansion of $W_m(\lambda)$ together with equations \eqref{m-infty}, \eqref{der-v-lam}, \eqref{nu-pos}  we get
\begin{align*}
\frac{v_m(\lambda,r)}{W_m(\lambda)} \ J_{|m+\alpha|}(\sqrt{\lambda}\, r') (1+i\cot(|m+\alpha|\, \pi)) & = g_1(r,r')\, \delta_{m,k(\alpha)} +  \frac{\mathcal{O}(\lambda^{\frac{|m+\alpha|+1}{2})})}{1+|m|}\ , \qquad \lambda\to 0+
\end{align*} 
with respect to the Hilbert-Schmidt norm on $(0,1)\times (1,\infty)$. In view of \eqref{rm-4} this implies that
\begin{equation} \label{hs-vjy}
\int_0^1\rho^{-2s}(r)   \int_1^\infty \, | G_{m,2}^+(\lambda, r,r')|^2\, \rho^{-2s}(r') \,  r'\, dr' \, r \, dr =  \frac{o(\lambda^{2\mu(\alpha)})}{1+m^2}\ , \qquad \lambda\to 0+.
\end{equation} 
By inserting \eqref{hs-bessel}, \eqref{M-uv2} and \eqref{hs-vjy} into \eqref{0-1} we arrive at 
\begin{equation} \label{aux-rest1}
\sup_{m\neq k(\alpha)} \| \rho^{-s}\, G_{m,2}^+(\lambda) \, \rho^{-s}\,\|^2_{\B(L^2(\R_+,r dr))} \, =\,  \mathcal{O}(\lambda)
\end{equation}
and
\begin{equation} \label{aux-rest2}
 \| \rho^{-s}\, G_{k(\alpha),2}(\lambda) \, \rho^{-s}\,\|^2_{\B(L^2(\R_+,r dr))} \, =\,  o(\sqrt{\lambda})
\end{equation}
The claim thus follows in view of  \eqref{norm-sup}. 
\end{proof}

\noindent Finally, we have to estimate also the operator $\nabla\, G_2^+(\lambda)$ in $\B(s,-s )$ for $\lambda$ small enough.

\begin{lemma} \label{r0-asymp2}
Let $\alpha\not\in\Z$.
Assume that $ 3/2 < s <  3/2 +\mu(\alpha)$. Then for $\lambda\to 0+$ we have $ |\nabla\, G_2^+(\lambda)| = o(\lambda^{\mu(\alpha)})$ in $\B(s,-s )$. 
\end{lemma}

\begin{proof}
The integral kernel of $\nabla\, G_2^+(\lambda)$ for $r \leq r'$ reads as 
 $$
\nabla G_2^+  (\lambda,r,r', \theta,\theta') =  \sum_{m\in\Z} \, e^{im(\theta-\theta')} \ \Big(\pd_{r'} G_{m,2}^+(r,r')\ , \, \frac{m}{r'}\, G_{m,2}^+(r,r') \Big).
$$
Hence the claim will follow if we show that 
\begin{equation} \label{eq-enough}
\sup_{m\in\Z} \|\, \rho^{-s}\, \widetilde G_{m,2}^+(\lambda) \, \rho^{-s}\,\|_{\B(L^2(\R_+,r dr))}  \, = \, o(\lambda^{\mu(\alpha)}), \qquad \lambda\to 0+,
\end{equation}
where $\widetilde G_{m,2}^+(\lambda)$ is the integral operator with the kernel given by
\begin{equation} \label{gm-tilde}
\big(\widetilde G^+_{m,2}(\lambda,r,r') \big)^2=  \big(\pd_{r'} G_{m,2}^+(\lambda, r,r')\big)^2 + \frac{m^2}{r'^2}\, (G^+_{m,2}(\lambda, r,r'))^2.
\end{equation}
As usual,  $r$ and $r'$ in the above formula have to be interchanged when $r'<r$. We use again the upper bound
\begin{align} 
\| \rho^{-s}\, \widetilde G_{m,2}^+(\lambda) \, \rho^{-s}\,\|^2_{\B(L^2(\R_+,r dr))} & \ \leq \ M^1_{(0,1)\times(0,1)} [G_{m,2}^+(\lambda)]\ M^2_{(0,1)\times(0,1)} [\widetilde G_{m,2}^+(\lambda)] \label{0-2} \\
& \ + 2 \int_0^1\rho^{-2s}(r)   \int_1^\infty \, | \widetilde G_{m,2}^+(\lambda, r,r')|^2\, \rho^{-2s}(r') \,  r'\, dr' \, r \, dr  \nonumber\\
& \ + \int_1^\infty \int_1^\infty \rho^{-2s}(r) \, | \widetilde G_{m,2}^+(\lambda, r,r')|^2\, \rho^{-2s}(r')\,  \, r  r'\, dr \,  dr. \nonumber
\end{align}
Let us consider the last term on the right hand side. Using inequality \eqref{eq-JJ-4} of Lemma \ref{lem-int-JJ'}  instead of \eqref{eq-JJ-1} we mimic the proof of the upper bound \eqref{hs-bessel} and obtain
\begin{equation*} 
 \int_1^\infty \int_1^\infty \rho^{-2s}(r) \, \Big [\big(\pd_{r'} G_{m,2}^+(r,r')\big)^2 + \frac{(m+\alpha)^2}{r'^2}\, G^2_{m,2}(r,r') \Big] \, \rho^{-2s}(r')\,  \, r  r'\, dr \,  dr'  \, \lesssim\  \frac{\lambda^{2\mu(\alpha)}}{1+|m|}    
\end{equation*}
for $\lambda$ small enough. This shows that 
\begin{equation} \label{hs-bessel-2}
\sup_{m\in\Z}\,  \int_1^\infty \int_1^\infty \rho^{-2s}(r) \, | \widetilde G_{m,2}^+(r,r')|^2 \, \rho^{-2s}(r')\,  \, r  r'\, dr \,  dr'  \, \lesssim\ \lambda^{2\mu(\alpha)}   
\end{equation}
as $\lambda\to 0+$.
As for the remaining terms on the right hand side of \eqref{0-2} we note that in view of \eqref{gm-tilde} 
$$
M^j_{(0,1)\times(0,1)} [\widetilde G_{m,2}^+(\lambda)] \  \lesssim \ (1+|m|)\ M^j_{(0,1)\times(0,1)} [G_{m,2}^+(\lambda)], \qquad j=1,2, 
$$
and
\begin{align*}
&  \int_0^1\rho^{-2s}(r)   \int_1^\infty \, | \widetilde G_{m,2}^+(\lambda, r,r')|^2\, \rho^{-2s}(r') \,  r'\, dr' \, r \, dr \,  \lesssim\,  \\
 &\qquad \qquad\quad\qquad\qquad  \lesssim\ (1+m^2)\,   \int_0^1\rho^{-2s}(r)   \int_1^\infty \, | G_{m,2}^+(\lambda, r,r')|^2\, \rho^{-2s}(r') \,  r'\, dr' \, r \, dr. 
\end{align*}
A combination of these estimates with \eqref{M-uv2} and \eqref{hs-vjy} gives
$$
\sup_{m\in\Z}\, M^j_{(0,1)\times (0,1)} [\widetilde G_{m,2}^+(\lambda)] \  =\  \mathcal{O}(\lambda^{2\mu(\alpha)}), \qquad j=1,2,
$$
and
$$
\sup_{m\in\Z}\,  \int_0^1\rho^{-2s}(r)   \int_1^\infty \, | \widetilde G_{m,2}^+(\lambda, r,r')|^2\, \rho^{-2s}(r') \,  r'\, dr' \, r \, dr = o(\lambda^{2\mu(\alpha)})
$$
as $\lambda\to 0+$. In view of \eqref{0-2} and \eqref{hs-bessel-2} this proves \eqref{eq-enough} and the claim follows.
\end{proof}

%%%%%%%%%%%%%%%%%%%%%%%%%%%%%%%%%%%%%%%%%%%%%%%%%%%%%
\subsection{The case $\lambda<0$} 
\label{sec-l-neg}
\noindent For negative values of $\lambda$ we repeat the same procedure as in the case $\lambda>0$. The calculations are identical to those made in section \ref{sec-l-pos}. We therefore omit some details. Recall that 
$k(\alpha)$ is defined \eqref{nalpha}. 

\begin{lemma} \label{lem-neg1}
Let $\alpha\not\in\Z$. Then for any $x,y\in\R^2$ and $|\lambda|$ small enough we have
\begin{equation*}
R_0(\lambda; x,y) = G_0(x,y) + \lambda^{\mu(\alpha)}\, G_1(x,y) +G_2^-(\lambda;x,y) ,
\end{equation*}
where $G_0(x,y)$ and $G_1(x,y)$ are given by \eqref{G0} and \eqref{G1}, and  $G_2^-(\lambda;x,y))= o(|\lambda|^{\mu(\alpha)})$ as  $\lambda\to 0-$.
\end{lemma}

\begin{proof}
We calculate the integral kernel of $R_0(\lambda)$ in the same way as above. Inside the unit disc the generalized eigenvalue equation \eqref{eq-fg} has the same solutions, i.e. $u_m$ and $v_m$, as for $\lambda>0$. Outside of the unit disc we have to replace the Bessel functions $J_\nu$ and $Y_\nu$ by a suitable linear combination of the modified Bessel functions $I_\nu$ and $K_\nu$, see Appendix \ref{app} for details. We then find that $R_0(\lambda; x,y)$ is given by \eqref{hk-gen} with $R_0^m(\lambda; r,r')$ replaced by 
\begin{equation} \label{rm-tilde} 
\widetilde R^m_0(\lambda; r,r') = \big( \widetilde W_m(\lambda)\big)^{-1}
\left\{
\begin{array}{l@{\quad}l}
\tilde f_{m,\lambda}(r)\, \tilde \phi_{m,\lambda}(r'), &\quad r \leq r'  , \\
 & \\
\tilde f_{m,\lambda}(r')\, \tilde \phi_{m,\lambda}(r), &\quad  r' < r ,
\end{array}
\right.
\end{equation}
where 
\begin{align*}
\tilde f_{m,\lambda}(r) &=  v_m(\lambda,r), \qquad \tilde \phi_{m,\lambda}(r)  = \widetilde C_m(\lambda)\, v_m(\lambda,r)+\widetilde D_m(\lambda)\, u_m(\lambda, r)  \quad  & r \leq 1  \\
\tilde f_{m,\lambda}(r) & =  \widetilde A_m(\lambda)\,  I_{|\alpha+m|}(\sqrt{|\lambda|}\, r) +\widetilde B_m(\lambda)\, K_{|\alpha+m|}(\sqrt{|\lambda|}\, r) , \qquad 
\tilde \phi_{m,\lambda}(r)  = K_{|m+\alpha|}(\sqrt{|\lambda|}\, r),  \quad  & 1  < r.
\end{align*}
and $\widetilde W_m(\lambda) =  \tilde \phi_{m,\lambda}\, \tilde f'_{m,\lambda} -\tilde \phi'_{m,\lambda} \, \tilde f_{m,\lambda}$ is the corresponding Wronskian.   From the matching conditions at $r=1$ and the properties of modified Bessel functions, see equations \eqref{i-der}, \eqref{k-der},  we then find
\begin{align*} 
\widetilde A_m(\lambda) & = (v'_m(\lambda,1) - |\alpha+m| \, v_m(\lambda,1) )\, K_{|\alpha+m|}(\sqrt{|\lambda|}) +\sqrt{|\lambda|}\ v_m(\lambda,1)\,  K_{|\alpha+m|+1}(\sqrt{|\lambda|})   \\
%& \nonumber \\
\widetilde B_m(\lambda) & =(|\alpha+m| \, v_m(\lambda,1)-v'_m(\lambda,1) )\, I_{|\alpha+m|}(\sqrt{|\lambda|}) +\sqrt{|\lambda|}\ v_m(\lambda,1)\, I_{|\alpha+m|+1}(\sqrt{|\lambda|})  
\end{align*} 
and 
\begin{align*} 
\widetilde C_m(\lambda) &=  \frac{ \Gamma(\frac 12+|m|+\frac{m\alpha}{\kappa})}{\Gamma(1+2 |m|)} \ \Big( (\, |\alpha+m|\, u_m(\lambda,1) - u'_m(\lambda,1) )\, I_{|\alpha+m|}(\sqrt{|\lambda|}) +\sqrt{|\lambda|}\ u_m(\lambda,1)\,  I_{|\alpha+m|+1}(\sqrt{|\lambda|})   \nonumber \\
& \quad  +i \, \big[ (  |\alpha+m|\, u_m(\lambda,1)- u'_m(\lambda,1) )\, K_{|\alpha+m|}(\sqrt{|\lambda|}) +\sqrt{|\lambda|}\ u_m(\lambda,1)\,  K_{|\alpha+m|+1}(\sqrt{|\lambda|})\, \big ] \Big)  \\
\widetilde D_m(\lambda) &=  \frac{ \Gamma(\frac 12+|m|+\frac{m\alpha}{\kappa})}{ \Gamma(1+2 |m|)} \  \widetilde A_m(\lambda). 
\end{align*}
Moreover, with the help of \eqref{w-ik} we obtain 
\begin{equation} \label{w-tilde}
\widetilde W_m(\lambda)=\widetilde A_m(\lambda).
\end{equation}
Using the behavior of functions $I_\nu$ and $K_\nu$ for small arguments, see equation \eqref{ik-z0}, we then get the asymptotic expansions 
\begin{align*}
\widetilde A_m(\lambda) & = \frac{\Gamma(|m+\alpha|)}{2}\, (a'_m +|m+\alpha|\, a_m)\, \big(\frac 12 \sqrt{|\lambda|}\big)^{-|\alpha+m|}\, \big(1+\mathcal{O}(\lambda)\big)  \nonumber\\
& \qquad\qquad  - \frac{a'_m -|m+\alpha|\, a_m}{\sin(|m+\alpha|\, \pi)\, \Gamma(|m+\alpha|+1)} \, \big(\frac 12 \sqrt{|\lambda|}\big)^{|\alpha+m|}\, \big(1+\mathcal{O}(\lambda)\big ) , \\
\widetilde B_m(\lambda) &= \frac{a'_m -|m+\alpha|\, a_m}{ \Gamma(|m+\alpha|+1)} \, \big(\frac 12 \sqrt{|\lambda|}\big)^{|\alpha+m|}\, \big(1+\mathcal{O}(\lambda)\big ) 
\end{align*}
where the error terms are uniform in $m$. Similarly we find that the expansion of $\widetilde C_m(\lambda)$ is given by the right hand side of \eqref{cm} with the factor $i-\cot(\mu(\alpha) \pi)$ replaced by $-\frac{1}{\sin(\mu(\alpha) \pi)}$. Hence when we insert the above equations into \eqref{rm-tilde} and use \eqref{ik-z0}, after a bit lengthly calculations, we arrive at
\begin{equation} \label{rm-tilde1}
\widetilde R^m_0(\lambda; r,r') = G_{m,0}(r,r') - \delta_{m, k(\alpha)} \, \frac{|\lambda|^{\mu(\alpha)}\,  g_1(r,r')}{\sin(\mu(\alpha) \pi) \, (i-\cot(\mu(\alpha) \pi))} + o(|\lambda|^{\mu(\alpha)} ) \qquad \lambda\to 0- .
\end{equation}
Recall that $g_1(\cdot,  \cdot)$ is defined by \eqref{g1}. Since $|\lambda|^{\mu(\alpha)} = \lambda^{\mu(\alpha)} \, (\cos(\mu(\alpha) \pi)-i \sin(\mu(\alpha) \pi)$, this implies 
$$
R_0(\lambda;x,y) = G_0(x,y) +  \lambda^{\mu(\alpha)} \, G_1(x,y) + o(|\lambda|^{\mu(\alpha)} ) \qquad \lambda\to 0- , \quad x,y\in\R^2.
$$
\end{proof}

\noindent As in the case $\lambda>0$ we need an estimate on $G_2^-(\lambda)$ and $\nabla G_2^-(\lambda) $ in 
${\B(s,-s )}$. 

\begin{lemma} \label{lem-neg2}
Let $G_2^-(\lambda)$ be the integral operator with the kernel $G_2^-(\lambda; x,y)$ defined in Lemma \ref{lem-neg1}. Assume that $ 3/2 < s <  3/2 +\mu(\alpha)$.  Then 
\begin{equation} \label{g2-}
\|\, G_2^-(\lambda) \|_{\B(s,-s )} = o(|\lambda|^{\mu(\alpha)}), \qquad \| |\nabla G_2^-(\lambda) |\|_{\B(s,-s )} = o(|\lambda|^{\mu(\alpha)}) \qquad  \lambda\to 0- .
\end{equation} 
\end{lemma}

\begin{proof}
To simplify the notation we write below $\mu$ instead of $\mu(\alpha)$.
Similarly as in the case $\lambda>0$, see Lemma \ref{r0-asymp1}, we note that 
\begin{equation} \label{norm-sup-neg}
\| G_2^-(\lambda) \|_{\B(s,-s )} =  \sup_{m\in\Z} \| \rho^{-s}\, G_{m,2}^-(\lambda) \, \rho^{-s}\,\|_{\B(L^2(\R_+,r dr))} ,
\end{equation}
where
\begin{equation*} \label{gm2-neg}
 G_{m,2}^-(\lambda, r,r') = R_0^m(\lambda, r,r') - G_{m,0}(r,r') - \lambda^{\mu }\, \delta_{m,k(\alpha)}\, g_1(r,r').
\end{equation*}	
As in \eqref{0-1} we have
\begin{align} 
\| \, \rho^{-s}\, G_{m,2}^-(\lambda) \, \rho^{-s}\,\|^2_{\B(L^2(\R_+,r dr))} & \ \leq \ M^1_{(0,1)\times(0,1)} [G_{m,2}^-(\lambda)]\ M^2_{(0,1)\times(0,1)} [G_{m,2}^-(\lambda)]  \label{01-neg} \\
& \ + 2 \int_0^1\rho^{-2s}(r)   \int_1^\infty \, | G_{m,2}^-(\lambda, r,r')|^2\, \rho^{-2s}(r') \,  r'\, dr' \, r \, dr  \nonumber\\
& \ + \int_1^\infty \int_1^\infty \rho^{-2s}(r) \, | G_{m,2}^-(\lambda, r,r')|^2\, \rho^{-2s}(r')\,  \, r  r'\, dr' \,  dr \nonumber
\end{align}
For $r < r' \leq 1$ is the kernel $R_0^m(\lambda, r,r')$ given by the same solutions, $v_m$ and $u_m$, as in the case $\lambda>0$. Hence from the proof of Lemma \ref{r0-asymp1}, namely from the proof of upper bound \eqref{M-uv2} , we deduce that 
\begin{equation}  \label{M-uv3}
M^j_{(0,1)\times (0,1)} [\, G_{m,2}^-(\lambda)] \ \lesssim\  \frac{|\lambda|^{2\mu }}{1+|m|} , \qquad j=1,2,
\end{equation}
holds for all $m$ and $|\lambda|$ small enough. Let us now consider the last term on the right hand side of \eqref{01-neg}. From \eqref{rm-tilde}, taking into account \eqref{w-tilde},  we deduce that  for $ 1 \leq r<r'$
\begin{equation} \label{rm-tilde2}
R_0^m(\lambda, r,r') =  I_{|\alpha+m|}(\sqrt{|\lambda|}\, r)\, K_{|m+\alpha|}(\sqrt{|\lambda|}\, r') +\frac{\widetilde B_m(\lambda)}{\widetilde A_m(\lambda)} \ K_{|\alpha+m|}(\sqrt{|\lambda|}\, r) \, K_{|m+\alpha|}(\sqrt{|\lambda|}\, r') .
\end{equation}
In order to estimate the right hand side we cannot use the modified splitting formula \eqref{ik}, as we did in the with the equation \eqref{jy} in case $\lambda>0$, since both functions $I_{-\nu}$ and $I_\nu$ are exponentially increasing at infinity. Instead we proceed as follows; we consider first the contribution to $G_{m,2}^-(\lambda, r,r')$ relative to the product $I_{|m+\alpha|}(\sqrt{|\lambda|}\, r)\, K_{|m+\alpha|}(\sqrt{|\lambda|}\, r')$. This means that we have to estimate the $L^2-$norm, restricted to $(1,\infty)\times(1,\infty)$ of the kernel $\rho^{-s}(r)\, w_m(\lambda, r,r')\, \rho^{-s}(r')$, where
\begin{equation} \label{mixed-ik}
w_m(\lambda, r,r') := I_{|m+\alpha|}(\sqrt{|\lambda|}\, r)\, K_{|m+\alpha|}(\sqrt{|\lambda|}\, r') - \frac{ \big(r/r'\big)^{|m+\alpha|}}{2\, |m+\alpha|}\, + \, \frac{\delta_{m,k(\alpha)}\  |\lambda|^{\mu }\pi\, 4^{-\mu } \, (r r')^{\mu }}{2 \sin(\mu  \pi)\, \mu^2(\alpha)\,  \Gamma^2(\mu )} ,
\end{equation}
see \eqref{g1} and \eqref{rm-tilde1}. Without loss of generality we may assume that $|\lambda| <1$. We have 
\begin{align}
& \int_1^\infty \int_r^\infty \rho^{-2s}(r) \,  | w_m(\lambda, r,r')|^2\, \rho^{-2s}(r')\,   r'\, dr' \,  r\, dr  =  \int_{|\lambda|^{-\frac 12}}^\infty \int_r^\infty    | w_m(\lambda, r,r')|^2\, (r r')^{-2-2\eps}  dr'  dr \nonumber\\
 & \quad + \int_1^{|\lambda|^{-\frac 12}} \int_{|\lambda|^{-\frac 12}}^\infty   | w_m(\lambda, r,r')|^2\, (r r')^{-2-2\eps} \, dr'  dr +\int_1^{|\lambda|^{-\frac 12}} \int_r^{|\lambda|^{-\frac 12}}  | w_m(\lambda, r,r')|^2\,  (r r')^{-2-2\eps}\,  dr' dr \nonumber\\
 & \qquad\qquad\qquad\qquad = : X_1(m,\lambda)+X_2(m,\lambda)+X_3(m,\lambda). \label{3x}
\end{align}
In $X_1(m,\lambda)$ and $X_2(m,\lambda)$ we bound each term on the right hand side of \eqref{mixed-ik} separately; note that by monotonicity of $K_{|m+\alpha|}$ and \eqref{ik-upperb} 
\begin{align*}
& \int_{|\lambda|^{-\frac 12}}^\infty \int_r^\infty   I^2_{|m+\alpha|}(\sqrt{|\lambda|}\, r)\, K^2_{|m+\alpha|}(\sqrt{|\lambda|}\, r') \ (r r')^{-2-2\eps}  dr'  dr \leq \frac{1}{|m+\alpha|^2} \, \int_{|\lambda|^{-\frac 12}}^\infty  r^{-3-4\eps}\, dr \, \lesssim\, \frac{|\lambda|^{1+2\eps}}{|m+\alpha|^2}\ .
\end{align*}
Similarly we obtain
\begin{align*}
& \int_1^{|\lambda|^{-\frac 12}} \int_{|\lambda|^{-\frac 12}}^\infty   I^2_{|m+\alpha|}(\sqrt{|\lambda|}\, r)\, K^2_{|m+\alpha|}(\sqrt{|\lambda|}\, r') \ (r r')^{-2-2\eps}  dr'  dr \\
& \qquad\qquad\qquad\qquad  \leq \  |\lambda|^{\frac 12+\eps} \, K^2_{|m+\alpha|}(1) \int_1^{|\lambda|^{-\frac 12}} I^2_{|m+\alpha|}(\sqrt{|\lambda|}\, r)\, r^{-3-4\eps}\, dr \, \lesssim \ \frac{|\lambda|^{|m+\alpha|+\frac 12+\eps}}{|m+\alpha|^2},
\end{align*}
where we have used equations \eqref{ik-z0}  to estimate $K^2_{|m+\alpha|}(1)$ and $I^2_{|m+\alpha|}(\sqrt{|\lambda|}\, r)$. Elementary calculations now show that the contributions from the remaining terms on the right hand side of \eqref{mixed-ik} to $X_1(m,\lambda)$ and $X_2(m,\lambda)$ are of the same order in $|\lambda|$ and $m$. Since $\mu \leq \frac 12$, this implies that  
\begin{equation*} 
X_1(m,\lambda)+X_2(m,\lambda) \ \lesssim \ \frac{|\lambda|^{|m+\alpha|+\frac 12+\eps}+|\lambda|^{1+2\eps}}{|m+\alpha|^2}\ \leq \ \frac{|\lambda|^{2 \mu +\eps}}{|m+\alpha|^2}\, .
\end{equation*}
It remains to estimate $X_3(m,\lambda)$. Here we use the fact that $\sqrt{|\lambda|}\, r \leq \sqrt{|\lambda|}\, r' \leq 1$, see \eqref{3x}. Hence equations \eqref{ik-z0} we then obtain a pointwise estimate on $w_m(\lambda, r, r')$ which yields
$$
X_3(m,\lambda) \ \lesssim \  \frac{|\lambda|^{2 \mu +\eps}}{|m+\alpha|^2}\, .
$$
The remaining part of $ | G_{m,2}^-(\lambda, r,r')|^2$ on $(1,\infty)\times(1,\infty)$ is estimated in the same way using 
the asymptotic behavior of the coefficients $\widetilde A_m(\lambda)$ and $\widetilde B_m(\lambda)$ established above. We thus conclude that 
\begin{equation} \label{hs-infty}
\int_1^\infty \int_1^\infty \rho^{-2s}(r) \, | G_{m,2}^-(\lambda, r,r')|^2\, \rho^{-2s}(r')\,  \, r  r'\, dr' \,  dr  \ \lesssim \  \frac{|\lambda|^{2 \mu +\eps}}{|m+\alpha|^2}.
\end{equation}
Finally, we consider the mixed term in \eqref{01-neg}. We have
$$
R_0^m(\lambda; r,r') = \frac{v_m(\lambda,r)}{\widetilde A_m(\lambda)} \  K_{|m+\alpha|}(\sqrt{|\lambda|}\, r') \qquad r \leq 1 < r'.
$$
Consequently, by \eqref{G02}, \eqref{g1} and \eqref{rm-tilde1} 
\begin{align}
G_{m,2}^-(\lambda, r,r') & = \frac{v_m(\lambda,r)}{\widetilde A_m(\lambda)} \  K_{|m+\alpha|}(\sqrt{|\lambda|}\, r')  -  \frac{v_{m}(0,r)\, (r')^{-|m+\alpha|} }{a'_m+|m+\alpha|\, a_m}  \nonumber \\
& \quad +  |\lambda|^\mu\  \frac{\delta_{m,k(\alpha)} \ \pi\ v_k(0, r) \, }{\mu\, \sin(\mu \pi)\, 4^\mu\, \Gamma^2(\mu)\, (a'_k +\mu\, a_k) }\ \Big(r'^{\mu}- \frac{a'_k-\mu\, a_k}{a'_k+\mu\, a_k}\, r'^{-\mu}\Big) \label{gm2-neg-2}
\end{align}
We proceed similarly as above and for $r' \geq |\lambda|^{-\frac 12}$ estimate the contribution of each term to the integral in \eqref{01-neg} separately. By \eqref{m-infty},  Lemma \ref{lem-M} and the asymptotic expansion of $\widetilde A_m(\lambda)$ it follows that 
$$
\int_0^1  \frac{v^2_m(\lambda,r)}{\widetilde A^2_m(\lambda)} \, \rho^{-2s}(r)\, r \, dr \ \lesssim\ \frac{4^{-|m+\alpha|}\ |\lambda|^{|m+\alpha|}}{ (1+|m|)^3\ \Gamma^2(|m+\alpha|)}
$$
Taking into account the fact that $K_{|m+\alpha|}$ is decreasing and using \eqref{ik-z0}, we thus get 
$$
\int_0^1  \rho^{-2s}(r)\, r \, dr  \int_{|\lambda|^{-\frac 12}}^\infty\, |G_{m,2}^-(\lambda, r,r')|^2\, \rho^{-2s}(r')\, r'\, dr'  \ \lesssim \ \frac{|\lambda|^{ \mu +\frac 12+\eps}}{(|m| +1)^3}.
$$
In the region where $r' \leq |\lambda|^{-\frac 12}$ we use the two-term asymptotic expansion of the function $K_\nu(z)$ for $0 <z\leq 1$, cf. \eqref{ik-z0}. This gives 
$$
\int_0^1  \rho^{-2s}(r)\, r \, dr  \int_r^{|\lambda|^{-\frac 12}}\,  |G_{m,2}^-(\lambda, r,r')|^2\, \rho^{-2s}(r')\, r'\, dr'  \ \lesssim \ \frac{|\lambda|^{ \mu +\frac 12+\eps}}{(|m| +1)^3}.
$$
Inserting the above estimates together with \eqref{hs-infty} and \eqref{M-uv3} into \eqref{01-neg} then yields
\begin{equation} \label{almost}
\| \, \rho^{-s}\, G_{m,2}^-(\lambda) \, \rho^{-s}\,\|_{\B(L^2(\R_+,r dr))}  \ \lesssim \ \frac{|\lambda|^{\mu +\eps}}{|m| +1}.
\end{equation}
This proves the first part or \eqref{g2-}. As in the proof of Lemma \ref{r0-asymp2}, to prove the second part of \eqref{g2-} it suffices to show that 
$$
\sup_{m\in\Z}\, \int_0^\infty\int_0^\infty  \rho^{-2s}(r) \, \big[  \big(\pd_{r'} G_{m,2}^-(\lambda, r,r'\big)^2 + \frac{m^2}{r'^2}\, (G^-_{m,2}(\lambda,r,r'))^2 \big] \, \rho^{-2s}(r') \,r\, r'\, dr dr'  = o (|\lambda|^{2\mu})
$$
as $\lambda \to 0-$, see the proof of Lemma \ref{r0-asymp2}. However, this follows from the explicit expression for $G_{m,2}^-(\lambda,r,r')$ and from \eqref{almost}.
\end{proof}

\subsection{Proof of Proposition \ref{prop-1}}
The claim of the Proposition follows by combining Lemmata \ref{lem-no-int}, \ref{lem-g0}, \ref{r0-asymp1}, \ref{r0-asymp2}, \ref{lem-neg1} and \ref{lem-neg2}.

%%%%%%%%%%%%%%%%%%%%%%%%%%%%%%%%%%%%%%%%%%%%%%%%%%%%%
\section{\bf The operator $H(B_0)$: integer flux}
\label{sec-integer}
\noindent 
For integer values of $\alpha$ we can use the results of the previous sections for the contributions to the asymptotic expansion for $R_0(\lambda)$ from the channels $m\neq -\alpha$. Hence it remains to analyze the contribution from $m=-\alpha$ which.

%%%%%%%%%%%%%%%%%%%%%%%%%%%%%%%%%%%%%%%%%%%%%%%%%%%%%
\subsection{The case $\lambda>0$} 

\begin{lemma} \label{lem-int}
Assume that $\alpha\in\Z, \, \alpha\neq 0$.  Then for every $x,y\in\R^2$ it holds 
\begin{equation} \label{r0-int}
R_0(\lambda; x, y) = \G_0(x,y) +\ (\log \lambda)^{-1}\, \G_1(x,y) +  \G_2^+(\lambda; x,y),
\end{equation}
where $\G_0(x,y)$ and  $\G_1(x,y)$  are given by equations \eqref{G0-a} and \eqref{G1-a} and $\G_2^+(\lambda; x,y) = o\big( (\log \lambda)^{-1} \big)$ as $\lambda\to 0+$.  
\end{lemma}

\begin{proof}
Without loss of generality we may assume that $\alpha >0$. We proceed in the same way as above using the functions $f_{m,\lambda}$ and $\phi_{m,\lambda}$ defined by \eqref{fm} and \eqref{fim}. For $m\neq -\alpha$ we calculate $R_0^m(\lambda; r,r')$ from the formulas \eqref{rtor} and \eqref{rm}-\eqref{wronsk}.  If $m=-\alpha<0$, then 
\begin{align} 
A_{-\alpha}(\lambda) & = -\frac{\pi}{2}\, \Big( v'_{-\alpha}(\lambda,1)\, Y_0(\sqrt{\lambda}) +\sqrt{\lambda}\ v_{-\alpha}(\lambda,1)\,  Y_1(\sqrt{\lambda}) \Big ) \label{a-eq} \\
%& \nonumber \\
B_{-\alpha}(\lambda) & =  \frac{\pi}{2}\, \Big( v'_{-\alpha}(\lambda,1) \, J_0(\sqrt{\lambda}) +\sqrt{\lambda}\  v_{-\alpha}(\lambda,1) \, J_1(\sqrt{\lambda}) \Big ) \label{b-eq},
\end{align} 
and
\begin{align} 
C_{-\alpha}(\lambda) &=  \frac{ \Gamma(\frac 12+2\alpha)}{\Gamma(1+2\alpha)} \ \Big( u'_{-\alpha}(\lambda, 1)\, (J_0(\sqrt{\lambda}) -i\, Y_0(\sqrt{\lambda}))  +\sqrt{\lambda}\ u_{-\alpha}(\lambda, 1)\,  (J_1(\sqrt{\lambda}) -i\, Y_1(\sqrt{\lambda})\, ) \Big)\label{c-eq}  \\
D_{-\alpha}(\lambda) &=  \frac{2\,  \Gamma(\frac 12+2\alpha)}{\pi\, \Gamma(1+2\alpha)} \ (i \, A_{-\alpha}(\lambda)-B_{-\alpha}(\lambda)).  \label{d-eq}
\end{align}
The rest of the calculation proceeds in the same way as is the case $\alpha\not\in\Z$. To simplify the notation we introduce the following shorthands:
$$
\af_\alpha = v_{-\alpha}(0,  1), \qquad \af'_\alpha = v'_{-\alpha}(0, 1) , \qquad \bb_\alpha = u_{-\alpha}(0, 1), \qquad \bb'_\alpha = u'_{-\alpha}(0, 1) .
$$
With the help of \eqref{hk-gen}, \eqref{rm}, and properties of Bessel functions, see equations \eqref{nu-pos} and \eqref{nu-zero}, we then obtain 
\begin{equation} \label{G0-a}
\lim_{\lambda\to 0+}\, R_0(\lambda;x,y) = \G_0(x,y)=  \sum_{m\neq -\alpha} G_{m,0}(r,r')\, e^{im(\theta-\theta')}  +\, \G_{\alpha,0} (r,r') \, e^{i \alpha(\theta'-\theta)},
\end{equation} 
with 
\begin{align*}
\G_{\alpha,0}(r,r') 
& = \frac{ \Gamma(\frac 12+2\alpha)}{\Gamma(1+2\alpha)}\, v_{-\alpha}(0,  r)\, \big( u_{-\alpha}(0, r')-
\frac{\bb'_\alpha}{ \af'_\alpha}  \ v_{-\alpha}(0,  r') \big)  & r< r' \leq 1, \\
%& \\
\G_{\alpha,0}(r,r') & = \frac{v_{-\alpha}(0,  r)}{\af'_\alpha} & r\leq 1 < r', \\
%& \\
\G_{\alpha,0}(r,r') & = \frac{\af_\alpha}{\af'_\alpha}\, +\log r & 1< r <r'. 
\end{align*}
Similarly, we find that 
\begin{equation} \label{G1-a}
\lim_{\lambda\to 0+}\,  \log \lambda\  \Big(R_0(\lambda;x,y) - \G_0(x,y) \Big) =:
\G_1(x,y) = k_1(\alpha; r,r') \, e^{i\alpha(\theta-\theta')},
\end{equation}
with 
\begin{align*}
k_1(\alpha; r,r')  & =  \frac{2\, v_{-\alpha}(0,  r)\, v_{-\alpha}(0,  r')}{(\af'_\alpha)^2} & r< r' \leq 1, \\
%& \\
k_1(\alpha; r,r') & = \frac{2\, v_{-\alpha}(0,  r)}{\af'_\alpha}\,  \Big(\frac{\af_\alpha}{\af'_\alpha} + \log r'\Big) & r<1 < r' , \\
%& \\
k_1(\alpha; r,r') & = 2\, \Big(\frac{\af_\alpha}{\af'_\alpha} + \log r\Big) \Big(\frac{\af_\alpha}{\af'_\alpha} + \log r'\Big) & 1\leq r <r'. 
\end{align*}
\end{proof}

\begin{remark}
Note that by a direct calculation using equations \eqref{vm}, \eqref{eq-M} and \eqref{mu-der} we have
$$
\af'_\alpha = |\alpha| \, e^{-|\alpha| }\, |2 \alpha|^{|\alpha|}\, \Big [\, |\alpha|\, M\Big(\frac 12, 1+2 |\alpha|, 2 |\alpha| \Big) + \frac{1}{2+4 |\alpha|}\, M\Big(\frac 32, 2+2|\alpha| , 2|\alpha| \Big)\Big] >0.
$$
\end{remark}

\smallskip

\begin{lemma} \label{lem-g01-int}
Let $ \ s>1$.  Then $\rho^{-s}\, \G_j\, \rho^{-s}$ and $\rho^{-s}\, \nabla\, \G_j\, \rho^{-s}$ with $j=0,1$ are compact operators from $L^2(\R^2)$ to $L^2(\R^2)$ and from $L^2(\R^2)$ to $L^2(\R^2,\C^2)$ respectively. 
\end{lemma}

\begin{proof}
By Lemma \ref{lem-g0} the operators with kernels
$$
\rho^{-s} \sum_{m\neq -\alpha} G_{m,0}(r,r')\, e^{im(\theta-\theta')} \ \rho^{-s} \qquad \text{and}\qquad   \rho^{-s}\, \Big(\, \nabla\! \sum_{m\neq -\alpha} G_{m,0}(r,r')\, e^{im(\theta-\theta')} \Big)\,  \rho^{-s}
$$
are compact from $L^2(\R^2)$ to $L^2(\R^2)$ and from $L^2(\R^2)$ to $L^2(\R^2,\C^2)$. On the other hand,  from the explicit equation for  $\G_{\alpha,0} (r,r') $ and from Lemmata \ref{lem-M} and \ref{lem-der-m} it easily follows that
the kernels 
$$
\rho^{-s}\, \G_{\alpha,0} (r,r') \, e^{i \alpha(\theta-\theta')}\, \rho^{-s} \qquad  \text{and} \qquad  
\rho^{-s}\, | \nabla \, \G_{\alpha,0} (r,r') \, e^{i \alpha(\theta-\theta')}|\, \rho^{-s}
$$ 
are Hilbert-Schmidt in $L^2(\R^2)$. 
\end{proof}

\begin{lemma} \label{r0-asymp3}
Let $\G_2^+(\lambda)$ denote the integral operator with the kernel $\G_2^+(\lambda; x,y)$ defined by \eqref{r0-int}. 
Assume that $s >  3/2+\eps, \,  0 <\eps <1$. Then, as $\lambda\to 0+$ we have $\G_2^+(\lambda) = o\big((\log\lambda )^{-1}\big)$ and $|\nabla\, \G_2^+(\lambda) |= o\big((\log\lambda )^{-1}\big)$ in $\B(s,-s )$. 
\end{lemma}

\begin{proof}
We have 
\begin{equation}
\G_2^+(\lambda; x,y) = \sum_{m\in\Z} \G_{m,2}^+(r,r')\, e^{im(\theta-\theta')}\, ,
\end{equation}
where
\begin{equation} \label{gm2b}
 \G_{m,2}^+(\lambda, r,r') = 
 \left\{
\begin{array}{l@{\quad}l}
R_0^m(\lambda, r,r') - G_{m,0}(r,r') &\quad m \neq -\alpha , \\
&\\
R_0^{-\alpha}(\lambda, r,r') - \G_{\alpha,0}(r,r') -(\log \lambda)^{-1}\, k_1(r,r'), & \quad  m=-\alpha.
\end{array}
\right.
\end{equation}
Let us first consider the case $m\neq -\alpha$. We are going to use the results of Lemma \ref{r0-asymp1}. Note that all the indexes of the Bessel function included in $G_{m,0}(r,r')$ are integers. Consequently, we define $Y_{|m+\alpha|}$  by its integral representation \eqref{int-j-2} instead of \eqref{jy}. By using Lemmata \ref{lem-int-JJ}, \ref{lem-int-JJ'} and \ref{lem-mixed} with the function $J_{-\nu}$ replaced by $Y_\nu$ it follows that the results of Lemma \ref{r0-asymp1} remain valid. In particular, from \eqref{aux-rest1} we infer that 
\begin{equation} \label{m-neq-a} 
\sup_{m\neq -\alpha} \|\,  \rho^{-s}  (R_0^m(\lambda) - G_{m,0})  \, \rho^{-s}\,\|_{\B(L^2(\R_+,r dr))} \, = \, \mathcal{O}(\lambda), \qquad \lambda \to 0+,
\end{equation}
where $R_0^m(\lambda)$ is the operator in $L^2(\R_+,r dr)$ generated by the kernel $R_0^m(\lambda, r,r') $. In order to treat the case $m=-\alpha$, we will need the following asymptotic expansions:
\begin{align}
\frac{A_{-\alpha}(\lambda)}{W_{-\alpha}(\lambda)} & = \frac{i\, \pi}{2} \ \Big (1+i\,\pi\, (\log\lambda )^{-1} - 2 \pi\, \Big(  \frac{\pi}{2} + i\, \frac{c_\alpha}{\af_\alpha'} \Big) \, (\log\lambda )^{-2} + o( (\log\lambda )^{-2}) \Big),   \label{aw-as} \\
\frac{B_{-\alpha}(\lambda)}{W_{-\alpha}(\lambda)} & = -\frac{i\, \pi^2}{2}\, (\log\lambda )^{-1} +\pi^2 \Big(\frac{\pi}{2} -\frac{c_\alpha}{\af_\alpha'}\Big)\,  (\log\lambda )^{-2} -2 i \pi^2\Big( \frac{c^2_\alpha}{\af_\alpha'^2} - \frac{\pi^2}{4}\Big)\, (\log\lambda )^{-3} + o((\log\lambda )^{-3}) ,\label{bw-as} \\
\frac{C_{-\alpha}(\lambda)}{W_{-\alpha}(\lambda)} & = \frac{2\,  \Gamma(\frac 12+\alpha)\,  \bb'_\alpha}{i\, \Gamma(1+2\alpha)\, \af_\alpha'} \ \Big (1+i\,\pi\, (\log\lambda )^{-1} - 2 \pi\, \Big(  \frac{\pi}{2} + i\, \frac{c_\alpha}{\af_\alpha'} \Big) \, (\log\lambda )^{-2} + o( (\log\lambda )^{-2}) \Big),   \label{cw-as}
\end{align}
as $ \lambda \to 0+$, where $c_\alpha= \af_\alpha - \af_\alpha'(\gamma-\log 2)$. These expansions can be derived directly from equations \eqref{a-eq}-\eqref{c-eq} and asymptotics \eqref{nu-pos}, \eqref{nu-zero}. By inserting the asymptotic equations \eqref{aw-as}-\eqref{cw-as} into the expression for $R_0^{-\alpha}(\lambda, r,r') $ and using inequalities \eqref{eq-jy0} of Lemma \ref{lem-jy0} we find that 
$$
\|\, \rho^{-s} (R_0^{-\alpha}(\lambda, r,r') - \G_{\alpha,0}(r,r') -(\log \lambda)^{-1}\, k_1(r,r') )\, \rho^{-s} \|_{HS(\R_+, r dr)} = o( (\log\lambda )^{-1}) , \quad \lambda \to 0+. 
$$
This in combination with \eqref{m-neq-a} and \eqref{gm2b} shows that $\G_2^+(\lambda) = o\big((\log\lambda )^{-1}\big)$ in $\B(s,-s )$. To prove the remaining claim we proceed in the similar way. From Lemma \ref{r0-asymp2} we conclude that 
\begin{equation*}  
\sup_{m\neq -\alpha} \|\,  \rho^{-s}\, \big(  \nabla  (R_0^m(\lambda) - G_{m,0})  \big) \, \rho^{-s}\,\|_{\B(L^2(\R_+,r dr))} \, = \, \mathcal{O}(\lambda), \qquad \lambda \to 0+.
\end{equation*}
On the other hand, with the help of \eqref{aw-as}-\eqref{cw-as} and inequalities \eqref{eq-jy0-der} of Lemma \ref{lem-jy0} we obtain
$$
\int_0^\infty \int_0^\infty \rho^{-2s}(r') \big | \pd_{r'}\, (R_0^{-\alpha}(\lambda, r,r') - \G_{\alpha,0}(r,r') -(\log \lambda)^{-1}\, k_1(r,r') )\, \big|^2 \rho^{-2s}(r)\, r r'\, dr\, dr'  = o( (\log\lambda )^{-2}) 
$$
as $\lambda \to 0+$. The statement now follows again by \eqref{m-neq-a} and \eqref{gm2b}. 
\end{proof}

%%%%%%%%%%%%%%%%%%%%%%%%%%%%%%%%%%%%%%%%%%%%%%%%%%%%%
\subsection{The case $\lambda<0$} 

\begin{lemma} \label{lem-int-2}
Let $\alpha\in\Z, \alpha\neq 0$. Assume that $s >  3/2+\eps, \,  0 <\eps <1$. Then for $\lambda<0$  and $|\lambda|$ small enough we have
\begin{equation*}
R_0(\lambda+i0) = \G_0 + (\log \lambda)^{-1}\, \G_1 +\G_2^-(\lambda)  \qquad \text{in} \qquad \B(s,-s ),
\end{equation*}
where 
\begin{equation*} 
\|\, \G_2^-(\lambda) \|_{\B(s,-s )} = o\big(\log |\lambda|)^{-1}\big), \qquad \| |\nabla \G_2^-(\lambda) |\|_{\B(s,-s )} = o\big(\log |\lambda|)^{-1}\big) \qquad  \lambda\to 0- .
\end{equation*} 
\end{lemma}

\begin{proof}
Taking into account the asymptotic equation \eqref{k-log}, the claim follows in the same way as in the case $\lambda>0$, cf. Lemma \ref{r0-asymp3}.
\end{proof}

\subsection{Proof of Proposition \ref{prop-2}}
The statement of the Proposition  follows from Lemmata \ref{lem-int}, \ref{lem-g01-int}, \ref{r0-asymp3} and \ref{lem-int-2}.

%%%%%%%%%%%%%%%%%%%%%%%%%%%%%%%%%%%%%%%%%%%%%%%%%%%%%
\section{\bf Auxiliary integral estimates}
\label{sec-aux-bessel}

\noindent In this section we prove several integral estimates on Bessel functions and their derivatives. These results will be used in the proof of Lemmata \ref{r0-asymp1} and \ref{r0-asymp2} .

\begin{lemma} \label{lem-jj}
Let $\nu \geq 0$.  Assume that $ s> \frac 32+\eps,\ 0 <\eps <1$. Then for all $\lambda\in(0,1)$ we have
$$
\int_1^\infty  J^2_{\nu}(\sqrt{\lambda}\, r) \, \rho^{-2s}(r)\, r\, dr  \  \lesssim\ (\lambda^{\nu} +\lambda^{\frac 12+\eps}) \ (1+\nu^2)^{-1}.
$$
\end{lemma}

\begin{proof}
For $\nu\leq 2$ use \eqref{J-unif} and \eqref{nu-pos}: 
\begin{align*}
\int_1^\infty  J^2_{\nu}(\sqrt{\lambda}\, r) \, \rho^{-2s}(r)\, r\, dr  &  \leq \ \lambda^{\frac 12+\eps} \int_{\sqrt{\lambda}}^\infty  J^2_{\nu}(t) \, t^{-2-2\eps}\, dt \ \lesssim \ \lambda^{\frac 12+\eps} \left(\int_{\sqrt{\lambda}}^1  t^{2\nu-2-2\eps}\, dt + \int_1^\infty\, t^{-2-2\eps}\, dt\right) \\
&  \lesssim \ \lambda^{\nu} +\lambda^{\frac 12+\eps} .
\end{align*}
For $\nu >2$ the  estimate follows from \eqref{wa-int-1} and the identity $\Gamma(\nu+\frac 32) = (\nu^2-\frac 14)\, \Gamma(\nu-\frac 12)$.
\end{proof}

\begin{lemma} \label{lem-int-JJ}
Let $\nu >0$. Assume that $s> \frac 32+\eps,\ 0 <\eps <1$. Then for all $\lambda\in(0,1)$ it holds 
\begin{align} \label{eq-JJ-1}
\int_1^\infty \int_r^\infty \left |\, \partial_\lambda \Big (\lambda^{-\nu} \, J_{\nu}(\sqrt{\lambda}\, r) \, J_{\nu}(\sqrt{\lambda}\, r')  \Big )\right |^2\, \rho^{-2s}(r')\, \rho^{-2s}(r)\, r\, r'\, dr dr' \ &\lesssim\  \lambda^{2\eps-1-2\nu}\, (1+\nu)^{-2} \\
\label{eq-JJ-2}
\int_1^\infty \int_r^\infty \left |\, \partial_\lambda \Big (\lambda^{\nu} \, J_{-\nu}(\sqrt{\lambda}\, r) \, J_{-\nu}(\sqrt{\lambda}\, r')  \Big )\right |^2\, \rho^{-2s}(r')\, \rho^{-2s}(r)\, r\, r'\, dr dr'
\ &\lesssim\  2^{4\nu}\, \Gamma^4(\nu)\,  \\
\label{eq-JJ-3}
\int_1^\infty \int_r^\infty \left |\, \partial_\lambda \Big (J_{\pm\nu}(\sqrt{\lambda}\, r) \, J_{\mp\nu}(\sqrt{\lambda}\, r')  \Big )\right |^2\, \rho^{-2s}(r')\, \rho^{-2s}(r)\, r\, r'\, dr dr' \ &\lesssim\  \lambda^{\eps-1}\, (1+\nu)^{-1}.
\end{align}
Moreover, the function $J_{-\nu}$ in the above estimates can be replaced by $Y_\nu$ throughout without changing the right hand side. 
\end{lemma}

\begin{proof}
Assume first that $\nu >2$ and consider the bound \eqref{eq-JJ-1}.  From \eqref{der-2} we find that 
\begin{equation} \label{jj-der}
\partial_\lambda \Big (\lambda^{-\nu}\, J_{\nu}(\sqrt{\lambda}\, r) \, J_{\nu}(\sqrt{\lambda}\, r') \Big ) = -  \frac{\lambda^{-\nu-\frac 12}}{2}\ \left( r\, J_{\nu+1}(\sqrt{\lambda}\, r) \ J_{\nu}(\sqrt{\lambda}\, r') + r'\, J_{\nu+1}(\sqrt{\lambda}\, r') \ J_{\nu}(\sqrt{\lambda}\, r) \right).
\end{equation}
With the help of \eqref{wa-int-1} we estimate the first term as follows 
\begin{align*}
& \int_1^\infty \int_r^\infty  J^2_{\nu+1}(\sqrt{\lambda}\, r) \, J^2_{\nu}(\sqrt{\lambda}\, r')\,  r^3\, r' \rho^{-2s}(r')\, \rho^{-2s}(r) \,dr dr'   \\
&\qquad \qquad  \leq \int_1^\infty\, J^2_{\nu+1}(\sqrt{\lambda}\, r)\, r^{-2\eps}\, dr  \int_1^\infty  \, J^2_{\nu}(\sqrt{\lambda}\, r')\,  r'^{-2-2\eps}  dr' \  \lesssim \ \lambda^{\eps-\frac 12}\, \lambda^{\eps+\frac 12}\, (1+\nu)^{-2}.
\end{align*}
The second term in \eqref{jj-der} is estimated in the same way. This proves \eqref{eq-JJ-1} for $\nu >2$. 
To prove we use again \eqref{der-1} and calculate
\begin{equation}  \label{jj-der2}
\partial_\lambda \Big( \lambda^{\nu}\, J_{-\nu}(\sqrt{\lambda}\, r) \, J_{-\nu}(\sqrt{\lambda}\, r') \Big ) = -\frac{\lambda^{\nu-\frac 12}}{2} \left(  r\, J_{1-\nu}(\sqrt{\lambda}\, r) \ J_{-\nu}(\sqrt{\lambda}\, r') + r'\, J_{1-\nu}(\sqrt{\lambda}\, r') \ J_{-\nu}(\sqrt{\lambda}\, r) \right).
\end{equation}
To control the first term on the right hand side we recall \eqref{jy^2} and \eqref{monoton} (with $j=0$). This gives
\begin{align*}
& \int_1^\infty \int_r^\infty  J^2_{1-\nu}(\sqrt{\lambda}\, r) \, J^2_{-\nu}(\sqrt{\lambda}\, r')\,  r^3\, r' \rho^{-2s}(r')\, \rho^{-2s}(r) \,dr dr' \leq \\
&\qquad \qquad \leq \int_1^\infty J^2_{1-\nu}(\sqrt{\lambda}\, r) \int_r^\infty  \big( J^2_{\nu}(\sqrt{\lambda}\, r') + Y^2_{\nu}(\sqrt{\lambda}\, r')\big) \,   r'^2 \rho^{-2s}(r')\, dr' \rho^{-2s}(r) \, r^2\, dr  \\
&\qquad \qquad \lesssim  \int_1^\infty  J^2_{1-\nu}(\sqrt{\lambda}\, r) \,\big( J^2_{\nu}(\sqrt{\lambda}\, r) + Y^2_{\nu}(\sqrt{\lambda}\, r)\big) \,  \rho^{-2s}(r) \, r^2\, dr  \\
&\qquad \qquad \leq  \int_1^\infty  \big( J^2_{\nu-1}(\sqrt{\lambda}\, r) + Y^2_{\nu-1}(\sqrt{\lambda}\, r)\big)  \,\big( J^2_{\nu}(\sqrt{\lambda}\, r) + Y^2_{\nu}(\sqrt{\lambda}\, r)\big) \,  \rho^{-2s}(r) \, r^2\, dr \\
&\qquad \qquad \lesssim \big( J^2_{\nu-1}(\sqrt{\lambda}) + Y^2_{\nu-1}(\sqrt{\lambda})\big)  \,\big( J^2_{\nu}(\sqrt{\lambda}) + Y^2_{\nu}(\sqrt{\lambda})\big) \ \lesssim \ \lambda^{1-2\nu}\, 2^{4\nu}\, \Gamma^4(\nu).
\end{align*}
In the last step we have used \eqref{jy} and \eqref{nu-pos}. The second term in \eqref{jj-der2} can be estimated with the help of \eqref{monoton} applied with $j=1$. Indeed, proceeding as above we find 
\begin{align*}
& \int_1^\infty \int_r^\infty  J^2_{1-\nu}(\sqrt{\lambda}\, r') \, J^2_{-\nu}(\sqrt{\lambda}\, r)\,  r'^3\, r \rho^{-2s}(r')\, \rho^{-2s}(r) \,dr dr' \leq \\
&\qquad \qquad \leq \int_1^\infty J^2_{-\nu}(\sqrt{\lambda}\, r)  \int_r^\infty  \big( J^2_{\nu-1}(\sqrt{\lambda}\, r') + Y^2_{\nu-1}(\sqrt{\lambda}\, r')\big) \,   r'^3 \rho^{-2s}(r')\, dr' \rho^{-2s}(r) \, r^2\, dr  \\
&\qquad \qquad \lesssim \ \int_1^\infty  J^2_{-\nu}(\sqrt{\lambda}\, r) \,\big( J^2_{\nu-1}(\sqrt{\lambda}\, r) + Y^2_{\nu-1}(\sqrt{\lambda}\, r)\big) \,  \rho^{-2s}(r) \, r^2\, dr  \\
&\qquad \qquad \lesssim\  \big( J^2_{\nu-1}(\sqrt{\lambda}) + Y^2_{\nu-1}(\sqrt{\lambda})\big)  \,\big( J^2_{\nu}(\sqrt{\lambda}) + Y^2_{\nu}(\sqrt{\lambda})\big) \ \lesssim \ \lambda^{1-2\nu}\, 2^{4\nu}\, \Gamma^4(\nu).
\end{align*}
As for \eqref{eq-JJ-3}, we note that \eqref{der-2} implies
\begin{equation}   \label{jj-der3}
\partial_\lambda \Big ( J_{\nu}(\sqrt{\lambda}\, r)  \, J_{-\nu}(\sqrt{\lambda}\, r) \Big ) \ = \  -\frac{1}{2\sqrt{\lambda}}\, \Big(r\, J_{\nu+1}(\sqrt{\lambda}\, r)\, J_{-\nu}(\sqrt{\lambda}\, r') + r'\, J_{1-\nu}(\sqrt{\lambda}\, r')\, J_{\nu}(\sqrt{\lambda}\, r) \Big). 
\end{equation}
We now recall again \eqref{monoton} with $j=0$ and estimate the first term on the right hand side of \eqref{jj-der3} as follows:
\begin{align}
& \int_1^\infty \int_r^\infty  J^2_{\nu+1}(\sqrt{\lambda}\, r) \, J^2_{-\nu}(\sqrt{\lambda}\, r')\,  r^3\, r' \rho^{-2s}(r')\, \rho^{-2s}(r) \,dr dr' \leq \nonumber \\
&\qquad \qquad \leq \ \int_1^\infty J^2_{\nu+1}(\sqrt{\lambda}\, r) \int_r^\infty  \big( J^2_{\nu}(\sqrt{\lambda}\, r') + Y^2_{\nu}(\sqrt{\lambda}\, r')\big) \,   r'^2 \rho^{-2s}(r')\, dr' \rho^{-2s}(r) \, r^2\, dr  \nonumber \\
&\qquad \qquad \lesssim\   \int_1^\infty  J^2_{\nu+1}(\sqrt{\lambda}\, r) \,\big( J^2_{\nu}(\sqrt{\lambda}\, r) + Y^2_{\nu}(\sqrt{\lambda}\, r)\big) \, r^{-1-4\eps}\, dr  \nonumber \\
 &\qquad \qquad  = \ \lambda^{2\eps} \int_{\sqrt{\lambda}}^\infty  J^2_{\nu+1}(t) \big( J^2_{\nu}(t) + Y^2_{\nu}(t)\big) \ t^{-1-4\eps} \, dt.  \label{nu-pm}
\end{align}
To proceed we split the integration in \eqref{nu-pm} with respect to $t$ in three parts as follows:
\begin{align*}
\int_{\sqrt{\lambda}}^1\ J^2_{\nu+1}(t) \big( J^2_{\nu}(t) + Y^2_{\nu}(t)\big) \ t^{-1-4\eps} \, dt \ & \lesssim\ \nu^{-2},
\end{align*}
where we have used \eqref{nu-pos},
\begin{align*}
\int_1^{\nu+1} J^2_{\nu+1}(t) \big( J^2_{\nu}(t) + Y^2_{\nu}(t)\big) \ t^{-1-4\eps} \, dt \, & \lesssim\, \nu^{-4/3}\, 
\int_1^{\nu+1}  t^{-1-4\eps} \, dt \lesssim\, \nu^{-4/3},
\end{align*}
where we have used \eqref{product-upb}, and 
\begin{align*}
\int_{\nu+1}^\infty J^2_{\nu+1}(t) \big( J^2_{\nu}(t) + Y^2_{\nu}(t)\big) \ t^{-1-4\eps} \, dt \ & \lesssim\ 
\int_{\nu+1}^\infty \frac{ t^{-1-4\eps}}{\sqrt{(t^2-(\nu+1)^2)\, (t^2-\nu^2)}}\   dt \ \lesssim\ \nu^{-1} \, ,
\end{align*}
in view of \eqref{z-large}.
The second term on the right hand side of \eqref{jj-der3} is treated in the analogous way; using \eqref{monoton} with $j=1$ and following the estimates used in \eqref{nu-pm} we get 
\begin{align*}
&  \int_1^\infty \int_r^\infty  J^2_{\nu}(\sqrt{\lambda}\, r) \, J^2_{1-\nu}(\sqrt{\lambda}\, r')\,  r^3\, r' \rho^{-2s}(r')\, \rho^{-2s}(r) \,dr dr'  \\
&\qquad \qquad \leq \ \int_1^\infty J^2_{\nu}(\sqrt{\lambda}\, r) \int_r^\infty  \big( J^2_{\nu-1}(\sqrt{\lambda}\, r') + Y^2_{\nu-1}(\sqrt{\lambda}\, r')\big) \,   r'^3 \rho^{-2s}(r')\, dr' \rho^{-2s}(r) \, r\, dr  \nonumber \\
&\qquad \qquad \lesssim\   \int_1^\infty  J^2_{\nu}(\sqrt{\lambda}\, r) \,\big( J^2_{\nu-1}(\sqrt{\lambda}\, r) + Y^2_{\nu-1}(\sqrt{\lambda}\, r)\big) \, r^{-1-4\eps}\, dr  \nonumber \\
 &\qquad \qquad  = \ \lambda^{2\eps} \int_{\sqrt{\lambda}}^\infty  J^2_{\nu}(t) \big( J^2_{\nu-1}(t) + Y^2_{\nu-1}(t)\big) \ t^{-1-4\eps} \, dt. 
  \end{align*}
By splitting the last integral in three parts in the same way as above and taking into account \eqref{jj-der3} we arrive at
$$
\int_1^\infty \int_r^\infty \left |\, \partial_\lambda \Big (J_{\nu}(\sqrt{\lambda}\, r) \, J_{-\nu}(\sqrt{\lambda}\, r')  \Big )\right |^2\, \rho^{-2s}(r')\, \rho^{-2s}(r)\, r\, r'\, dr dr' \ \lesssim\  \lambda^{\eps-1}\, (1+\nu)^{-1} , \quad \nu >2.
$$
The term $J_{-\nu}(\sqrt{\lambda}\, r) \, J_{\nu}(\sqrt{\lambda}\, r')$  in \eqref{eq-JJ-3} is estimated in the same way. This completes the proof of the Lemma in the case $\nu >2$.  If $\nu\leq 2$, then the bounds  \eqref{eq-JJ-1}-\eqref{eq-JJ-3} follow directly from \eqref{jj-der}, \eqref{jj-der2} and \eqref{jj-der3} by \eqref{z-large}. 

When $J_{-\nu}$ is replaced by $Y_\nu$ on the left hand side of  \eqref{eq-JJ-1}- \eqref{eq-JJ-3}, then we proceed in the same way as above using the obvious inequality $Y^2_\nu(z) \leq Y^2_\nu(z)+J^2_\nu(z)$ instead of \eqref{jy^2}.
\end{proof}

\begin{lemma} \label{lem-int-JJ'}
Let $\nu >0$. Assume that $s> \frac 32+\eps,\ 0 <\eps <1$. Denote
$$
J_\nu^{(1)}(\sqrt{\lambda}\, r) = \pd_r J_{\nu}(\sqrt{\lambda}\, r), \qquad 
 J_\nu^{(2)}(\sqrt{\lambda}\, r)  = \frac{\nu}{r}\, J_{\nu}(\sqrt{\lambda}\, r).
$$
Then for all $\lambda\in(0,1)$  and $n=1,2$ it holds 
\begin{align} \label{eq-JJ-4}
\int_1^\infty \int_r^\infty \left |\, \partial_\lambda \Big (\lambda^{-\nu} \, J_{\nu}(\sqrt{\lambda}\, r) \, J^{(n)}_{\nu}(\sqrt{\lambda}\, r')  \Big )\right |^2\, \rho^{-2s}(r')\, \rho^{-2s}(r)\, r\, r'\, dr dr' \ &\lesssim\  \lambda^{\eps-1-2\nu} \, (1+\nu)^{-2}\\
\label{eq-JJ-5}
\int_1^\infty \int_r^\infty \left |\, \partial_\lambda \Big (\lambda^{\nu} \, J_{-\nu}(\sqrt{\lambda}\, r) \, J^{(n)}_{-\nu}(\sqrt{\lambda}\, r')  \Big )\right |^2\, \rho^{-2s}(r')\, \rho^{-2s}(r)\, r\, r'\, dr dr'
\ &\lesssim\  2^{4\nu}\, \Gamma^4(\nu)\,  \\
\label{eq-JJ-6}
\int_1^\infty \int_r^\infty \left |\, \partial_\lambda \Big (J_{\pm\nu}(\sqrt{\lambda}\, r) \, J^{(n)}_{\mp\nu}(\sqrt{\lambda}\, r')  \Big )\right |^2\, \rho^{-2s}(r')\, \rho^{-2s}(r)\, r\, r'\, dr dr' \ &\lesssim\  \lambda^{\eps-1}\, (1+\nu)^{-1}
\end{align}
\end{lemma}

\begin{proof}
From \eqref{der-1} and \eqref{der-2} we obtain
$$
J_\nu^{(1)}(\sqrt{\lambda}\, r) = \frac{\sqrt{\lambda}}{2}\ \Big [J_{\nu-1}(\sqrt{\lambda}\, r)-J_{\nu+1}(\sqrt{\lambda}\, r) \Big], \quad 
J_\nu^{(2)}(\sqrt{\lambda}\, r) = \frac{\sqrt{\lambda}}{2}\  \Big [J_{\nu-1}(\sqrt{\lambda}\, r) + J_{\nu+1}(\sqrt{\lambda}\, r) \Big].
$$
Hence the result follows in the same way as the proof of Lemma \ref{lem-int-JJ}.
\end{proof}

\begin{lemma} \label{lem-mixed}
Let $\nu > 1$. Assume that $s> \frac 32+\eps,\ 0 <\eps <1$. Then for all $\lambda\in(0,1)$ it holds 
\begin{align}
\int_1^\infty \left |\, \partial_\lambda \big (\lambda^{\frac{\nu}{2}} \, J_{-\nu}(\sqrt{\lambda}\, r) \big )\right |^2\,  \rho^{-2s}(r)\, r \, dr  & \ \lesssim\  \lambda^\eps\, 4^\nu\, \Gamma^2(\nu-1),  \label{mix-1}\\
\int_1^\infty \left |\, \partial_\lambda \big (\lambda^{\frac{\nu}{2}} \, J_{\nu}(\sqrt{\lambda}\, r) \big )\right |^2\,  \rho^{-2s}(r)\, r\, dr  & \ \lesssim\  \lambda^{\nu-\frac 32+\eps}\  . \label{mix-2}
\end{align}
Moreover, the function $J_{-\nu}$ in the first bound can be replaced by $Y_\nu$ without changing the right hand side. 
\end{lemma}

\begin{proof}
By \eqref{der-2} it holds
\begin{equation} \label{eq-mixed-1}
\partial_\lambda \big (\lambda^{\frac{\nu}{2}} \, J_{-\nu}(\sqrt{\lambda}\, r) \big ) = -\lambda^{\frac{\nu-1}{2}}\, \frac r2\, J_{1-\nu}(\sqrt{\lambda}\, r).
\end{equation}
Now we use \eqref{jy^2} and \eqref{monoton} with $j=1$ to get
\begin{align}
\int_1^\infty  J^2_{1-\nu}(\sqrt{\lambda}\, r) \big) \,  \rho^{-2s}(r)\, r^3\, dr &\  \leq\  \int_1^\infty  (J^2_{\nu-1}(\sqrt{\lambda}\, r) +Y^2_{\nu-1}(\sqrt{\lambda}\, r) \big) \,   r^{-2\eps}\, dr  \label{mix-3}\\
& = \ \lambda^{\eps-\frac 12} \int_{\sqrt{\lambda}}^\infty t\, (J^2_{\nu-1}(t) +Y^2_{\nu-1}(t) \big) \,   t^{-1-2\eps}\, dt\nonumber \\
& \lesssim \ \lambda^{\eps}\ (J^2_{\nu-1}(\sqrt{\lambda}) +Y^2_{\nu-1}(\sqrt{\lambda}) \big) \, \lesssim \,  \lambda^{\eps+1-\nu}\, 4^\nu\, \Gamma^2(\nu-1), \nonumber
\end{align}
where we have applied \eqref{nu-pos} in the last step. This in combination with \eqref{eq-mixed-1} proves \eqref{mix-1}. When $J_{-\nu}$ is replaced by $Y_{\nu}$ in \eqref{mix-1}, then we obtain the same estimate
as can be seen from the first line of \eqref{mix-3}.  The second estimate of the Lemma is proven in a similar way; by \eqref{der-1} we have  
$$
\partial_\lambda \big (\lambda^{\frac{\nu}{2}} \, J_{\nu}(\sqrt{\lambda}\, r) \big ) = \lambda^{\frac{\nu-1}{2}}\, \frac r2\, J_{\nu-1}(\sqrt{\lambda}\, r).
$$
Hence
$$
\int_1^\infty \left |\, \partial_\lambda \big (\lambda^{\frac{\nu}{2}} \, J_{\nu}(\sqrt{\lambda}\, r) \big )\right |^2\,  \rho^{-2s}(r)\, r\, dr\  \lesssim \  \lambda^{\nu-1}\ \int_1^\infty J^2_{\nu}(\sqrt{\lambda}\, r)\, r^{-2\eps}\, dr
$$
and \eqref{mix-2} follows from Lemma \ref{lem-jj}.
\end{proof}

\medskip

%%%%%%%%%%%%%%%%%%%%%%%%%%%%%%%%%%%%%%%%%%%%%%%%%%%%

\appendix

\section{Confluent hypergeometric functions}
\label{sec-mu}

\noindent Recall first the definition  of the Kummer's hypergeometric series;
\begin{equation} \label{eq-M}
M(a,b,z) = \sum_{n=0}^\infty \, \frac{(a)_n}{(b)_n}\ \frac{z^n}{n!},
\end{equation}
where 
\begin{equation} \label{M-ab}
(a)_n = a(a+1)\cdots(a+n-1), \qquad (b)_n = b(b+1)\cdots(b+n-1), \qquad a_0=b_0=1.
\end{equation}
Noìte that by \eqref{eq-M} 
\begin{equation}  \label{mu}
 1\  \leq \  \Big | \, M\Big (\frac 12+m+|m|, 1+2|m|, z \Big) \, \big | \ \leq  \ e^{|z|} \qquad \forall\ z\in\R, \ \ \forall\ m\in\Z.
\end{equation}  
For the function $U$, for Re$\, a,\, z>0$ we will often use its integral representation
\begin{equation} \label{u-int}
\Gamma(a)\, U(a,b,z) = \int_0^\infty e^{-z t}\, t^{a-1}\, (1+t)^{b-a-1}\, dt,
\end{equation}
see \cite[13.2.5]{as}.  By \cite[Sect.13.4]{as} the functions $M$ and $U$ are related to their derivatives in the following way:
\begin{equation} \label{mu-der}
\frac{d}{dz}\, M(a,b,z) = \frac ab\, M(a+1,b+1,z), \qquad \frac{d}{dz}\, U(a,b,z) = -a\, U(a+1,b+1,z). 
\end{equation}

\smallskip

\begin{lemma} \label{lem-U}
Let $z>0$ and suppose that $a$ and $b>1$. Then
\begin{equation}  \label{u-upperb}
\Gamma(a)\, \left |\, U(a,b,z) \right| \ \leq 
\left\{
\begin{array}{l@{\quad}l}
 e^z\, z^{1-b}\,   \Gamma(b-1)   &\quad  a\geq 1,\\
& \\
 a^{-1}\, 2^{b-a-1} +2^{1-a}\ e^z\,  z^{1-b}\,  \Gamma(b-1) & \quad a<1  \\
\end{array}
\right.
\end{equation}
Moreover, for any $a>0, b>0$ we have
\begin{equation} \label{u'-upperb}
\Gamma(a)\, \big |\, \frac{d}{dz} \, U(a,b,z) \, \big | \ \leq \ \Gamma(b)\ e^z\, z^{-b}\, 
\end{equation}
\end{lemma}

\begin{proof}
For $a\geq 1$ equation \eqref{u-int} gives
$$
\Gamma(a)\, |U(a,b,z) | \leq \, \int_0^\infty e^{-zt}\, (t+1)^{b-2}\, dt \ \leq \ e^z \int_0^\infty e^{-zt}\, t^{b-2}\, dt =   \Gamma(b-1)\, e^z\, z^{1-b}.
$$
This implies the first bound in \eqref{u-upperb}.  Assume now that $a <1$. Then \eqref{u-int} implies
\begin{align*}
\Gamma(a)\, |U(a,b,z) | &  = \, \int_0^1  e^{-zt}\, t^{a-1}\, (t+1)^{b-a-1}\, dt + \int_1^\infty  e^{-zt}\, t^{a-1}\, (t+1)^{b-a-1}\, dt \\
& \leq \frac 1a\, 2^{b-a-1} + \int_1^\infty  e^{-zt}\, \Big(\frac{1+t}{t}\Big)^{1-a}\, (t+1)^{b-2}\, dt \\
& \leq \frac 1a\, 2^{b-a-1} +2^{1-a}\ e^z\,  z^{1-b}\,  \Gamma(b-1). 
\end{align*} 
The bound \eqref{u'-upperb} follows from \eqref{mu-der}, the first part of the proof and the identity $\Gamma(a+1)=a\, \Gamma(a)$.
\end{proof}

\begin{lemma} \label{lem-M}
Fix  $z>0$ and assume that $\lambda \leq \frac 34\, \alpha^2$. Then 
\begin{equation} \label{M-1}
\Big|\, M\Big(\frac 12+j+|m|+m\frac{\alpha}{\kappa}, 1+j+2|m|, \, z\Big)\, \Big| \, \leq\, e^{2 z} \qquad \forall\ m\in\Z, \quad j=0,1.
\end{equation}
Moreover, there exists $\lambda_c\leq \frac 34\, \alpha^2$ and $m_c \in\N$, independent of $\lambda_c$, such that for all $\lambda\in (0,\lambda_c)$ and all $m\in\Z$ with $|m| \geq m_c$ we have 
\begin{equation} \label{M-2}
M\Big(\frac 12+|m|+m\frac{\alpha}{\kappa}, 1+2|m|, \, z\Big)\,  \, \geq\,  \frac 12. 
\end{equation}
\end{lemma}

\begin{proof}
Let $j=0$. Since $|\frac 12+|m|+m\frac{\alpha}{\kappa} |< 2(1+2|m|)$ for all $m\in\Z$ by assumption,  with the notation introduced in \eqref{M-ab} we find that
$$
 \Big( \Big|\frac 12+|m|+m\frac{\alpha}{\kappa}\Big| \Big)_n \, \leq\, \, 2^n\, \big(1+2|m|\big)_n \, .  
$$
Since $|M(a,b,z)|\leq M(|a|, b,z)$ if $b>0$, this in combination with \eqref{eq-M} implies \eqref{M-1}. The proof for $j=1$ follows in the same way. 

\smallskip

\noindent Next we note that if $m\geq 0$ then \eqref{M-2} follows immediately from \eqref{eq-M}. We may thus assume that $m<0$. Note that 
\begin{equation} \label{m-lim}
\lim_{m\to-\infty} \frac{|\frac 12+|m|+m\frac{\alpha}{\kappa}|}{1+2|m|} = \frac{\alpha-\sqrt{\alpha^2-\lambda}}{2\sqrt{\alpha^2-\lambda}}\, = : q(\lambda)
\end{equation}
uniformly with respect to $\lambda\in(0,\frac 34 \alpha^2)$. Hence there exists $m_c\in\N$ such that for all $m\leq -m_c$ and all $\lambda\leq \frac 34 \alpha^2$ we have  
$$
\frac{|\frac 12+|m|+m\frac{\alpha}{\kappa}|}{1+2|m|} \leq 2\, q(\lambda).
$$
Since $m$ is negative and $2\, q(\lambda)\leq 1$, by assumption, in view of \eqref{M-ab} we conclude that 
$$
 \Big( \Big|\frac 12+|m|+m\frac{\alpha}{\kappa}\Big| \Big)_n \, \leq\, \,  2\, q(\lambda) \, \big(1+2|m|\big)_n\qquad \forall\ n\geq 1 \quad \forall\ m\leq -m_c.
$$
By \eqref{eq-M} this implies 
$$
M\Big(\frac 12+|m|+m\frac{\alpha}{\kappa}, 1+2|m|, \, z\Big)\,  \, \geq\, 1 - 2\, q(\lambda) \sum_{n=1}^\infty \frac{z^n}{n!}.
$$
To finish the proof it suffices to take $\lambda_c$ small enough, cf. \eqref{m-lim}, so that 
$$
2\, q(\lambda) \sum_{n=1}^\infty \frac{z^n}{n!} \leq \frac 12 \qquad \forall\ \lambda\leq \lambda_c.
$$
\end{proof}

\begin{lemma} \label{lem-der-m}
Let $b>0$ and $z>0$. Then 
\begin{equation}  \label{ma'}
\big | \frac{d}{da}\, M(a,b,z) \big | \ \lesssim \ \frac{M(|a|, b, 2z)}{1+|a|}\ .
\end{equation}
\end{lemma}

\begin{proof}
From the definition of $(a)_n$, see \eqref{M-ab}, we find that 
$$
 | \frac{d}{da}\, (a)_n  | \, \lesssim\, \frac{n\, (|a|)_n}{1+|a|} \qquad \forall\ n\in\N.
$$
Hence 
\begin{align*}
\big | \frac{d}{da}\, M(a,b,z) \big | & \ \lesssim \ \frac{1}{1+|a|}\, \sum_{n=1}^\infty\, \frac{(|a|)_n}{(b)_n}\, \frac{ n\, z^n}{n!} \ \leq \ \frac{1}{1+|a|} \, \sum_{n=1}^\infty\, \frac{(|a|)_n}{(b)_n}\, \frac{ (2z)^n}{n!} \\
& \ \leq \ \frac{1}{1+|a|} \ M(|a|, b, 2z).
\end{align*}
\end{proof}

\begin{lemma} \label{lem-der-u}
Let $a,b>0$ and $z>0$. Then 
\begin{equation}  \label{ua'}
\big | \frac{d}{da}\, \big( \Gamma(a)\, U(a,b,z)\big) \big | \ \lesssim \ 2^{b-a-1} + e^z\, z^{1-b}\, \Gamma(b-1).
\end{equation}
\end{lemma}

\begin{proof}
By \eqref{u-int} we have  
\begin{align*}
 \frac{d}{da} ( \Gamma(a)\, U(a,b,z))  & \ = \ \int_0^\infty e^{-zt}\, \log\left(\frac{t+1}{t}\right)\, t^{a-1}\, (1+t)^{b-a-1}\, dt \ \leq\  2^{b-a-1}\, \int_0^1\log\left(\frac{t+1}{t}\right)\, t^{a-1}\, dt \\
 &  \ + \int_1^\infty e^{-zt}\, t^{a-1}\, (1+t)^{b-a-1}\, dt \ \lesssim\ 2^{b-a-1} + e^z\int_1^\infty e^{-z (t+1)}\, (t+1)^{b-2}\, dt \\
 & \ \leq \ 2^{b-a-1} + e^z\, z^{1-b}\, \Gamma(b-1).
\end{align*}
\end{proof}

%%%%%%%%%%%%%%%%%%%%%%%%%%%%%%%%%%%%%%%%%%%%%%
\section{Bessel functions $J$ and $Y$}
\label{app-bessel}

\noindent The functions $J_\nu(z)$ and $Y_\nu(z)$ are related through the identity \cite[Eq.9.1.2]{as}
\begin{equation} \label{jy}
Y_\nu(z) = \frac{J_\nu(z) \, \cos(\nu \pi) -J_{-\nu}(z)}{\sin(\nu\pi)} \, . 
\end{equation}
The right hand side is to be replaced by a limiting value when $\nu\in\Z$. The above formula implies
\begin{equation} \label{jy^2}
J^2_{-\nu}(z) \, \leq \, J^2_{\nu}(z) + Y^2_\nu(z).
\end{equation}
By \cite[Eq.9.1.16]{as} their Wronskian is independent of $\nu$ and reads as follows:
\begin{equation} \label{w-jy}
J'_\nu(z)\, Y_\nu(z) -J_\nu(z)\, Y_\nu'(z) = J_{\nu+1}(z)\, Y_\nu(z)-Y_{\nu+1}(z)\, J_\nu(z) = \frac{2}{\pi\, z} .
\end{equation}
For $z\to 0$ we have 
\begin{equation} \label{nu-pos}
\begin{aligned}
J_\nu(z) & =\frac{(z/2)^{\nu}}{\Gamma(1+\nu)}(1+\mathcal{O}(z^2)),   & \nu \geq 0 \\
Y_\nu(z) & = -\frac{\Gamma(\nu)\,(z/2)^{-\nu}}{\pi}(1+\mathcal{O}(z^2)) +  \frac{\cot(\nu \pi)\, (z/2)^{\nu}}{\Gamma(1+\nu)}(1+\mathcal{O}(z^2)),  \ \   & \nu>0,
\end{aligned}
\end{equation}
where the error terms  are uniform with respect to $\nu$. When $\nu=0$,  then
\begin{equation} \label{nu-zero}
Y_0(z)= \frac{2}{\pi} (\log z +\gamma-\log 2) + o(z), \qquad z\to 0,
\end{equation}
where $\gamma\simeq 0,577$ is the Euler gamma constant, see \cite[Eqs.9.1.7-9]{as}. 
From \cite[pp.446]{wa} we learn that if $\nu > 1/2$, then
\begin{equation} \label{monoton}
\frac{d}{dz}\left [z^j\, ( Y_\nu^2(z) + J^2_\nu(z)) \right] \leq 0, \qquad \forall\ z >0, \quad j=0,1.
\end{equation}
We will also need a representations of $J_\nu(z)$  in terms of a power series in $z$:
\begin{equation} \label{j-series}
J_\nu(z) = (z/2)^\nu\, \sum_{k=0}^\infty \, \frac{(-1)^k\, z^{2k}}{4^k\, k!\ \Gamma(\nu+k+1)}, \qquad \nu\neq -1,-2,\dots
\end{equation}
and integral representation of $Y_\nu$ for $\nu > 0$: 
%\begin{equation} \label{int-j-1}
%J_\nu(z) = \frac{2 \, (\frac z2)^\nu}{\sqrt{\pi}\ \Gamma(\nu+\frac 12)}\, \int_0^1 (1-t^2)^{\nu-\frac 12}\, \cos(zt)\, dt, \
%\end{equation}
%and 
\begin{equation} \label{int-j-2}
Y_\nu (z) = \frac{2 \, (\frac z2)^\nu}{\sqrt{\pi}\ \Gamma(\nu+\frac 12)}\, \Big[\int_0^1 (1-t^2)^{\nu-\frac 12}\, \sin(zt)\, dt ) -\int_0^\infty e^{-zt}\, (1+t^2)^{\nu-\frac 12}\, dt\Big].
\end{equation}
We refer to  \cite[Eqs.9.1.10]{as} and \cite[Sect.6.1, Eq.(4)]{wa} respectively for the above formulas. 
%Often we will make use of the following equation for the integral of $J_\nu$; 
%\begin{align} \label{j^2}
%(4\nu^2-1) \int_x^\infty |\el_\nu(t)|^2\ \frac{dt}{t^2} & = \frac 4\pi -x\left[ \Big(\frac{\el_\nu(x)}{x}+\el_\nu'(x)\Big)^2+2\big(1-\frac{\nu^2}{x^2}\big) \el_\nu^2(x) +(\el_\nu')^2(x)\right],
%\end{align}
%where $\el_\nu(t) = a J_\nu(t) + b Y_\nu(t)$, see \cite[p.447]{wa}. 
Let us also recall the well-known relations between the Bessel functions and their derivatives: with the above notation we have
\begin{align} 
\frac{d}{dz}\, \el_\nu(z) & = \el_{\nu-1}(z) -\frac{\nu}{z} \ \el_\nu(z) \label{der-1} \\
\frac{d}{dz}\, \el_\nu(z) & = -\el_{\nu+1}(z) +\frac{\nu}{z} \ \el_\nu(z) \label{der-2} .
\end{align}

\noindent Next we recall several pointwise bounds on $J_\nu$ and $Y_\nu$. 
By \cite[Eq.9.1.60]{as}  
\begin{equation} \label{J-unif}
|J_\nu(z)| \, \leq \, 1 \qquad \quad  \forall\ \nu \geq 0, \quad \forall \ z >0.  
\end{equation}
For large values of $z$  and any fixed $\nu >0$ it holds
\begin{equation} \label{z-large}
Y_\nu^2(z) + J^2_\nu(z) \ \leq\, \frac{2}{\pi\, \sqrt{z^2-\nu^2}}, \qquad \forall\ z\geq \nu,
\end{equation}
see \cite[p.447, eq.(1)]{wa}. Finally we mention an integral identity due to \cite[p.403, eq.(2)]{wa}: for any $a>0$ we have
\begin{equation} \label{wa-int-1}
\int_0^\infty\, J^2_\nu(a\, t)\ t^{-\beta}\, dt \ =\  \frac{a^{\beta-1}\, \Gamma(\beta)\ \Gamma(\nu +\frac{1-\beta}{2})}{2^\beta\, \Gamma^2(\frac{1+\beta}{2})\ \Gamma(\nu +\frac{1+\beta}{2})} \, , \qquad 2\nu+1 > \beta>0.
\end{equation}

\smallskip

\begin{lemma} \label{lem-product}
Let $\nu \geq 1$. Then
\begin{equation} \label{product-upb}
\sup_{1\leq t \leq \nu+j}  \Big ( \, J^2_{\nu}(t) + |J_{\nu+j}(t)\ Y_\nu(t)| + |J_{\nu}(t)\ Y_{\nu+j}(t)| \, \Big) \ \lesssim\ \nu^{-2/3}, \qquad j=0,1,2.
\end{equation}
\end{lemma}

\begin{proof}
We use the substitution $t= z \nu$. Suppose first that $1\leq t \leq \nu$. In this case  $z\in [\nu^{-1},1]$. By \cite[Eq.9.3.6, Eq.9.3.38]{as} we  have 
\begin{align}
J_\nu(t) & = J_\nu(\nu z) = \left( \frac{4\, \xi^{2/3}}{1-z^2}\right)^{1/4}\, \left[ \frac{\ai\big(( \nu \xi)^{2/3}\big)}{\nu^{1/3}} + \frac{e^{-\nu\, \xi }}{1+|\nu\, \xi|^{1/6}}\ \mathcal{O}(\nu^{-4/3})\right] \label{jnuz} \\
Y_\nu(t) & = Y_\nu(\nu z) = -\left( \frac{4\, \xi^{2/3}}{1-z^2}\right)^{1/4}\, \left[ \frac{\bi\big(( \nu \xi)^{2/3}\big)}{\nu^{1/3}} + \frac{e^{\nu\, \xi }}{1+|\nu\, \xi |^{1/6}}\ \mathcal{O}(\nu^{-4/3})\right] \label{ynuz} ,
\end{align}
where
$$
\xi= \xi(z) = \frac 32\, \int_z^1\frac{\sqrt{1-t^2}}{t}\, dt, \qquad z < 1,
$$
and $\ai, \bi$  are Airy functions of the first and second kind. From their asymptotic behavior it follows that 
\begin{equation*} \label{airy-upb}
|\ai(x)| \ \lesssim\ 
\left\{
\begin{array}{l@{\quad}l}
x^{-\frac 14}\, \exp(-\frac 23 x^{3/2}) &\  1\leq x   \\
1  & \    0< x<1, 
\end{array}
\right.
\quad 
|\bi(x)| \ \lesssim\ 
\left\{
\begin{array}{l@{\quad}l}
x^{-\frac 14}\, \exp(\frac 23 x^{3/2})  &\  1\leq x   \\
1  & \    0< x<1, 
\end{array}
\right.
\end{equation*}
Since $\xi(z)^{2/3}(1-z^2)^{-1}$ remains bounded  as $z\to 1$, which follows from the definition of $\xi(z)$, the above bounds on the Airy functions together with\eqref{jnuz} and \eqref{ynuz} imply
$$
|J_{\nu}(t)| = | J_\nu(\nu z) | \ \lesssim\  \nu^{-1/3}\, e^{-\nu\, \xi(z)}, \qquad |  Y_\nu(t) | = |  Y_\nu(\nu z) | \ \lesssim\  \nu^{-1/3}\, e^{\nu\, \xi(z)}, \qquad \frac{1}{\nu} \leq z \leq 1. 
$$
This proves \eqref{product-upb} for $j=0$. Now if $\nu \leq t \leq \nu +j$, then we estimate $|J_{\nu+j}(t)|$ and $|Y_{\nu+j}(t)|$ in the same way as in the last equation with $\nu$ replaced by $\nu+j$. It remains to estimate $J_\nu(t)$ and $Y_\nu(t)$ on $[\nu, \nu+j]$. To this end we note that \eqref{jnuz} and \eqref{ynuz} still hold with $1\leq z\leq (\nu+j)/\nu$ and 
\begin{equation} \label{xi-neg}
\xi(z) = -\frac 32\, \int_1^z\frac{\sqrt{t^2-1}}{t}\, dt, \qquad z > 1,
\end{equation}
see \cite[Eq.9.3.39]{as}. Therefore, following the above procedure we find that 
$$
|J_{\nu}(t)| = | J_\nu(\nu z) | \ \lesssim\  \nu^{-1/3}\, e^{\nu\, |\xi(z)|}, \quad |Y_{\nu}(t)| = |Y_\nu(\nu z) | \ \lesssim\  \nu^{-1/3}\, e^{\nu\, |\xi(z)|},
\qquad 1\leq z \leq \frac{\nu+j}{\nu}.
$$
Since $\nu\, |\xi(z)|$ is uniformly bounded with respect to $\nu$ on $[1, \frac{\nu+j}{\nu}]$, see \eqref{xi-neg}, this completes the proof. 
\end{proof}

\begin{lemma} \label{lem-jy0}
Let $z>0$. Then 
\begin{equation} \label{eq-jy0}
|J_0(z)-1| \, \lesssim\, \min\{z^2 , \, 1\} \, \qquad \big |Y_0(z) -\frac 2\pi \big ( \log \frac z2+ \gamma\big) \big | \, \lesssim\, ( |\log z| +1)\, \min\{ z^2   , \, 1 \},
\end{equation}
and 
\begin{equation} \label{eq-jy0-der}
|J_0'(z)| \, \lesssim\, \min\{z , \, \sqrt{z}\} \, , \qquad \big |Y'_0(z) -\frac{2}{\pi z} \big | \, \lesssim\, \min\{z , \, \sqrt{z}\}
\end{equation}
\end{lemma}

\begin{proof}
The upper bounds in \eqref{eq-jy0} follow from the integral representaions 
$$
J_0(z) = \frac 1\pi \int_0^\pi \cos(z\cos t)\, dt, \qquad Y_0(z) = \frac{4}{\pi^2} \int_0^{\pi/2} \cos(z\cos t) (\gamma+\log(2z\sin^2t))\, dt,
$$ 
see \cite[Sec.9.1]{as}. To prove \eqref{eq-jy0-der} we use the fact that $J'_0(z) = -J_1(z)$ and $Y_0'(z) = -Y_1(z)$, see \eqref{der-2}.  With the help of \eqref{nu-pos} and \eqref{z-large} we then arrive at \eqref{eq-jy0-der}. 
\end{proof}

\smallskip

%%%%%%%%%%%%%%%%%%%%%%%%%%%%%%%%%%%%%%%%%%%%%%%
%\appendix

\section{Modified Bessel functions $I$ and $K$}
\label{app} 
 
\noindent The modified Bessel functions $I_\nu$ and $K_\nu$ satisfy the relation
 \begin{equation} \label{ik}
 K_\nu(z) = \frac{\pi}{2}\, \frac{I_{-\nu}(z)-I_\nu(z)}{\sin(\nu\, \pi)} ,
 \end{equation}
 see \cite[Eq.9.6.2]{as}, where as usual the right hand side is replaced by its limiting value if $\nu\in\Z$. From \cite[Eq.9.6.15]{as} we learn that the Wronskian of $I_\nu$ and $K_\nu$ reads as follows:
 \begin{equation} \label{w-ik}
 K'_\nu(z) \, I_\nu(z) - K_\nu(z) \, I'_\nu(z) =  K_{\nu+1}(z) \, I_\nu(z) + K_\nu(z) \, I_{\nu+1}(z) = \frac 1z.
 \end{equation}
As $z\to 0$, then 
\begin{equation}\label{ik-z0}
\begin{aligned}  
I_\nu(z) & =  \frac{(z/2)^{\nu}}{\Gamma(1+\nu)}(1+\mathcal{O}(z^2)), \qquad \nu \geq 0,\\
 K_\nu(z) & = \frac{\Gamma(\nu)\,(z/2)^{-\nu}}{2}\, (1+\mathcal{O}(z^2)) - \frac{\pi\, (z/2)^{\nu}}{2\sin(\nu\pi)\, \Gamma(\nu+1)}\,  (1+\mathcal{O}(z^2)), \qquad \nu>0, 
\end{aligned} 
\end{equation}
with the errors terms uniform in $\nu$, see \cite[Eq.9.6.10]{as} and \eqref{ik}. For $\nu=0$ we have by \cite[Eq.9.6.13]{as}
\begin{equation} \label{k-log}
K_0(z) = -\log z -\gamma+\log 2 + o(z) \qquad z\to 0.
\end{equation}
Finally, we recall the relations between the functions  $I_\nu$ and $K_\nu$ and their derivatives:
\begin{align}
I'_\nu(z) & = I_{\nu+1}(z) +\frac{\nu}{z}\, I_\nu(z) \label{i-der} \\
K'_\nu(z) & = -K_{\nu\pm 1} \pm\frac{\nu}{z}\, K_\nu(z) \label{k-der},
\end{align}
see  \cite[Eq.9.6.26]{as}. 

\begin{lemma} \label{lem-ik}
Let $\nu>0$. Then for every $z>0$ it holds
\begin{align} 
%I_\nu(z)\, K_\nu(z) & \leq \frac{1}{2\nu} \label{ik-upperb1} \\
I_{\nu+j}(z)\, K_\nu(z) & \,  \leq \,  \frac{1}{2\nu} \qquad j=0,1. \label{ik-upperb} 
\end{align}
\end{lemma}

\begin{proof}
From equation \eqref{w-ik} and the positivity of $I_\nu(z), K_\nu(z)$ (for $z>0$) it follows that
\begin{equation} \label{ik-upperb2} 
I_{\nu+1}(z)\, K_\nu(z)\, \leq\, \frac 1z, \qquad I_{\nu}(z)\, K_{\nu+1}(z) \, \leq \, \frac 1z.
\end{equation}
On the other hand, by \cite[Eq.9.6.26]{as} 
\begin{align} 
\frac{2\nu}{z}\, K_\nu(z) & = -K_{\nu-1}(z) + K_{\nu+1}(z) \label{rec-k} 
\end{align}
Hence 
\begin{align*}
\frac{2\nu}{z}\, I_\nu(z)\, K_\nu(z) & = K_{\nu+1}(z)\,  I_\nu(z) -K_{\nu-1}(z)\, I_\nu(z) \leq K_{\nu+1}(z)\,  I_\nu(z) \leq  \frac 1z,
\end{align*}
where we have used \eqref{ik-upperb2} and the positivity of $I_\nu(z)$ and $K_\nu(z)$. This proves \eqref{ik-upperb} for $j=0$. Similarly we obtain from \eqref{rec-k} the upper bound
$$
\frac{2\nu}{z}\, K_\nu(z)\, I_{\nu+1}(z) =  I_{\nu+1}(z)\,  K_{\nu+1}(z) - I_{\nu+1}(z)\,  K_{\nu-1}(z) \, \leq \, I_{\nu+1}(z)\,  K_{\nu+1}(z)
\leq \frac{1}{\nu+1},
$$
where we applied \eqref{ik-upperb} with $j=0$. It follows that  $ K_\nu(z)\, I_{\nu+1}(z) \leq \frac{z}{\nu (\nu+1)}$. This in combination with \eqref{ik-upperb2} proves \eqref{ik-upperb} for $j=1$. 
\end{proof}

%%%%%%%%%%%%%%%%%%%%%%%%%%%%%%%%%%%%%%%%%%%%%%%%%%%%%%

%\newpage

%%%%%%%%%%%%%%%%%%%%%%%%%%%
\section*{\bf Acknowledgements}
\noindent 
I would like to thank Georgi Raikov for many useful remarks. 
The support of MIUR-PRIN2010-11 grant for the project  ''Calcolo delle variation'' is gratefully acknowledged.

%%%%%%%%%%%%%%%%%%%%%%%%%%%%%%%%%%%%%%%%%%%%%%%%%%%%

\medskip

\bibliographystyle{amsalpha}

\end{document}